\documentclass[12pt,reqno]{amsart}

\usepackage{amsmath,amsfonts,amsthm,amssymb,amsxtra,enumitem,color,changepage}
\usepackage{bbm} 
\usepackage[hidelinks]{hyperref} 

\newtheorem{theorem}{Theorem}[section]
\newtheorem{proposition}[theorem]{Proposition}
\newtheorem{lemma}[theorem]{Lemma}
\newtheorem{corollary}[theorem]{Corollary}

\theoremstyle{definition}

\theoremstyle{remark}

\newtheorem{remark}[theorem]{Remark}

\DeclareRobustCommand{\SkipTocEntry}[5]{}

\newcommand{\1}{\mathbbm{1}}

\renewcommand{\epsilon}{\varepsilon}

\newcommand{\Haus}{\mathcal{H}}
\newcommand{\loc}{{\rm loc}}

\newcommand{\N}{\mathbb{N}}

\renewcommand{\phi}{\varphi}
\newcommand{\R}{\mathbb{R}}

\newcommand{\Z}{\mathbb{Z}}

\DeclareMathOperator{\dist}{dist}

\DeclareMathOperator{\dom}{dom}

\DeclareMathOperator{\supp}{supp}
\DeclareMathOperator{\sgn}{sgn}
\DeclareMathOperator{\Tr}{Tr}

\newcommand{\limplus}{{\mathchoice{\vcenter{\hbox{$\scriptstyle +$}}}
		{\vcenter{\hbox{$\scriptstyle +$}}}
		{\vcenter{\hbox{$\scriptscriptstyle +$}}}
		{\vcenter{\hbox{$\scriptscriptstyle +$}}}
}}
\newcommand{\limminus}{{\mathchoice{\vcenter{\hbox{$\scriptstyle -$}}}
		{\vcenter{\hbox{$\scriptstyle -$}}}
		{\vcenter{\hbox{$\scriptscriptstyle -$}}}
		{\vcenter{\hbox{$\scriptscriptstyle -$}}}
}}

\begin{document}

\title[Riesz means asymptotics on Lipschitz domains]{Riesz means asymptotics for Dirichlet and Neumann Laplacians on Lipschitz domains}

\author{Rupert L. Frank}
\address[Rupert L. Frank]{Mathe\-matisches Institut, Ludwig-Maximilians Universit\"at M\"unchen, The\-resienstr.~39, 80333 M\"unchen, Germany, and Munich Center for Quantum Science and Technology, Schel\-ling\-str.~4, 80799 M\"unchen, Germany, and Mathematics 253-37, Caltech, Pasa\-de\-na, CA 91125, USA}
\email{r.frank@lmu.de}

\author{Simon Larson}
\address{\textnormal{(Simon Larson)} Department of Mathematical Sciences, Chalmers University of Technology and the University of Gothenburg, SE-41296 Gothenburg, Sweden}
\email{larsons@chalmers.se}


\thanks{\copyright\, 2025 by the authors. This paper may be reproduced, in its entirety, for non-commercial purposes.\\
	Partial support through US National Science Foundation grant DMS-1954995 (R.L.F.), as well as through German Research Foundation grants EXC-2111-390814868 and TRR 352-Project-ID 470903074 (R.L.F.), the Knut and Alice Wallenberg foundation grant KAW 2017.0295 (S.L.), as well as the Swedish Research Council grant no.~2023-03985 (S.L.) is acknowledged.}

\begin{abstract}
    We consider the eigenvalues of the Dirichlet and Neumann Laplacians on a bounded domain with Lipschitz boundary and prove two-term asymptotics for their Riesz means of arbitrary positive order. Moreover, when the underlying domain is convex, we obtain universal, non-asymptotic bounds that correctly reproduce the two leading terms in the asymptotics and depend on the domain only through simple geometric characteristics. Important ingredients in our proof are non-asymptotic versions of various Tauberian theorems.
\end{abstract}

\maketitle

\setcounter{tocdepth}{1}

\tableofcontents

\phantomsection
\addcontentsline{toc}{part}{Introduction}

\section{Introduction and main results}

A famous theorem of Weyl from 1911 describes the asymptotic behaviour of eigenvalues of the Laplacian \cite{Weyl_11}. Weyl's motivation came, at least in part, from a question related to black-body radiation \cite{Lorentz,Sommerfeld}, but this result and its ramifications continue to be of fundamental importance in the spectral theory of differential operators and in mathematical physics. Various macroscopic theories in physics, including Thomas--Fermi theory in atomic physics \cite{LiebSimon_77} and Ginzburg--Landau theory in the physics of superconductivity \cite{FrHaSeSo}, have been rigorously understood on the basis of Weyl's law and the corresponding semiclassical analysis. Recent applications include the study of the area law for the entanglement entropy of a free Fermi gas and its logarithmic violation \cite{Sob,LeSoSp}.

In the present paper we revisit the original problem studied by Weyl, namely the asymptotic distribution of Laplace eigenvalues. Our focus lies on assuming rather minimal regularity of the boundary of the underlying domain and on keeping track of the geometric dependence of the error terms. One motivation for doing this comes from a certain spectral shape optimization problem, which we describe in more detail later in this introduction. Our methods are also in the spirit of semiclassics at low regularity, which is fundamental when using such techniques in the study of quantum many-body systems, for instance in the derivation of Thomas--Fermi and Ginzburg--Landau theory mentioned before.

Let us be more specific. For an open set $\Omega\subset\R^d$ we denote by $-\Delta_\Omega^{\rm D}$ and $-\Delta_\Omega^{\rm N}$ the Dirichlet and Neumann realizations of the Laplacian in $\Omega$, defined using the method of quadratic forms; see, e.g., \cite[Section 3.1]{LTbook}. We write $-\Delta_\Omega^\#$ when we make statements that refer to either of these operators. Assuming that $-\Delta_\Omega^\#$ has discrete spectrum, its eigenvalues in nondecreasing order and repeated according to multiplicities are denoted by $\lambda_n(-\Delta_\Omega^\#)$, $n\in\N=\{1, 2, 3, \ldots\}$. Then Weyl's law states that, as $n\to\infty$,
\begin{equation}
    \label{eq:weylintro}
    \lambda_n(-\Delta_\Omega^\#) = (L_{0,d}^{\rm sc})^{-\frac{2}d} |\Omega|^{-\frac{2}d} n^{\frac{2}d} + o(n^{\frac{2}d})
\end{equation}
with a certain constant $L_{0,d}^{\rm sc}$ depending only on $d$; see \eqref{eq:ltconst} below for its explicit expression. It is remarkable and one of the reasons behind the ubiquity of this law that the leading term in the asymptotics depends on $\Omega$ only through its measure and not through the details of its shape. In the Dirichlet case the asymptotics \eqref{eq:weylintro} are valid under the sole assumption that $\Omega\subset\R^d$ is open with finite measure, as shown by Rozenblum \cite{Rozenblum_Weyl_72}. In the Neumann case, a sufficient condition for the validity is that $\Omega\subset\R^d$ is a bounded open set with the extension property (whose definition we recall before Proposition \ref{prop:rieszptwbgn} below). For proofs and references, see, e.g., \cite[Sections 3.2 and 3.3]{LTbook}. 

As is well known, the proof of Weyl asymptotics and also the formulation of their finer properties become more natural in terms of the eigenvalue counting function, defined by 
$$
N(\lambda,-\Delta_\Omega^\#) := \#\{ n:\ \lambda_n(-\Delta_\Omega^\#) <\lambda \}
\qquad\text{for} \ \lambda\geq 0 \,.
$$
In terms of this function, the asymptotics \eqref{eq:weylintro} can be equivalently stated as
\begin{equation}
    \label{eq:weylintro2}
    N(\lambda,-\Delta_\Omega^\#) = L_{0,d}^{\rm sc} |\Omega| \lambda^{\frac d2} + o(\lambda^{\frac d2}) 
    \qquad\text{as}\ \lambda\to\infty \,.
\end{equation}

It is natural to ask whether the error term $o(\lambda^{\frac{d}2})$ can be improved. This was achieved by Seeley \cite{Seeley78,Seeley80}, who showed that for domains $\Omega$ with smooth boundary one has
\begin{equation}
    \label{eq:seeleyintro}
    N(\lambda,-\Delta_\Omega^{\rm D}) = L_{0,d}^{\rm sc} |\Omega| \lambda^{\frac d2} + O(\lambda^{\frac{d-1}2}) \,.
\end{equation}
This remains valid for Neumann boundary conditions. Ivrii \cite{Ivrii80} has shown that, if apart from the smoothness of the boundary a certain dynamical condition on $\Omega$ holds, then one has a two-term asymptotic expansion
\begin{equation}\label{eq:ivrii}
\begin{aligned}
    N(\lambda,-\Delta_\Omega^{\rm D}) = L_{0,d}^{\rm sc} |\Omega| \lambda^{\frac d2} - \frac14 L_{0,d-1}^{\rm sc} \mathcal H^{d-1}(\partial\Omega) \lambda^{\frac{d-1}2} + o(\lambda^{\frac{d-1}2}) \,, \\
    N(\lambda,-\Delta_\Omega^{\rm N}) = L_{0,d}^{\rm sc} |\Omega| \lambda^{\frac d2} + \frac14 L_{0,d-1}^{\rm sc} \mathcal H^{d-1}(\partial\Omega) \lambda^{\frac{d-1}2} + o(\lambda^{\frac{d-1}2}) \,.
\end{aligned}
\end{equation}
Here $\mathcal H^{d-1}(\partial\Omega)$ denotes the surface measure of the (smooth) boundary $\partial\Omega$ of $\Omega$ and we emphasize that, while the leading-order terms coincide, there is a difference in the signs of the second terms in the asymptotics. We do not recall the definition of Ivrii's dynamical condition, as we will not need it in what follows, but we note that it is believed to be satisfied for all open sets $\Omega$ with smooth boundary. This conjecture has been verified in a very limited number of cases, but remains open in general. For textbook presentations of \eqref{eq:seeleyintro} and \eqref{eq:ivrii} we refer to \cite[Section 17.5]{Hormander_bookIII} and \cite{SafarovVassiliev_book}. In \cite{BronsteinIvrii_03,Ivrii_03} Bronstein and Ivrii claim that, still imposing the dynamical condition, the regularity assumption on $\partial\Omega$ for \eqref{eq:ivrii} to hold can be reduced to a $C^1$ Dini condition.

Our results in this paper concern open sets $\Omega$ that have only a rather minimal amount of smoothness, namely we only assume that their boundary is Lipschitz. Among classical smoothness assumptions on the boundary this seems to be quite optimal to have the second term $\mathcal H^{d-1}(\partial\Omega)$ well defined. For a textbook definition of $\mathcal H^{d-1}$ as well as its properties, the reader is referred to \cite[Chapter 2]{EvansGariepy}. We also emphasize that Lipschitz boundaries can be quite rough and may exhibit fractal behaviour.

An informal statement of our first main result is that asymptotics \eqref{eq:ivrii} remain valid for Lipschitz domains, provided there is a tiny amount of averaging with respect to the spectral parameter $\lambda$. No dynamical condition is needed. The latter may have been known to experts in the area when the boundary is smooth, but we have not found a corresponding assertion in the literature.

By `averaging with respect to spectral parameter' we mean that we consider Riesz means of order $\gamma>0$, defined by
$$
\sum_n ( \lambda_n(-\Delta_\Omega^\#) - \lambda )_\limminus^\gamma =
\Tr(-\Delta_\Omega^\# -\lambda)_\limminus^\gamma = \gamma \int_0^\lambda (\lambda-\mu)^{\gamma-1} N(\mu,-\Delta_\Omega^\#)\,d\mu \,. 
$$
Here and in what follows, we use the notation $x_\pm = \frac{1}{2}(|x|\pm x)$. Note that $\mu\mapsto\gamma\lambda^{-1} (1-\mu/\lambda)^{\gamma-1}_\limplus$ is a probability density, which explains why Riesz means are certain averages (or means) of the eigenvalue counting function $N(\mu,-\Delta_\Omega^\#)$. We also note that as $\gamma$ decreases to zero, the probability density concentrates more and more on the right endpoint $\mu=\lambda$ and converges in the sense of measures to $\delta(\mu-\lambda)$. In this sense, the amount of averaging decreases as $\gamma$ tends to zero. By saying that our results will be valid `with a tiny amount of averaging' we mean that they are valid for Riesz means $\Tr(-\Delta_\Omega^\# -\lambda)_\limminus^\gamma$ of arbitrarily small order $\gamma>0$.

Riesz means of eigenvalues of the Dirichlet and Neumann Laplacian, and also of more general operators, are a classical object in spectral theory and have been studied, for instance, in \cite{HormanderRieszMeans}. In the context of Schr\"odinger operators they gained popularity in the context of Lieb--Thirring inequalities \cite{LiebThirring_76}.

It follows easily from \eqref{eq:weylintro2} that for any $\gamma>0$ we have
$$
\Tr(-\Delta_\Omega^\# -\lambda)_\limminus^\gamma
= L_{\gamma,d}^{\rm sc} |\Omega| \lambda^{\gamma+ \frac d2} + o(\lambda^{\gamma+ \frac d2})
\qquad\text{as}\ \lambda\to\infty
$$
with a certain constant $L_{\gamma,d}^{\rm sc}$ depending only on $\gamma$ and $d$; see \eqref{eq:ltconst} below for its explicit expression. Similarly, \eqref{eq:seeleyintro} and \eqref{eq:ivrii} imply corresponding variants for $\gamma>0$, under the same assumptions for which \eqref{eq:seeleyintro} and \eqref{eq:ivrii} are valid. In contrast, in this paper we will show that the Riesz means asymptotics corresponding to \eqref{eq:ivrii}, and consequently also the Riesz means analogue of \eqref{eq:seeleyintro}, are valid under much weaker assumptions than \eqref{eq:ivrii} and~\eqref{eq:seeleyintro}.

In what follows we restrict ourselves to dimensions $d\geq 2$, as in the one-dimensional case explicit formulas are available.

The following is our first main result.

\begin{theorem}\label{thm: Weyl asymptotics Lipschitz}
	Let $d\geq 2$, $\gamma>0$ and let $\Omega\subset\R^d$ be a bounded open set with Lipschitz boundary. Then, as $\lambda\to\infty$,
	\begin{equation}
		\label{eq:maind}
		\Tr(-\Delta_\Omega^{\rm D} -\lambda)_\limminus^\gamma = L_{\gamma,d}^{\rm sc} |\Omega| \lambda^{\gamma+\frac d2} - \frac14 L_{\gamma,d-1}^{\rm sc} \mathcal H^{d-1}(\partial\Omega) \lambda^{\gamma+\frac{d-1}{2}} + o(\lambda^{\gamma+\frac{d-1}{2}})
	\end{equation}
	and
	\begin{equation}
		\label{eq:mainn}
		\Tr(-\Delta_\Omega^{\rm N} -\lambda)_\limminus^\gamma = L_{\gamma,d}^{\rm sc} |\Omega| \lambda^{\gamma+\frac d2} + \frac14 L_{\gamma,d-1}^{\rm sc} \mathcal H^{d-1}(\partial\Omega) \lambda^{\gamma+\frac{d-1}{2}} + o(\lambda^{\gamma+\frac{d-1}{2}}) \,.
	\end{equation}
\end{theorem}

We emphasize again that the two crucial points of this theorem are that it is valid for arbitrarily small $\gamma>0$ and that it is valid assuming only Lipschitz regularity of the boundary.

Theorem \ref{thm: Weyl asymptotics Lipschitz} is a vast improvement over our main result in \cite{FrankLarson_Crelle20}, which concerned the Dirichlet case for $\gamma\geq 1$. As we will explain below, fundamentally new ingredients are needed to reach parameter values $\gamma<1$.

The error term $o(\lambda^{\gamma+\frac{d-1}{2}})$ in Theorem \ref{thm: Weyl asymptotics Lipschitz} cannot be improved within the class of Lipschitz domains. In the Dirichlet case this follows from our results in \cite{FrankLarson_JMP20}, which show that for any nonnegative function $R:(0,\infty)\to\R$ with $\lim_{\lambda\to\infty} R(\lambda)=0$ there is a bounded, open and connected set $\Omega\subset\R^d$ with Lipschitz boundary such that for all $\gamma\geq 0$ we have
$$
\limsup_{\lambda\to\infty} \frac{\Tr(-\Delta_\Omega^{\rm D} -\lambda)_\limminus^\gamma - L_{\gamma,d}^{\rm sc} |\Omega| \lambda^{\gamma+\frac d2} + \frac14 L_{\gamma,d-1}^{\rm sc} \mathcal H^{d-1}(\partial\Omega) \lambda^{\gamma+\frac{d-1}{2}}}{\lambda^{\gamma+\frac{d-1}2} \, R(\lambda)} = \infty \,.
$$

We now turn to our second main result, for which we restrict our attention to \emph{convex} sets $\Omega$. Since convex sets are Lipschitz, Theorem \ref{thm: Weyl asymptotics Lipschitz} is applicable to them and we have two-term asymptotics. The point of the following theorem is that it provides a uniform, non-asymptotic bound that depends on $\Omega$ only through the simple geometric characteristics $|\Omega|$, $\mathcal H^{d-1}(\partial\Omega)$ and $r_{\rm in}(\Omega)$, the latter being the inradius of $\Omega$, see \eqref{eq:inrad}. The bound is asymptotically sharp in the sense that in the limit $\lambda \gg r_{\rm in}(\Omega)^{-2}$ it recovers the first two terms in the asymptotics in Theorem \ref{thm: Weyl asymptotics Lipschitz}.

\begin{theorem}\label{thm: Weyl asymptotics convex}
	Let $d\geq 2$ and let $\Omega \subset \R^d$ be an open, bounded, and convex set. Then, for all $\lambda > 0$
	\begin{align*}
		\biggl|\Tr(-\Delta_\Omega^{\rm D}-&\lambda)_\limminus^\gamma - L_{\gamma, d}^{\rm sc} |\Omega| \lambda^{\gamma+\frac d2} + \frac{1}{4}L_{\gamma, d-1}^{\rm sc}\mathcal H^{d-1}(\partial\Omega)\lambda^{\gamma+\frac{d-1}2}\biggr| \hspace{135pt}\\
		&\leq C \mathcal H^{d-1}(\partial\Omega)\lambda^{\gamma+ \frac{d-1}{2}}\bigl(r_{\textup{in}}(\Omega)\sqrt{\lambda}\bigr)^{-\frac{\alpha}{11}}\,, 
    \end{align*}
    and
    \begin{align*}
		\biggl|\Tr(-\Delta_\Omega^{\rm N}-&\lambda)_\limminus^\gamma - L_{\gamma, d}^{\rm sc} |\Omega| \lambda^{\gamma+\frac d2} - \frac{1}{4}L_{\gamma, d-1}^{\rm sc}\mathcal H^{d-1}(\partial\Omega)\lambda^{\gamma+\frac{d-1}2}\biggr|\hspace{135pt} \\
        &\leq C \mathcal H^{d-1}(\partial\Omega)\lambda^{\gamma+ \frac{d-1}{2}}\Bigl[\bigl(1+\ln_\limplus\bigl(r_{\rm in}(\Omega)\sqrt{\lambda}\bigr)\bigr)^{-\alpha\max\{1, \gamma\}}+\bigl(r_{\textup{in}}(\Omega)\sqrt{\lambda}\bigr)^{1-d}\Bigr],
	\end{align*} 
	with
	\begin{equation*}
		\alpha = 1 \mbox{ for }\gamma \geq 1 \quad \mbox{and any} \quad \alpha \in (0, \gamma) \mbox{ for }0< \gamma< 1\,,
	\end{equation*}
	and where $C$ depends only on $\gamma, \alpha$, and the dimension.
\end{theorem}

The order of the error terms in Theorem \ref{thm: Weyl asymptotics convex} is probably not optimal, but it shows the correct dependence on the product $r_{\rm in}(\Omega)\sqrt\lambda$ and it will be sufficient for the applications that we have in mind, which we describe next.

Similarly like Theorem \ref{thm: Weyl asymptotics Lipschitz}, Theorem \ref{thm: Weyl asymptotics convex} had a predecessor in our work \cite{FrankLarson_Crelle20}, where we treated the Dirichlet case for $\gamma\geq 1$. Both to reach parameter values $\gamma<1$ and to treat the Neumann case requires significantly different arguments. Earlier results in a similar spirit include \cite{Melas_03,Weidl,KoVuWe,GeLaWe,KoWe,LarsonPAMS,HarrellStubbe_18,Harrell_etal_21}.


\subsection*{A spectral shape optimization problem}

Our motivation for Theorem \ref{thm: Weyl asymptotics Lipschitz} and, in particular, Theorem \ref{thm: Weyl asymptotics convex} comes from the following spectral shape optimization problems
\begin{equation}\label{eq:shapeopt}
\begin{aligned}
& \sup\left\{ \Tr(-\Delta_\Omega^{\rm D}-\lambda)_\limminus^\gamma :\ \Omega \in\mathcal C  \,,\ |\Omega| = 1 \right\}, \\
& \inf\left\{ \Tr(-\Delta_\Omega^{\rm N}-\lambda)_\limminus^\gamma :\ \Omega \in\mathcal C  \,,\ |\Omega| = 1 \right\},
\end{aligned}
\end{equation}
where $\mathcal C$ is a given class of open sets in $\R^d$. While these problems are interesting for a fixed value of $\lambda$, we focus on solving them in the regime of large $\lambda$. If the asymptotics in \eqref{eq:maind} and \eqref{eq:mainn} were completely uniform with respect to open sets $\Omega\in\mathcal C$, then asymptotically these optimization problems would reduce to the isoperimetric problem
$$
\inf\left\{ \mathcal H^{d-1}(\partial\Omega) :\ \Omega\in\mathcal C \,, |\Omega| = 1 \right\}.
$$
In particular, if $\mathcal C$ contains the ball of unit measure, the latter is a solution of this isoperimetric problem and we arrive at the conjecture that in the limit $\lambda\to \infty$ optimizing sets $\Omega_\lambda$ for the shape optimization problems \eqref{eq:shapeopt} should converge, in some sense, to a ball of unit measure. This problem was first investigated by one of us in \cite{LarsonJST}. In \cite{FrankLarson_Crelle20} we proved convergence to a ball in Hausdorff sense when $\mathcal C$ is the class of open convex sets and $\gamma\geq 1$.

In a companion paper \cite{FrankLarson_Convex2024} we use Theorem \ref{thm: Weyl asymptotics convex} to extend our shape optimization result in \cite{FrankLarson_Crelle20} to the Neumann case. Moreover, in both the Dirichlet and Neumann cases, we further extend these results to a range of $\gamma>\gamma_d$ for certain parameters $\gamma_d<1/2$ that we characterize (see \cite[Theorem~1.7]{FrankLarson_Convex2024}). This leads to interesting compactness questions. 

The problem with the heuristics that were used to arrive at the conjecture is that the two-term Weyl asymptotics are far from uniform for $\Omega\in\mathcal C$ (when $\mathcal C$ is large enough). While we have not stated it explicitly, it is clear from its proof that Theorem \ref{thm: Weyl asymptotics Lipschitz} is uniform within any subclass of Lipschitz domains for which certain (rather explicit) geometric characteristics are uniformly controlled. Therefore, given our Theorem \ref{thm: Weyl asymptotics Lipschitz}, the proof of the conjecture reduces to establishing sufficiently good control of these geometric quantities for optimizing sets $\Omega_\lambda$. This remains an open problem.

The situation is somewhat simpler, but still nontrivial, when $\mathcal C$ is the class of open convex sets. In this case Theorem \ref{thm: Weyl asymptotics convex} allows us to deal with the situation where the optimizing sets $\Omega_\lambda$ satisfy $\sqrt\lambda\, r_{\rm in}(\Omega_\lambda)\gg 1$. In \cite{FrankLarson_Convex2024} we show how to exclude the case $\liminf_{\lambda\to\infty} \sqrt\lambda\, r_{\rm in}(\Omega_\lambda)<\infty$.

As shape optimization in spectral theory is an active field of research that dates back at least to Rayleigh~\cite{Rayleigh_1877}, it is impossible to provide an overview of the topic that does it justice. The interested reader is instead referred to~\cite{Henrot_17} and references therein. Problems concerning the asymptotic behaviour of solutions of spectral shape optimization problems saw a rise in interest in recent years, largely motivated by a connection to the famous P\'olya conjecture highlighted in~\cite{ColboisElSoufi_14} (see also~\cite{Freitas_etal_21}); see, for instance, \cite{AntunesFreitas_13,BucurFreitas_13, vdBerg_15, vdBergBucurGittins_16, vdBergGittins_17, GittinsLarson_17, Freitas_17, LarsonAFM, Lagace20, BuosoFreitas_20} where a variety of problems of this type are studied.


\subsection*{Some elements of our proofs}

Let us conclude this introduction by briefly describing the methods used to prove our main results. We mostly focus on Theorem \ref{thm: Weyl asymptotics Lipschitz}. As a general rule, the arguments for Theorem \ref{thm: Weyl asymptotics convex} are similar, but significantly more involved, since the dependence on the geometry needs to be tracked more carefully.

As we already mentioned, Theorem \ref{thm: Weyl asymptotics Lipschitz} is a vast extension of the main result of \cite{FrankLarson_Crelle20}, which concerns the Dirichlet case for $\gamma\geq 1$. A property that lies at the basis of our work in \cite{FrankLarson_Crelle20} was a variational principle for $\Tr(-\Delta_\Omega^{\rm D} -\lambda)_\limminus^\gamma$ when $\gamma=1$. This variational principle was used both in the proof of the upper and the lower bound on this trace. Similar variational principle exist also for $\gamma>1$, but do not exist for $\gamma<1$. This is connected with the fact that $E\mapsto (E-\lambda)_\limminus^\gamma$ is convex if and only if $\gamma\geq 1$.

In this paper we mostly deal with the case $\gamma<1$, where no variational principle is available. We will compensate the lack of a variational principle by certain results from classical analysis, in the spirit of Tauberian theorems. A first, important observation is
\begin{quote}
    \emph{To prove asymptotics for $\gamma$ it suffices to prove asymptotics for some $\gamma_1>\gamma$ and an order-sharp bound for some $\gamma_0<\gamma$.}
\end{quote}
This is a rather straightforward consequence of a classical convexity theorem of Riesz \cite{Riesz1922} (see Proposition \ref{prop: Riesz logconvexity}), but seems to have been largely overlooked in the context of spectral asymptotics. An exception is H\"ormander's paper \cite{HormanderRieszMeans}, but there the focus is somewhat different and lies on (non-sharp) error bounds, rather than on asymptotics.

According to this principle, our work consists in 
\begin{adjustwidth}{25pt}{25pt}\vspace{5pt}
\begin{enumerate}
    \item[Task I:] Proving order-sharp bounds for arbitrarily small $\gamma>0$
    \item[Task II:] Proving asymptotics for some (possibly large) $\gamma$
\end{enumerate}\vspace{5pt}
\end{adjustwidth}
For Task II we can choose $\gamma=1$ and make use of the variational principle. Indeed, in the Dirichlet case the corresponding asymptotics are already known from our previous work \cite{FrankLarson_Crelle20}. The proof in the Neumann case requires some additional ideas, which we will describe below.

The methods that we will use to tackle Task I and the Neumann aspect of Task~II are based on \emph{Tauberian theorems}. Before describing our way of using them, let us provide some historical context. For background on these theorems we refer to the monograph~\cite{Korevaar_TauberianTheory_book}. Their use in connection with Weyl asymptotics goes back to Carleman \cite{Carleman34} with many further developments in the 1950s and 1960s. Tauberian theorems provide a robust tool to prove Weyl asymptotics, not only for the Laplacian, but also for higher order elliptic operators with variable coefficients, see, e.g., \cite{Garding54}. The Tauberian method consists of two steps:

\begin{adjustwidth}{25pt}{25pt}\vspace{5pt}
\begin{enumerate}
    \item[Step I:] Proving asymptotics for the trace of the heat semigroup or of (a power of) the resolvent
    \item[Step II:] Proving a Tauberian theorem for the Laplace or Stieltjes transform (when using the heat semigroup or resolvents, respectively)
\end{enumerate}\vspace{5pt}
\end{adjustwidth}

In most applications the corresponding Tauberian theorems are already known and the main part of the work goes into carrying out Step I. These asymptotics can be obtained by building local approximations for the solutions of the corresponding parabolic or elliptic equations.

While Tauberian theorems are a powerful method to derive the leading order term in Weyl asymptotics, it appears to be general wisdom that the Tauberian method (in the above form) performs poorly with respect to terms of subleading order. Even if a second term in the heat or resolvent trace asymptotics is known, this does not lead, in general, to a second term in the asymptotics for the number of eigenvalues. In particular, both in Task I and in the Neumann aspect of Task II, where we are interested in second terms (or at least in a bound of the same order as the expected second term), standard Tauberian theorems seem to be of little use.

There are so-called Tauberian remainder theorems (see \cite[Chapter VII]{Korevaar_TauberianTheory_book}), but their consequences for spectral asymptotics, at least when applied naively, are rather disappointing. Power-like remainders in asymptotics for the Laplace transform result in logarithmic remainders for the underlying measure, while in order to get power-like remainder for the underlying measure one would need exponentially small remainders for the Laplace transform. In our situation two-term heat trace asymptotics on Lipschitz domains are known from work of Brown \cite{Brown93}; see Theorem \ref{thm: Weyl asymptotics Lipschitz heat} below. The remainders there are not exponentially small and therefore it seems, at first sight, to be hopeless to apply Tauberian remainder theorems.

Here is how we get around the poor performance of Tauberian theorems when it comes to second terms:
\begin{enumerate}
    \item[(a)] In the bulk of $\Omega$ we turn our attention to \emph{pointwise asymptotics}. For these we do have an exponentially small remainder, which does allow us to get an order-sharp remainder term through Tauberian theorems.
    \item[(b)] When interested in the asymptotics in the Neumann case and assuming that the asymptotics in the Dirichlet case are already known, we consider the difference between the Dirichlet and Neumann Riesz means. In the difference the leading orders cancel, so the second term, which is what we are interested in, becomes the leading term. For Riesz means with $\gamma \geq 1$ the variational principle implies that the difference is \emph{monotone} and we are able to apply rather standard Tauberian methods.
\end{enumerate}
Items (a) and (b) will be used to accomplish Task I and the Neumann part of Task~II, respectively. 

The observation in (a) is rooted in work of Avakumovi\'{c} \cite{Avakumovic52}, but we track the geometric dependence more carefully and show that this allows one to get a sharp remainder term for the Riesz means as soon as $\gamma>0$. This is the content in Theorems \ref{thm: Sharp Weyld} and \ref{thm: Sharp Weyln}. In this connection we mention Avakumovi\'{c}'s order-sharp remainder bound for the Laplacian on manifolds without boundary \cite{Avakumovic56} and the adaption of his method to a larger class of operators \cite{Niemeyer65,Gromes70,Bruning74} and to the presence of singular potentials \cite{FrSa}.

The observation in (b) is elementary, but seems to be new. It is worth mentioning that the difference of the Dirichlet and Neumann Riesz means is not the only natural quantity that has the key properties necessary for carrying out the argument in (b) to accomplish Task II. An alternative construction, which is interesting in its own right, is discussed in Section \ref{sec:variations}.

In our discussion so far, we have focused on the Tauberian method based on the heat semigroup or resolvent powers. There is an alternative method based on the wave equation, consisting in
\begin{adjustwidth}{25pt}{25pt}\vspace{5pt}
\begin{enumerate}
    \item[Step I${}^{'}$:] Proving asymptotics for the wave propagator
    \item[Step II${}^{'}$:] Proving a Tauberian theorem for the Fourier transform
\end{enumerate}\vspace{5pt}
\end{adjustwidth}
This method goes back to Levitan \cite{Levitan54} and was perfected in \cite{Hormander_Acta68}. It is usually the method of choice to obtain more precise information, used for instance in the proofs of \eqref{eq:seeleyintro} and \eqref{eq:ivrii}. The drawback of this method, however, is that it needs, in general, very strong smoothness assumption on the boundary of the domain (and on the coefficients of the operator, if these are variable).

We also give an alternative proof of the order-sharp remainder for $\gamma>0$ in Theorems \ref{thm: Sharp Weyld} and \ref{thm: Sharp Weyln}, which bypasses Avakumovi\'{c}'s observation and is based on Levitan's wave equation method. However, heat equation methods enter also in this proof in order to handle the boundary region, where wave equation methods require too much regularity.

This concludes our discussion of Steps I and I${}^{'}$ of the Tauberian method.

Concerning Steps II and II${}^{'}$, the proof of Tauberian theorems, we note that some of the theorems that we need appear as announcements without proofs \cite{Avakumovic52,Ganelius56} or in journals that are not easily accessible \cite{Ganelius54}. Also, we need uniform versions of these results that track the dependence of the error term in terms of various parameters. We have not found the corresponding statements in the literature and a considerable fraction of our work is devoted to providing complete proofs. This constitutes Part 2 of this paper. 

This finishes our rough sketch of how we prove Theorem \ref{thm: Weyl asymptotics Lipschitz}. As we said, the proof of Theorem \ref{thm: Weyl asymptotics convex} follows a similar route, except that asymptotic statements are replaced by corresponding uniform, non-asymptotic statements. A crucial ingredient in the Tauberian Step I is a heat trace analogue of Theorem \ref{thm: Weyl asymptotics convex}. For the Dirichlet case such a result can be deduced from our results in \cite{FrankLarson_Crelle20} while the Neumann version is proved in \cite{FrankLarson_NeumannHeat2025}.

Finally, we point out that the proof strategy is rather general and can be adapted to different situations. In Section \ref{sec:variations} we illustrate this claim by sketching four further applications. Here we only highlight Theorems \ref{thm: three term asymptotics in polygons} and \ref{thm: three term asymptotics in polygons, Neumann}, which give \emph{three-term asymptotics} for $\gamma>1$ when $\Omega\subset\R^2$ has polygonal boundary. 


\subsection*{Structure of the paper}
This paper consists of two parts. So far in this introduction we have almost exclusively mentioned the result of the first part. The core of this first part consists in Sections \ref{sec:ptwweyl}, \ref{sec:intweyl}, \ref{sec:mainprooflarge} and \ref{sec:mainproofsmall}, whose content we have outlined above. In the additional section, Section \ref{sec:variations}, we briefly sketch some further applications for the methods that we develop.

The results in Part 2 are independent of those in Part 1.

Part 2 of this paper is devoted to various Tauberian-type theorems, which are needed in our proofs in Part 1. While some basic versions of these results are known, the references are often not easily accessible or do not contain the refined form of the results that we need. For this reason we have decided to include full details of all the results that we are using. In Section \ref{sec: Riesz convexity} we present a proof due to Riesz concerning a certain convexity property of Riesz means. In Section \ref{sec: Laplace Tauber} we present uniform versions of Tauberian theorems for Laplace transforms, which have their roots in work of Avakumovic, Freud, Garnelius, Korevaar and others. In Section \ref{sec: Fourier tauber} we present an extension to Riesz means of Tauberian theorems for the Fourier transform, as used by Levitan, H\"ormander, Safarov and others.

In Appendix \ref{app: Convex geometry} we summarize some more or less known results from convex geometry, which will play a role in the proof of Theorem \ref{thm: Weyl asymptotics convex}.


\part{Spectral asymptotics}


\section{Pointwise one-term Weyl law}\label{sec:ptwweyl}

\subsection{Main results of this section}

It is well known that the spectral projectors $\1(-\Delta_\Omega^\#<\lambda)$ are integral operators with integral kernels that are smooth on $\Omega\times\Omega$; see \cite{Garding_ScandCongr54,Garding54} or Lemma \ref{lem: spectral func} below. This implies the corresponding fact for Riesz means. In particular, we have the identity
\begin{equation*}
	(-\Delta_\Omega^\# -\lambda)_\limminus^\gamma(x,x) = \gamma \int_0^\lambda (\lambda-\mu)^{\gamma-1} \1(-\Delta_\Omega^\#<\mu)(x,x) \,d\mu \,.
\end{equation*}

This section is devoted to the proof of pointwise bounds on the on-diagonal Riesz means kernel. We will consider the cases $\gamma=0$ and $\gamma>0$ simultaneously and put
$$
(-\Delta_\Omega^\#-\lambda)_\limminus^0 := \1(-\Delta_\Omega^\# <\lambda) \,.
$$
The asymptotics will involve the constant
\begin{equation}
    \label{eq:ltconst}
    L_{\gamma,d}^{\rm sc} := \frac{\Gamma(\gamma+1)}{(4\pi)^{\frac{d}2} \ \Gamma(\gamma+\frac d2 +1)}
    \qquad\text{for}\ \gamma\geq 0 \,.
\end{equation}
Finally, our pointwise bounds will depend on the distance of the relevant point from the boundary, for which we use the abbreviation
\begin{equation}
    \label{eq:distance}
    d_\Omega(x) := \dist(x,\Omega^c) \,.
\end{equation}

The following theorem and three propositions are the main results of this section. We state and discuss them in this subsection and devote the remaining subsections in this section to their proofs. While it will not be important for us, we emphasize that the bounds in this section do not require the underlying set $\Omega$ to be bounded or to have finite measure.

The first result concerns the Riesz mean kernels in the bulk of $\Omega$, i.e.\ the region where $d_\Omega(x)\gtrsim 1/\sqrt{\lambda}$.
\begin{theorem}\label{thm: bulk}
    Let $d\geq 1$, $\gamma\geq 0$ and $\kappa >0$. Then there is a constant $C_{\gamma, d, \kappa}$ such that for any open set $\Omega\subset\R^d$, all $x\in\Omega$ and all $\lambda> 0$ with $d_\Omega(x)\geq \kappa/\sqrt{\lambda}$ we have
    $$
    \bigl| (-\Delta_\Omega^{\#}-\lambda)_\limminus^\gamma(x,x) - L_{\gamma,d}^{\rm sc} \lambda^{\gamma+\frac{d}2} \bigr| \leq C_{\gamma, d, \kappa} \lambda^{\frac{\gamma+d-1}2} d_\Omega(x)^{-1-\gamma} \,.
    $$
\end{theorem}

We find it remarkable that the constant in Theorem \ref{thm: bulk} depends only on $d, \gamma$ and $\kappa$ and not in any way on the geometry of $\Omega$ or the boundary conditions.

An important aspect of our bound in Theorem~\ref{thm: bulk} is that the decay of the relative remainder term $\lambda^{-\frac{1+\gamma}2} d_\Omega(x)^{-1-\gamma}$ with respect to $d_\Omega(x)$ increases with $\gamma$. This is particularly important in the next section when we integrate the pointwise bounds in order to get bounds on the Riesz means. For $\gamma>0$ this feature allows us to achieve an error bound of the correct order of magnitude.

We will give two proofs of Theorem \ref{thm: bulk}, one based on the heat equation and one on the wave equation. Those will be presented in the sections \ref{sec: Heat method} and \ref{sec: Wave method}, respectively. In both proofs we will apply Tauberian theorems, whose proofs can be found in Part~2 of this paper.

The next three propositions provide further bounds for the Riesz mean kernels. These bounds are less precise than that in Theorem~\ref{thm: bulk} but with the important feature that they remain valid also in the boundary region  $d_\Omega(x)\lesssim 1/\sqrt{\lambda}$. In this region both the geometry of $\Omega$ as well as the boundary conditions come into play. 

The first result concerns the Dirichlet case.
\begin{proposition}\label{prop:rieszptwbgd}
	Let $d\geq 1$ and $\gamma\geq 0$. There is a constant $C_{\gamma,d}$ such that for all open sets $\Omega\subset\R^d$, all $\lambda>0$ and all $x\in\Omega$,
	$$
	0\leq (-\Delta_\Omega^{\rm D} - \lambda)_\limminus^\gamma(x,x) \leq C_{\gamma,d} \lambda^{\gamma+\frac{d}2}\,.
	$$
\end{proposition}

In contrast to the Dirichlet case, in the Neumann case it is necessary to impose some further assumption on $\Omega$ for the validity of the corresponding bounds. Here we shall prove a bound that is valid in sets with the extension property. Recall that an open set $\Omega \subset \R^d$ is said to have the \emph{extension property} if there is a bounded, linear operator $\mathcal{E}\colon H^1(\Omega)\to H^1(\R^d)$ such that if $u\in H^1(\Omega)$ then $\mathcal{E}u(x) = u(x)$ for almost every $x\in \Omega$.

\begin{proposition}\label{prop:rieszptwbgn}
    Let $d\geq 1$ and $\gamma\geq 0$. For every open set $\Omega \subset \R^d$ with the extension property and every $\lambda_0>0$ there is a constant $C_{\gamma,\Omega,\lambda_0}$ such that for all $\lambda>0$ and all $x\in\Omega$ we have
    	$$
    	0\leq (-\Delta_\Omega^{\rm N} - \lambda)_\limminus^\gamma(x,x) \leq C_{\gamma, \Omega, \lambda_0} \lambda^{\gamma+\frac{d}2} \max\bigl\{1, (\lambda_0/\lambda)^{\frac{d}2}\bigr\} \,.
        $$   
\end{proposition}

In our applications to shape optimization problems it is important to have bounds that depend in a simple and uniform manner on the set $\Omega$. In the above results we emphasize that this is the case in both Theorem \ref{thm: bulk} and Proposition \ref{prop:rieszptwbgd}, as the constants in these results are independent of the set $\Omega$. In contrast, the constant that appears in Proposition~\ref{prop:rieszptwbgn} depends on the geometry in a more complicated manner. In fact, in the Neumann case any bound for the kernels in the region $d_\Omega(x)\lesssim 1/\sqrt{\lambda}$ needs to take into account more delicate properties of the boundary as otherwise such bounds would contradict a bound of Li--Yau \cite{LiYau86} for the corresponding heat kernel; see also Lemma~\ref{lem: off-diagonal kernel bound convex} below. However, that $\Omega$ has the extension property is a rather weak assumption on the boundary. Under stronger assumptions it is possible to prove more explicit bounds for the Neumann kernels also in the boundary region, $d_\Omega(x) \lesssim 1/\sqrt{\lambda}$. Our next result provides such a bound under the assumption that $\Omega$ is convex.
\begin{proposition}\label{prop:rieszptwbgnc}
	 Let $d\geq 2$ and $\gamma\geq 0$. Then there is a constant $C_{\gamma,d}$ such that for any convex open set $\Omega\subset\R^d$, all $x\in\Omega$, all $\lambda> 0$ and all $t>0$
    \begin{align*}
 		0\leq (-\Delta_\Omega^{\rm N}-\lambda)_\limminus^\gamma(x, x) \leq C_{\gamma,d} \, \frac{t^{-\gamma} e^{t\lambda}}{|\Omega \cap B_{\sqrt t}(x)|} \,.
 	\end{align*}
\end{proposition}

In the remainder of this section we first make some historical comments and then we provide proofs of Theorem~\ref{thm: bulk} and Propositions \ref{prop:rieszptwbgd}, \ref{prop:rieszptwbgn}, and \ref{prop:rieszptwbgnc}. 

\subsubsection*{Historical remarks}
The dependence on the distance from the boundary of the remainder in the bound of Theorem \ref{thm: bulk} is crucial for the proof of our main results. We have not found it in this form stated in the literature, but there are some precursors that we would like to mention. Avakumovi\'c~\cite{Avakumovic52} for $\gamma=0$ and Levitan~\cite{Levitan54} and G\aa rding~\cite{Garding54} for $\gamma\geq 0$ have shown that for $x$ in compact subsets of $\Omega$ one has
\begin{equation}
	\label{eq:interiorbound}
  	|(-\Delta_\Omega^\# - \lambda)_\limminus^\gamma(x,x) - L_{\gamma,d}^{\rm sc}\, \lambda^{\gamma+\frac{d}2} | \lesssim \lambda^{\frac{\gamma+d-1}2}
\end{equation}
for $\lambda$ large enough with an implicit constant depending on the compact subset. An alternative proof is given by H\"ormander \cite[Corollary 5.4]{HormanderRieszMeans}. Br\"uning \cite{Bruning74} attributes bounds showing how \eqref{eq:interiorbound} degenerates close to the boundary for $\gamma=0$ to unpublished lectures of Avakumovi\'c (1965). Since we do not have access to notes of these lectures, we cannot say which boundary conditions are imposed and which regularity of the boundary is assumed. We note that in several papers \cite{Niemeyer65,Gromes70,Bruning74}, that seem to be influenced by the work of Avakumovi\'c, similar bounds are shown in the case $\gamma=0$ for more complicated operators than $-\Delta$, but under stronger assumptions on the boundary than ours and only asymptotically and not uniformly as ours. Independently, Agmon~\cite{Agmon68} proved results about the degeneration of the interior bound \eqref{eq:interiorbound} for $\gamma=0$ as $x$ approaches the boundary, but his bounds are off by an arbitrarily small exponent $\epsilon>0$. These bounds are not good enough for our purposes.
        
All the bounds mentioned so far, with the exception of that of G\aa rding~\cite{Garding54}, concern the case $\gamma=0$. Safarov \cite{Safarov01} has essentially proved the special case $\gamma=1$ of Theorem~\ref{thm: bulk}; see Remark \ref{safarovrem} for what we believe to be a (fixable) gap in this proof. Apart from this, we are not aware of previous works that concern the boundary degeneration of the interior bounds \eqref{eq:interiorbound} for $\gamma>0$, although such bounds might be known to experts in the field.

The thrust of Propositions \ref{prop:rieszptwbgd}, \ref{prop:rieszptwbgn}, and \ref{prop:rieszptwbgnc} is that they are valid up to the boundary. Bounds of this type are known to experts in the field and follow rather directly, for instance, from corresponding heat kernel bounds.


\subsection{Heat kernels and a priori bounds}

The purpose of this subsection is twofold. On the one hand, we will introduce the heat kernels and state some basic bounds for them. On the other hand, we will use them to derive the a priori bounds on the spectral function and its Riesz means in Propositions \ref{prop:rieszptwbgd}, \ref{prop:rieszptwbgn} and \ref{prop:rieszptwbgnc}. 

We recommend Davies's book \cite{Davies_heatkernels} for background on heat kernels. It is known that for any $t>0$ the operators $e^{t\Delta_\Omega^\#}$ are integral operators; see \cite[Theorems 2.3.6 and 2.4.4]{Davies_heatkernels}. Their integral kernels are denoted by $k_\Omega^\#(t,\cdot,\cdot)$, that is,
$$
(e^{t\Delta_\Omega^\#}f)(x) = \int_\Omega k_\Omega^\#(t,x,x') f(x')\,dx'
\qquad\text{for all}\ f\in L^2(\Omega) \,.
$$

We recall the following leading order bounds for the heat kernels. They are classical and can be found, for instance, in~\cite{Davies_heatkernels}.

\begin{lemma}[{\cite[Eq.~(1.9.1)]{Davies_heatkernels}}]\label{heatdapriori}
    Let $d\geq 1$ and let $\Omega\subset\R^d$ be open. Then, for all $x,x'\in\R^d$ and $t>0$
    $$
    0\leq k_\Omega^{\rm D}(t,x,x') \leq (4\pi t)^{-\frac{d}2} \,.
    $$
\end{lemma}

\begin{lemma}[{\cite[Theorem 2.4.4]{Davies_heatkernels}}]\label{heatnapriori}
    Let $d\geq 1$ and let $\Omega\subset\R^d$ be an open set with the extension property. For any $t_0>0$ there is a constant $C_{\Omega,t_0}$ such that, for all $x,x'\in\R^d$ and $t>0$
    $$
    0\leq k_\Omega^{\rm N}(t,x,x') \leq C_{\Omega,t_0} t^{-\frac{d}2}\max\bigl\{1,(t_0/t)^{-\frac{d}2}\bigr\} \,.
    $$
\end{lemma}

\begin{remark} Two minor technicalities:
    \begin{enumerate}
    \item First, the bounds in both equation (1.9.1) and Theorem 2.4.4 of~\cite{Davies_heatkernels} are stated for connected sets $\Omega\subset \R^d$. 
    The claims in Lemmas~\ref{heatdapriori} and~\ref{heatnapriori} follow by applying the bounds of Davies to each connected component and recalling that a set with the extension property has a finite number of connected components (see, e.g., \cite[Lemma 2.94]{LTbook}). 
    \item Secondly, Theorem 2.4.4 of~\cite{Davies_heatkernels} is stated for heat kernels associated to general uniformly elliptic operators and with $0<t<1$. The bound claimed above follows by applying this general result with the operator $-t_0\Delta_\Omega^{\rm N}$ and then extend the obtained bound to all $t$ by using the monotonicity argument in the proof of~\cite[Theorem 3.2.9]{Davies_heatkernels}.
    \end{enumerate}
\end{remark}

In the case of the Neumann Laplacian on a \emph{convex} set, we shall also use the following bound due to Li and Yau \cite{LiYau86}; see also \cite[Theorem 5.5.6]{Davies_heatkernels}.

\begin{lemma}\label{lem: off-diagonal kernel bound convex}
	Fix $\delta >0$ and let $\Omega\subset \R^d$ be an open convex set. For all $x, x'\in\Omega, t>0$ it holds that
\begin{equation*}
	\frac{c_{d,\delta}e^{-\frac{|x-x'|^2}{(1-\delta)4t}}}{|\Omega \cap B_{\sqrt{t}}(x)|^{\frac{1}2}|\Omega \cap B_{\sqrt{t}}(x')|^{\frac{1}2}}  
	\leq 
	k_\Omega^{\rm N}(t, x, x')
	\leq 
	\frac{C_{d,\delta}e^{-\frac{|x-x'|^2}{(1+\delta)4t}}}{|\Omega \cap B_{\sqrt{t}}(x)|^{\frac{1}2}|\Omega \cap B_{\sqrt{t}}(x')|^{\frac{1}2}}\,,
\end{equation*}
where the constants $c_{d,\delta}, C_{d, \delta}$ depend only on $\delta, d$.
\end{lemma}

As a simple application of these heat kernel bounds we now deduce the pointwise bounds in Propositions~\ref{prop:rieszptwbgd}, \ref{prop:rieszptwbgn}, and \ref{prop:rieszptwbgnc}.

\begin{proof}[Proof of Propositions~\ref{prop:rieszptwbgd}, \ref{prop:rieszptwbgn}, and \ref{prop:rieszptwbgnc}]
    For simplicity we only write details for $\gamma>0$, the proof for $\gamma=0$ follows the same steps. 
    
    We shall use the elementary bound
    $$
    \alpha^\gamma \leq \left( \frac{\gamma}{e} \right)^\gamma e^\alpha 
    \qquad\text{for all}\ \alpha\geq 0\,.
    $$
    This implies that
    $$
    (E-\lambda)_\limminus^\gamma \leq \left( \frac{\gamma}{et} \right)^\gamma e^{t\lambda} \, e^{-tE} \,.
    $$    
    By the spectral theorem we deduce that
    $$
    (-\Delta_\Omega^\# - \lambda)_\limminus^\gamma \leq \left( \frac{\gamma}{et} \right)^\gamma e^{t\lambda} \, e^{t\Delta_\Omega^\#} \,,
    $$
    which, in turn, implies that
    $$
    (-\Delta_\Omega^\# - \lambda)_\limminus^\gamma(x,x) \leq \left( \frac{\gamma}{et} \right)^\gamma e^{t\lambda} \, k_\Omega^\#(t,x,x) \,.
    $$

    In the Dirichlet case we can now insert the upper bound $k_\Omega^\#(t,x,x)\leq (4\pi t)^{-\frac{d}2}$ from Lemma \ref{heatdapriori}. After optimization with respect to $t$, we obtain
    \begin{align*}
    (-\Delta_\Omega^\# - \lambda)_\limminus^\gamma(x,x) 
    & \leq (4\pi)^{-\frac{d}2} \left( \frac{\gamma}{e} \right)^\gamma \lambda^{\gamma+\frac{d}2} \sup_{\mu>0} \mu^{-\gamma-\frac{d}2} e^\mu \\ 
    & = (4\pi)^{-\frac{d}2} \left( \frac{\gamma}{e} \right)^\gamma \left( \frac{e}{\gamma+d/2} \right)^{\gamma+\frac{d}2} \lambda^{\frac{d}2} \,,        
    \end{align*}
    corresponding to the choice $\mu=t\lambda=\gamma+d/2$.

     In the Neumann case we can argue similarly but based on Lemma~\ref{heatnapriori} with $t_0 = \lambda_0^{-1}$. Instead of optimizing in the choice of $t$ we simply choose $t = \lambda^{-1}$, which proves that bound.

    Finally, in the Neumann case on convex sets, we argue similarly, but use Lemma~\ref{lem: off-diagonal kernel bound convex} instead of Lemma~\ref{heatnapriori}. In this case we leave $t$ as a free parameter. This completes the proof of all three propositions.    
\end{proof}


\subsection{A first proof of Theorem~\ref{thm: bulk} using the heat equation}
\label{sec: Heat method}

In this subsection we give a heat kernel proof of Theorem \ref{thm: bulk}. When it comes to the Neumann case our heat kernel based methods yield a weaker result; this is commented on in the proof below. The proof we give is based on estimating the difference $k_\Omega^\#(t,x,x)-(4\pi t)^{-\frac{d}2}$. According to the (heuristic) principle of not-feeling the boundary, this difference should be small (with respect to $(4\pi t)^{-\frac{d}2}$) when $x$ is far away (with respect to $\sqrt{t}$) from the boundary. The following two known bounds make this principle quantitative and show that the difference is actually exponentially small.

We begin with the case of the Dirichlet heat kernel. The following bound is due to Minakshisundaram \cite{Minakshisundaram49}; see also the proof due to Weyl in the appendix of that paper.  

\begin{lemma}\label{heatdlower}
	Let $d\geq 1$ and let $\Omega \subset \R^d$ be open. For all $x, x'\in\Omega, t>0$ it holds that
	$$
	 0 
	\leq
	(4\pi t)^{-\frac{d}2} e^{-\frac{|x-x'|^2}{4t}}-k_\Omega^{\rm D}(t, x, x')
	\leq e^{\frac{d}2}(4\pi t)^{-\frac{d}2} e^{-\frac{1}{4t}\max\{|x-x'|^2,\, d_\Omega(x)^2,\, d_\Omega(x')^2\}}\,.
	$$
	In particular, for all $x\in\Omega, t>0$,
	$$
	0 \leq (4\pi t)^{-\frac{d}2} - k_\Omega^{\rm D}(t, x, x) \leq
	e^{\frac{d}{2}}(4\pi t)^{-\frac{d}2} e^{-\frac{d_\Omega(x)^2}{4t}} \,.
	$$
\end{lemma}
We are indebted to Nikolay Filonov for pointing out an inaccuracy in Lemma~\ref{heatdlower} in an earlier version of this paper concerning the bound for large values of $t$ which is not covered by the argument in \cite{Minakshisundaram49}. The proof that follows shows how one can extend the result from \cite{Minakshisundaram49} to all $t>0$.
\begin{proof}
    The lower bound follows from Lemma~\ref{heatdapriori}. For $t \leq \frac{d_\Omega(x)^2}{2d} =: t_0$ the upper bound is proved in \cite{Minakshisundaram49} (without the factor $e^{\frac{d}2}$). For $t>t_0$ we observe that, by Lemma~\ref{heatdapriori} and the monotonicity of $t \mapsto e^{-\frac{d_\Omega(x)^2}{4t}}$,
    \begin{align*}
        (4\pi t)^{-\frac{d}2} e^{-\frac{|x-x'|^2}{4t}} - k_\Omega^D(t,x,x')\leq (4\pi t)^{-\frac{d}2}
        = (4 \pi t)^{-\frac{d}2} e^{\frac{d}2} e^{- \frac{d_\Omega(x)^2}{4t_0}} 
        \leq e^{\frac{d}2}(4 \pi t)^{-\frac{d}2}  e^{- \frac{d_\Omega(x)^2}{4t}}\,,
    \end{align*}
    this completes the proof.
\end{proof}

We now turn our attention to the Neumann heat kernel. The following bound on Lipschitz sets is essentially due to Brown \cite[Eq. (2.4+)]{Brown93}. A complete proof of the result as stated here can be found in \cite{FrankLarson_NeumannHeat2025}.

\begin{lemma}[{\cite[Eq. (2.4+)]{Brown93}}]\label{heatnlower}
    Let $d\geq 1$ and let $\Omega\subset\R^d$ be a bounded open set with Lipschitz boundary. For any $\eta>0$ there is a constant $C_{\Omega,\eta}$ such that for all $t>0$ and $x\in \Omega$ with $d_\Omega(x)\geq \eta \sqrt{t}$ we have
    $$
    \Bigl|k_\Omega^{\rm N}(t, x, x)- (4\pi t)^{-\frac{d}2}\Bigr| \leq C_{\Omega,\eta} t^{-\frac{d}2}e^{-\frac{d_\Omega(x)^2}{8 t}}\,.
    $$
\end{lemma}

For the Neumann Laplace operator in a \emph{convex} set we shall utilize the following result proved in \cite{FrankLarson_NeumannHeat2025}. The upshot of this result compared to Lemma \ref{heatnlower} is the uniformity of the constant with respect to $\Omega$.

\begin{lemma}\label{lemm: diagonal kernel bound bulk convex}
Let $d\geq 2$ and $\eta>0$. There is a constant $C_{d,\eta}>0$ such that for any open convex set $\Omega \subset \R^d$, all $t>0$ and $x\in\Omega$ with $d_\Omega(x)\geq \eta \sqrt{t}$ we have
	\begin{equation*}
		\Bigl|k_\Omega^{\rm N}(t, x, x)- (4\pi t)^{-\frac{d}2}\Bigr| \leq C_{d,\eta} t^{-\frac{d}2}e^{-\frac{d_\Omega(x)^2}{8t}}\,.
	\end{equation*}
\end{lemma}

The $8$ in the exponential rate of decay in Lemmas~\ref{heatnlower} and \ref{lemm: diagonal kernel bound bulk convex} is not sharp, the optimal value is likely $4$ (as in Lemma~\ref{heatdlower}). The proofs in \cite{Brown93} and \cite{FrankLarson_NeumannHeat2025} can be modified to yield a corresponding result with $8$ replaced by any number greater than $4$ at the price of increasing the multiplicative constants in the bounds. For our purposes the sharp value is not important, only that the exponential rate in the bound is independent of $\Omega$.

We combine the preceding three lemmas with Tauberian remainder theorems to derive the pointwise bounds on the spectral function and its Riesz means. The argument will completely prove Theorem~\ref{thm: bulk} for $\#= {\rm D}$ but, as mentioned above, only a weaker version when $\# = {\rm N}$. A full proof of Theorem~\ref{thm: bulk} will be given in Section~\ref{sec: Wave method}.

\begin{proof}[Proof of Theorem \ref{thm: bulk} (in a weaker form for $\# = {\rm N}$)]
    The proof is based on a non-asymptotic version of a Tauberian theorem for the Laplace transform that essentially goes back to work of Ganelius~\cite{Ganelius54}. The version of Ganelius' result that we shall use will be proved in Section~\ref{sec: Laplace Tauber}, specifically we shall utilize Corollary \ref{cor: Tauberian cor}.
    
    Fix $x \in \Omega$ and set
    \begin{equation*}
        f_x(\lambda) := \1(-\Delta_\Omega^\#<\lambda)(x, x)- L_{0, d}^{\rm sc}\lambda_\limplus^{\frac{d}2}\,.
    \end{equation*}
    As a difference of two monotone functions $f_x$ has locally bounded variation and as such defines a Borel measure $df_x$ on $\R$. Note that the Borel measure on $\R$ defined by
    \begin{equation*}
        \tilde\mu(\omega) := \int_\omega df_x(\lambda) + \frac{d}{2}L_{0,d}^{\rm sc}\int_{\omega \cap (0, \infty)}\lambda^{\frac{d}2-1}\,d\lambda
    \end{equation*}
    is nonnegative.
    
    One computes that for all $\gamma\geq 0, \lambda>0$
    \begin{equation*}
        (-\Delta_\Omega^\#-\lambda)_\limminus^\gamma(x, x) - L_{\gamma, d}^{\rm sc}\lambda^{\gamma+\frac{d}2} = \int_0^\lambda (\lambda-\Lambda)^\gamma\, df_x(\Lambda)\,,
    \end{equation*}
    and for all $t>0$
    \begin{equation*}
        k_\Omega^\#(t, x, x) - (4\pi t)^{-\frac{d}2} = \int_0^\infty e^{-t\lambda}\, df_x(\lambda)\,.
    \end{equation*}

    For $\#={\rm D}$, Lemma~\ref{heatdlower} and the fact that $x^{\alpha}e^{-\beta x} \leq \bigl(\frac{\alpha}{\beta}\bigr)^\alpha e^{-\alpha}$ for all $x\geq 0, \alpha, \beta >0$ imply that\footnote{Here and in what follows we sometimes use the notation $\lesssim$ to mean that the left side is bounded by a constant times the right side, where the implied constant is independent of the relevant parameters. Sometimes, as here, we put a subscript under the symbol $\lesssim$, emphasizing that the implied constant only depends on the parameters appearing in this subscript.}
    \begin{equation*}
        \Biggl|\int_0^\infty e^{-t\lambda}\, df_x(\lambda)\Biggr| \lesssim_d d_\Omega(x)^{-d}e^{-\frac{d_\Omega(x)^2}{8t}} \quad \mbox{for all } t>0\,.
    \end{equation*}
    By Corollary~\ref{cor: Tauberian cor} with $\epsilon=1, N=1, K_0=0, \nu_1=d/2,$ and $K_1= L_{0,d}^{\rm sc}$, it holds that, for any $\gamma \geq 0, \lambda>0$,
    \begin{align*}
        \Bigl|(-\Delta_\Omega^{\rm D}-\lambda)_\limminus^\gamma(x, x) - L_{\gamma, d}^{\rm sc}\lambda^{\gamma+\frac{d}2}\Bigr| 
        &\lesssim_{\gamma, d} \lambda^\gamma (1+d_\Omega(x)^2 \lambda)^{- \frac{1+\gamma}2}(d_\Omega(x)^{-d} + \lambda^{\frac{d}2})\,.
    \end{align*}
    Using the assumption that $d_\Omega(x) \geq \kappa/\sqrt{\lambda}$ we deduce that
    \begin{align*}
        \Bigl|(-\Delta_\Omega^{\rm D}-\lambda)_\limminus^\gamma(x, x) - L_{\gamma, d}^{\rm sc}\lambda^{\gamma+\frac{d}2}\Bigr| 
        &\lesssim_{\gamma, d, \kappa} \lambda^{\frac{\gamma+d-1}2} d_\Omega(x)^{-1-\gamma}\,.
    \end{align*}
    This completes the proof of Theorem \ref{thm: bulk} for $\# = {\rm D}$. 
    
    As mentioned at the start of this subsection a weaker form of Theorem~\ref{thm: bulk} when $\#= {\rm N}$ can be proved in the same manner. Indeed, if one in the above proof instead of Lemma~\ref{heatdlower} applies Lemma~\ref{heatnlower} (with suitably chosen $\eta$) this yields
    \begin{equation*}
        \left| (-\Delta_\Omega^{\rm N}-\lambda)_\limminus^\gamma(x,x) - L_{\gamma,d}^{\rm sc} \lambda^{\gamma+\frac{d}2} \right| \leq C_{\gamma, \Omega, \kappa} \lambda^{\frac{\gamma+d-1}2} d_\Omega(x)^{-1-\gamma}\,,
    \end{equation*}
    for all $x\in \Omega, \lambda>0$ with $d_\Omega(x)\geq \kappa/\sqrt{\lambda}$. When applying Lemma~\ref{heatnlower} one needs to choose the parameter $\eta$ so small that the obtained bound is valid for all $\lambda \geq \kappa^2/d_\Omega(x)^{2}$. Any $\eta$ so that $\kappa^2 \geq D\eta^2(1+\eta^{2})$ where $D$ is the constant appearing in Corollary~\ref{cor: Tauberian cor} suffices. Note that this bound is weaker than that in Theorem~\ref{thm: bulk} for two reasons; firstly, it requires that $\Omega$ is bounded with Lipschitz boundary and, secondly, the constant in the bound depends on $\Omega$. However, if one additionally assumes that $\Omega$ is convex then the boundedness assumption and the dependence of the constant on $\Omega$ can be dropped by using Lemma~\ref{lemm: diagonal kernel bound bulk convex} instead of Lemma~\ref{heatnlower}.
\end{proof}


\subsection{A second proof of Theorem~\ref{thm: bulk} using the wave equation}
\label{sec: Wave method}

In this subsection we give a wave equation proof of Theorem \ref{thm: bulk}. Our proof is not only valid for the Dirichlet and Neumann Laplacian but for general selfadjoint extensions $H$ of the symmetric operator $-\Delta|_{C^\infty_c(\Omega)}$. To simplify the exposition, we restrict ourselves to nonnegative extensions~$H$.

\begin{proposition}\label{ptwwave}
    For any $d\in\N$ and $\gamma\geq 0$ there is a constant $C_{\gamma,d}$ such that the following holds. Let $\Omega\subset\R^d$ be open and let $H$ be a nonnegative selfadjoint operator in $L^2(\Omega)$ such that $C^\infty_c(\Omega)\subset\dom H$ and for any $u\in C^\infty_c(\Omega)$ one has $Hu=-\Delta u$. Then for all $x\in\Omega$ and $\lambda>0$
    \begin{align*}
        & \left| (H-\lambda)_\limminus^\gamma(x,x) - L_{\gamma,d}^{\rm sc} \lambda^{\gamma+\frac{d}2} \right| \\
        & \leq C_{\gamma,d} \biggl( \frac{\lambda^{\frac{\gamma+d-1}2}}{d_\Omega(x)^{1+\gamma}} \, \1(d_\Omega(x)>1/\sqrt{\lambda}) + \frac{\lambda^\gamma}{d_\Omega(x)^d} \, \1(d_\Omega(x)\leq1/\sqrt{\lambda}) \biggr).
    \end{align*}
\end{proposition}

\begin{remark} A couple of remarks:
    \begin{enumerate}
        \item We emphasize that the constant $C_{\gamma,d}$ does not depend on $\Omega$ nor the chosen selfadjoint extension of the Laplacian.
        \item While for the Dirichlet and Neumann Laplacians the error term $\lambda^\gamma d_\Omega(x)^{-d}$ in the boundary region $d_\Omega(x)\leq 1/\sqrt{\lambda}$ can be improved (in the Neumann case at least when the boundary has some regularity), it is interesting to note that this is not possible for all selfadjoint extensions $H$ of the Laplacian; see discussion following Corollary~\ref{bergman} below.
        \item For $\gamma=0$ and $\gamma=1$ Proposition \ref{ptwwave} and its proof are closely related to bounds due to Safarov \cite{Safarov01}, although, disregarding the value of the constant $C_{\gamma,d}$, the bounds in \cite{Safarov01} for $\gamma=1$ are slightly worse than ours close to the boundary; see also Remark \ref{safarovrem} for an inaccuracy in \cite[Lemma~2.7]{Safarov01}. Our contribution here is to extend these bounds to all parameters $\gamma\geq 0$. We do this with the help of a novel Fourier Tauberian theorem, which will be stated and proved in Section \ref{sec: Fourier tauber} below.
    \end{enumerate}
\end{remark}

In the statement of Proposition \ref{ptwwave} we implicitly use the fact that any nonnegative selfadjoint extension $H$ of $-\Delta|_{C^\infty_c(\Omega)}$ has a spectral function. This was shown by G\aa rding \cite{Garding_ScandCongr54} (see also \cite{Garding54}), even for a larger class of operators. Since these references are not easily accessible and since we need some properties of G\aa rding's construction we reproduce a sketch of his argument.

\begin{lemma}\label{lem: spectral func}
	Let $\Omega\subset\R^d$ be open and let $H$ be a nonnegative selfadjoint operator in $L^2(\Omega)$ such that $C^\infty_c(\Omega) \subset\dom H$ and $Hu=-\Delta u$ for all $u\in C^\infty_c(\Omega)$. Then for every $\lambda> 0$ there is a function $e_\lambda\in C(\Omega\times\Omega)$ such that for all compactly supported $f,f'\in L^2(\Omega)$,
	$$
	(f,\1(H<\lambda) f') = \iint_{\Omega\times\Omega} \overline{f(y)} e_\lambda(y,y') f'(y')\,dy\,dy' \,.
	$$ 
\end{lemma}

\begin{proof}[Proof of Lemma~\ref{lem: spectral func}]
	We fix an integer $k>d/4$ and set
	$$
	F(x,y) := c_{k,d} |x-y|^{2k-d} \,,
	$$
	where the constant $c_{k,d}$ is chosen such that $(-\Delta)^k F(\cdot,y)=\delta(\cdot-y)$ for all $y\in\R^d$.
	
	Let $U\Subset\Omega$ be open and choose $\chi\in C^\infty_c(\Omega)$ such that $\chi=1$ on $U$. For $y\in U$, we consider the following functions on $\Omega$,
	$$
	F_0(\cdot,y) := (-\Delta)^k ((1-\chi)F(\cdot,y)) \,,
	\qquad
	F_1(\cdot,y) := \chi F(\cdot,y) \,.
	$$
	We claim that
	\begin{equation}
		\label{eq:continuity}
		U \ni y \mapsto F_j(\cdot,y)\in L^2(\Omega)
		\qquad\text{is continuous for}\ j=0,1 \,.
	\end{equation}
	This assertion for $j=1$ follows easily from the fact that $F(\cdot,y)\in L^2_\loc(\R^d)$ since $k>d/4$. The assertion for $j=0$ follows from the fact that $(1-\chi)F(\cdot,y)$ is smooth because $1-\chi$ vanishes in a neighbourhood of the point $y$ where $F(\cdot,y)$ is singular.
	
	Next, we claim that for any $f\in C^\infty_c(U)$, $\int_U F_1(\cdot,y)f(y)\,dy \in\dom H^k$ and
	\begin{equation}
		\label{eq:identity}
		\int_U F_0(\cdot,y)f(y)\,dy + H^k \int_U F_1(\cdot,y)f(y)\,dy = f \,.
	\end{equation}
	Indeed, since $f\in C^\infty_c(U)$, we have $\int_U F(x,y) f(y)\,dy \in C^\infty(\R^d)$. Multiplying this by $\chi\in C^\infty_c(\Omega)$, we infer that  $\int_U F_1(\cdot,y)f(y)\,dy \in C^\infty_c(\Omega) \subset \dom H$ and
	$$
	H \int_U F_1(\cdot,y)f(y)\,dy = -\Delta \int_U F_1(\cdot,y)f(y)\,dy \,.
	$$
	According to what we have just mentioned, the right side lies in $C^\infty_c(\Omega)\subset\dom H$. Therefore $\int_U F_1(\cdot,y)f(y)\,dy \in C^\infty_c(\Omega) \subset \dom H^2$ and
	$$
	H^2 \int_U F_1(\cdot,y)f(y)\,dy = \Delta^2 \int_U F_1(\cdot,y)f(y)\,dy \,.
	$$
	Iterating this argument, we deduce that 
	$\int_U F_1(\cdot,y)f(y)\,dy \in C^\infty_c(\Omega) \subset \dom H^k$ and
	$$
	H^k \int_U F_1(\cdot,y)f(y)\,dy = (-\Delta)^k \int_U F_1(\cdot,y)f(y)\,dy \,.
	$$
	The integral on the right side is equal to
	$$
	(-\Delta)^k \int_U F_1(\cdot,y)f(y)\,dy = (-\Delta)^k \int_U F(\cdot,y)f(y)\,dy
	- \int_U F_0(\cdot,y)f(y)\,dy \,,
	$$
	and the first term on the right is equal to $f$. This proves \eqref{eq:identity}.
	
	Now let $U'\subset\Omega$ be another bounded and open set and choose $\chi'\in C^\infty_c(\Omega)$ with $\chi'=1$ on $U'$. We define the functions $F_0'$ and $F_1'$ similarly as before.
	
	For $y\in U$ and $y'\in U'$ we define
	\begin{equation}
	    \label{eq:defspecfcn}
        e_\lambda(y,y') := \sum_{j,j'=0}^1 \int_{[0,\lambda)} \mu^{k(j+j')} d(F_j(\cdot,y),\1(H<\mu) F_{j'}'(\cdot,y')) \,.
	\end{equation}
	From \eqref{eq:continuity} we deduce that $e_\lambda\in C(U\times U')$. Moreover, for $f\in C^\infty_c(U)$ and $f'\in C^\infty_c(U')$ we find
	\begin{align*}
		& \iint_{U\times U'} \overline{f(y)} e_\lambda(y,y') f'(y')\,dy\,dy' \\
		& \quad = \sum_{j,j'=0}^1 \int_{[0,\lambda)} \mu^{k(j+j')} d \left( \int_U F_j(\cdot,y)f(y)\,dy ,\1(H<\mu) \int_U F_{j'}'(\cdot,y')f'(y')\,dy' \right)
	\end{align*}
	We have shown above that $\int_U F_j(\cdot,y)f(y)\,dy \in\dom H^{kj}$ (this is trivial for $j=0$) and similarly for the primed quantities. Thus,
	\begin{align*}
		& \int_{[0,\lambda)} \mu^{k(j+j')} d \left( \int_U F_j(\cdot,y)f(y)\,dy ,\1(H<\mu) \int_U F_{j'}'(\cdot,y')f'(y')\,dy' \right) \\
		& \quad = \int_{[0,\lambda)} d \left( H^{kj} \int_U F_j(\cdot,y)f(y)\,dy ,\1(H<\mu) H^{kj'} \int_U F_{j'}'(\cdot,y')f'(y')\,dy' \right) \\
		& \quad = \left( H^{kj} \int_U F_j(\cdot,y)f(y)\,dy ,\1(H<\lambda) H^{kj'} \int_U F_{j'}'(\cdot,y')f'(y')\,dy' \right).
	\end{align*}
	Using identity \eqref{eq:identity} we obtain
	\begin{align*}
		& \sum_{j,j'=0}^1 \left( H^{kj} \int_U F_j(\cdot,y)f(y)\,dy ,\1(H<\lambda) H^{kj'} \int_U F_{j'}'(\cdot,y')f'(y')\,dy' \right) \\
		&\quad = \left( \sum_{j=0}^1 H^{kj} \int_U F_j(\cdot,y)f(y)\,dy ,\1(H<\lambda) \sum_{j'=0}^1 H^{kj'} \int_U F_{j'}'(\cdot,y')f'(y')\,dy' \right) \\
		&\quad = \left( f ,\1(H<\lambda) f' \right).
	\end{align*}
	This proves the claimed identity for $f\in C^\infty_c(U)$ and $f'\in C^\infty_c(U')$.
	
	Since $e_\lambda\in C(U\times U')$ the identity extends to $f$ and $f'$ whose supports are compact subsets of $U$ and $U'$, respectively.
	
	Since $U$ and $U'$ are arbitrary bounded open sets whose closures are contained in $\Omega$, one easily concludes that the function $e_\lambda$ is defined on all of $\Omega\times\Omega$, is continuous there and satisfies the claimed identity.
\end{proof}

\begin{proof}[Proof of Proposition~\ref{ptwwave}]
    We fix a point $x_0\in\Omega$ (which plays the role of the point $x$ in the inequality in the proposition). For $r<d_\Omega(x_0)$ and $f\in C^\infty_c(B_r(x_0))$ we consider
    $$
    u(t):= \cos(t\sqrt H) f \,,
    $$
    which satisfies $\partial_t^2 u +H u = 0$. Moreover, $u(0) = f$ and $\partial_t u(0)=0$. Testing the equation against a test function in $C^\infty_c(\Omega)\subset\dom H$, we see that for any $t\in\R$ we have $\partial_t^2 u -\Delta u = 0$ in $\Omega$ in the sense of distributions. Moreover, noting that $u(t)\in\dom H$ for all $t$ (since $f\in C^\infty_c(\Omega)\subset\dom H$) and that $\dom H\subset H^2_\loc(\Omega)$ (by interior elliptic regularity since $-\Delta u=Hu\in L^2(\Omega)$ for $u\in\dom H$), we see that the equation
    $\partial_t^2 u -\Delta u = 0$ holds in $L^2_\loc(\Omega)$.
    
    We now deduce, using finite speed of propagation, that 
    $$
    \supp u(t) \subset B_{r+|t|}(x_0) \cup \{ x:\ d_\Omega(x) \leq |t| \} \,.
    $$
    To prove this we can use the classical energy method; see, e.g., \cite[Theorem 2.4.6]{Evans_PDE}. Indeed, given $x' \notin B_{r+|t|}(x_0) \cup \{ x:\ d_\Omega(x) \leq |t| \}$ it holds that the backward cone
    \begin{equation*}
          C_{x',t} := \{(y, s) \in \Omega \times [0, t]: |y-x'|\leq t-s\}
    \end{equation*}
    satisfies that $C_{x',t} \subset \Omega \times [0, t]$. Since $u(0), \partial_t u(0)$ both vanish in $B_t(x') = \{y\in \Omega : (y, 0)\in C_{x', t}\}$ we can now argue precisely as in \cite[Theorem 2.4.6]{Evans_PDE} to conclude that $u$ vanishes in $C_{x',t}$ and in particular $u(t)$ vanishes at $x'$. Note that in the proof of \cite[Theorem 2.4.6]{Evans_PDE} one only needs to consider balls whose closure is contained in $\Omega$ and on such balls $u(t)$ belongs to $H^2$, so all the manipulations are justified.

    As an aside we mention that if $\Omega$ has sufficiently regular boundary and if $H$ corresponds to either a Dirichlet, Neumann or Robin boundary condition, then the above conclusion can be strengthened to $\supp u(t) \subset B_{r+|t|}(x_0)$. For the proof one can argue as in \cite[Theorem 2.4.6]{Evans_PDE}, but now also including the intersection of balls with $\Omega$. This improved bound on the support of $u(t)$ could be used to improve the constants in our error bound, but does not significantly improve the result.

    Restricting ourselves to $t$ such that $|t|\leq\frac12\, ( d_\Omega(x_0)-r)$ and denoting by $\chi(t)$ the characteristic function of $\{ x\in\Omega:\ d_\Omega(x) >|t|\}$, we see that $\chi(t)u(t)$ solves the wave equation in $\Omega$.

    Extending $f$ by zero to a function on all of $\R^d$, we define
    $$
    \widetilde u(t) := \cos(t\sqrt{-\Delta_{\R^d}})f \,,
    $$
    which satisfies $\partial_t^2-\Delta u= 0$ on $\R^d\times\R$, as well as $\widetilde u(0)=f$ and $\partial_t \widetilde u(0)=0$. Applying the uniqueness theorem \cite[Theorem 7.2.4]{Evans_PDE} for the wave equation, we find that
    $$
    \chi(t) u(t) = \widetilde u(t)
    \qquad\text{for all}\ t\in [-\tfrac12(d_\Omega(x_0)-r),\tfrac12(d_\Omega(x_0)-r)] \,.
    $$

    Let $\phi\in\mathcal S(\R)$ with $\supp\hat\phi\subset(-\frac12 d_\Omega(x_0), \frac12 d_\Omega(x_0)))$. Then for all sufficiently small $r>0$ (specifically for all $r>0$ such that $[-\frac12(d_\Omega(x_0)-r),\frac12(d_\Omega(x_0)-r)]\supset \supp\hat\phi$) and all $f\in C^\infty_c(B_r(x_0))$ we have
    $$
    \int_\R \hat\phi(t) u(t)\,dt = \int_\R \hat\phi(t) \widetilde u(t)\,dt \,.
    $$
    Taking the inner product with $f$ and appealing to the spectral theorem, we deduce that
    \begin{align*}
    \int_\R \hat\phi(t) \biggl(\int_{[0,\infty)} \cos(t\sqrt\lambda) &\, d(f,\1(H<\lambda)f)\biggr)dt\\
    &= \int_\R \hat\phi(t) \biggl(\int_{[0,\infty)} \cos(t\sqrt\lambda) \, d(f,\1(-\Delta_{\R^d}<\lambda)f)\biggr)dt
    \end{align*}
    
    Interchanging the integrals, we obtain
    \begin{align*}
    \int_{[0,\infty)} \frac12(\phi(\sqrt\lambda)+\phi(-\sqrt\lambda)) &\, d(f,\1(H<\lambda)f)\\
    &= \int_{[0,\infty)} \frac12(\phi(\sqrt\lambda)+\phi(-\sqrt\lambda)) \, d(f,\1(-\Delta_{\R^d}<\lambda)f)
    \end{align*}
    Let us set
    $$
    \eta_f(\tau) := (\sgn\tau) \, (f,\1(H<\tau^2)f)
    \qquad\text{and}\qquad
    \widetilde\eta_f(\tau) := (\sgn\tau) \, (f,\1(-\Delta_{\R^d}<\tau^2)f) \,,
    $$
    which are nondecreasing functions on $\R$, so $d\eta$ and $d\widetilde\eta$ are measures on $\R$ via the Lebesgue--Stieltjes construction. We see that
    \begin{equation}
        \label{eq:waveidentity}
        \int_\R \phi(\tau) \,d\eta_f(\tau) = \int_\R \phi(\tau) \,d\widetilde\eta_f(\tau) \,.
    \end{equation}
    We recall that we have proved this identity for every $f\in C^\infty_c(B_r(x_0))$ with $r$ sufficiently small (depending on $\phi$).

    Given $\chi\in C^\infty_c(\R^d)$ with $\int_{\R^d} \chi(x)\,dx = 1$, we apply the above identity to $f_\rho(x):=\rho^{-d} \chi((x-x_0)/\rho)$ with $\rho>0$ sufficiently small. We claim that
    $$
    \lim_{\rho \to 0^\limplus} \int_\R \phi(\tau) \,d\eta_{f_\rho}(\tau) = \int_\R \phi(\tau) \,d\eta_{x_0}(\tau)
    \quad\text{and}\quad
    \lim_{\rho \to 0^\limplus}\int_\R \phi(\tau) \,d\widetilde\eta_{f_\rho}(\tau) = \int_\R \phi(\tau) \,d\widetilde\eta_{x_0}(\tau)
    $$
    where
    \begin{align*}
        \eta_{x_0}(\tau) := (\sgn\tau) \, \1(H<\tau^2)(x_0,x_0)
        \quad\text{and}\quad
        \widetilde\eta_{x_0}(\tau) := (\sgn\tau) \, \1(-\Delta<\tau^2)(x_0,x_0) \,.    
    \end{align*}
    The first convergence (the second one is proved similarly) follows from the fact that for any bounded, measurable and rapidly decaying function $\psi$ on $[0,\infty)$ one has
    $$
    \lim_{\rho \to 0^\limplus}\int_\R \psi(\lambda) \,d (f_\rho,\1(H<\lambda) f_\rho) = \int_\R \psi(\lambda) \,d e_\lambda(x_0,x_0) \,.
    $$
    To prove the latter convergence, we recall the definition \eqref{eq:defspecfcn} of the spectral function, choosing $U=U'$, an open neighbourhood of $x_0$. We obtain
    $$
    (f_\rho,\1(H<\lambda) \, f_\rho) = \sum_{j,j'=0}^1 \int_{[0,\lambda)} \mu^{k(j+j')} \, d \left( F_{j,\rho}, \1(H<\mu) \, F_{j',\rho} \right) 
    $$
    with $F_{j,\rho}:= \int_U F_j(\cdot,y) f_\rho(y)\,dy$. Thus,
    $$
    \int_\R \psi(\lambda) \,d (f_\rho,\1(H<\lambda) f_\rho)
    = \sum_{j,j'=0}^1 (F_{j,\rho}, H^{k(j+j')}\Psi(H) F_{j',\rho}) 
    $$
    with $\Psi(\mu):= \int_\mu^\infty \psi(\lambda)\,d\lambda$. By our assumptions on $\psi$ the operator $H^{k(j+j')}\Psi(H)$ is bounded. The same arguments show that
    $$
    \int_\R \psi(\lambda) \,d e_\lambda(x_0,x_0) = \sum_{j,j'=0}^1 (F_j(\cdot,x_0), H^{k(j+j')}\Psi(H) F_{j'}(\cdot,x_0)) 
    $$
    Since $F_{j,\rho}\to F_j(\cdot,x_0)$ in $L^2(\Omega)$ and since $H^{k(j+j')}\Psi(H)$ is a bounded operator, we obtain the claimed convergence.
 
    As a consequence of this convergence, we deduce from \eqref{eq:waveidentity} that
    \begin{equation}
        \label{eq:waveidentity2}
        \int_\R \phi(\tau) \,d\eta_{x_0}(\tau) = \int_\R \phi(\tau) \,d\widetilde\eta_{x_0}(\tau) \,.    
    \end{equation}
    
    Let us show that the measure $d\eta_{x_0}$ is a tempered distribution. Since the derivative of a tempered distribution is a tempered distribution, it suffices to show that the function $\lambda\mapsto\eta_{x_0}(\lambda)$ is locally bounded and has at most polynomial growth. These properties follow immediately from the definition \eqref{eq:defspecfcn} (with $x_0\in U=U'$), which implies
    $$
    0\leq e_\lambda(x_0,x_0) \leq \sum_{j,j'=0}^1 \lambda^{k(j+j')} \|F_j(\cdot,x_0)\|_{L^2(\Omega)}^2 \,.
    $$

    Moreover, using the explicit diagonalization of $-\Delta_{\R^d}$ in terms of the Fourier transform we see that
    $$
    \widetilde\eta_{x_0} = (\sgn\tau) \, L_{0,d}^{\rm sc} \, |\tau|^d \,.
    $$
    
    Since $d\eta_{x_0}$ and $d\widetilde\eta_{x_0}$ are tempered distributions and since \eqref{eq:waveidentity2} is valid for all $\phi\in\mathcal S(\R)$ with $\supp\hat\phi\subset(-\frac12d_\Omega(x_0),\frac12d_\Omega(x_0)))$, it follows that the (distributional) Fourier transform of the difference $d\eta_{x_0}-d\widetilde\eta_{x_0}$ vanishes on the set $(-\frac12d_\Omega(x_0),\frac12d_\Omega(x_0))$. Consequently, we have
    $$
    \phi_a *(d\eta_{x_0}-d\widetilde\eta_{x_0})=0
    $$ 
    for any $\phi\in\mathcal S(\R)$ with $\supp\hat\phi\subset(-1,1)$ and $\phi_a(\tau)= a^{-1} \phi(\tau/a)$ with $a = 2/d_\Omega(x_0)$. Replacing $\phi$ with a dilate and letting the dilation parameter tend to 1, we see that the support assumption on $\hat\phi$ may be relaxed to $\supp\hat\phi\subset[-1,1]$.

    We are now in position to apply Corollary \ref{cor: Fourier Tauberian general gamma}. Indeed, for $\mu=d\eta_{x_0}$ and $\nu=d\tilde\eta_{x_0}$ we have shown above that the assumptions of the corollary are satisfied with parameters $a=a_0=a_1 = 2/d_\Omega(x_0)$, $\alpha=d-1$, $M_0$ a constant depending only on $d$ and $M_1=0$. The resulting bound states that
    \begin{equation*}
        |R_{d\eta_{x_0}}^\gamma(\tau)-R^\gamma_{d\tilde\eta_{x_0}}(\tau)| \leq C_{\gamma, d} d_\Omega(x_0)^{-d}(1+d_\Omega(x_0)|\tau|)^{d-1-\gamma}\,.
    \end{equation*}
    Taking $\tau =\sqrt{\lambda}$ this leads to the claimed bound, since for all $\tau>0$
    $$
        R^\gamma_{d\eta_{x_0}}(\tau)=\frac{2\gamma}{\tau} \int_0^\tau \left( 1- \frac{\sigma^2}{\tau^2} \right)^{\gamma-1} \frac{\sigma}{\tau} \, \eta_{x_0}(\sigma)\,d\sigma = \tau^{-2\gamma} \,(H-\tau^2)_\limminus^\gamma(x_0,x_0)
    $$
    and
    $$
        R^\gamma_{d\tilde\eta_{x_0}}(\tau)=\frac{2\gamma}{\tau} \int_0^\tau \left( 1- \frac{\sigma^2}{\tau^2} \right)^{\gamma-1} \frac{\sigma}{\tau} \, \widetilde\eta_{x_0}(\sigma)\,d\sigma = L_{\gamma,d}^{\rm sc} \, \tau^{d} \,.
    $$
    This completes the proof of the proposition.    
\end{proof}

\begin{proof}[Proof of Theorem \ref{thm: bulk}]
    To prove Theorem~\ref{thm: bulk} we apply Proposition~\ref{ptwwave} with $H=-\Delta_\Omega^{\#}$. If $\kappa \geq 1$ then we are done. If $\kappa <1$ then we observe that for $x \in \Omega$ with $\kappa/\sqrt{\lambda} \leq d_\Omega(x) \leq 1/\sqrt{\lambda}$ it holds that
    \begin{equation*}
        \frac{\lambda^\gamma}{d_\Omega(x)^d} = (d_\Omega(x)\sqrt{\lambda})^{-d+1+\gamma}\frac{\lambda^{\frac{\gamma+d-1}2}}{d_\Omega(x)^{1+\gamma}} \leq \max\{\kappa^{-d+1+\gamma},1\} \frac{\lambda^{\frac{\gamma+d-1}2}}{d_\Omega(x)^{1+\gamma}}\,.
    \end{equation*}
    This completes the proof.
\end{proof}

While not related to our main purpose, we find the following consequence of Proposition \ref{ptwwave} noteworthy. In particular, this result implies the sharpness of Proposition~\ref{ptwwave} close to the boundary. The result concerns the harmonic Bergman space and we refer to \cite[Chapter 8]{ABR01} for further information. We note that in a general open set $\Omega\subset\R^d$ the space of harmonic functions in $L^2(\Omega)$ is a closed subspace of $L^2(\Omega)$. Moreover, the orthogonal projection $R_\Omega$ onto this subspace has an integral kernel, which we are going to denote by $R_\Omega(x, x')$. The following corollary gives a bound on the singularity of this kernel when $x$ approaches the boundary.

\begin{corollary}\label{bergman}
    Let $d\geq 1$. There is a constant $C_d$ such that for any open set $\Omega\subset\R^d$ the kernel $R_\Omega(\cdot,\cdot)$ of the orthogonal projection in $L^2(\Omega)$ onto the subspace of harmonic functions satisfies for all $x\in\Omega$
    $$
    R_\Omega(x, x) \leq C_d \, d_\Omega(x)^{-d} \,.
    $$
\end{corollary}

\begin{proof}
    We consider the selfadjoint operator $H$ in $L^2(\Omega)$ defined as the Krein extension of $-\Delta$ defined on $C^\infty_c(\Omega)$ \cite{AlonsoSimon80,GesztesyMitrea11}. It follows immediately from the definition of this extension that $\ker H$ is the subspace of harmonic functions in $L^2(\Omega)$. From Proposition \ref{ptwwave} we deduce that, for any $x\in\Omega$ and any $\lambda>0$,
    \begin{align*}
        R_\Omega(x, x) & \leq \1(H<\lambda)(x,x) \\
        & \leq L_{0,d}^{\rm sc} \lambda^{\frac{d}2} + C_{0,d} \biggl( \frac{\lambda^{\frac{d-1}2}}{d_\Omega(x)} \, \1(d_\Omega(x)>1/\sqrt{\lambda}) + \frac{1}{d_\Omega(x)^d} \, \1(d_\Omega(x)\leq1/\sqrt{\lambda}) \biggr).    
    \end{align*}
    For fixed $x$, the right side tends to $C_{0,d}\,d_\Omega(x)^{-d}$ as $\lambda\to 0$. This proves the claimed bound.
\end{proof}

The first remarkable fact about the bound in Corollary \ref{bergman} is that the order of the singularity $d_\Omega(x)^{-d}$ is sharp. This follows when $\Omega$ is a ball or a halfspace from the explicit formulas for $R_\Omega$ in \cite[Theorems 8.13 and 8.24]{ABR01} and when $\Omega$ is a smooth domain from~\cite{Englis15}. The second remarkable fact of the bound is that the constant depends only on $d$ and not on the set $\Omega$.

Let us now return to the question of sharpness of the bound in Proposition~\ref{ptwwave}. It is natural to wonder if the error bound $\lambda^\gamma d_\Omega(x)^{-d}$ in the boundary region $d_\Omega(x)\leq 1/\sqrt{\lambda}$ is optimal for certain selfadjoint extensions $H$.  We argue that it is, at least for $\gamma=0$ and for smooth sets $\Omega$. Indeed, this follows from the sharpness of the bound in Corollary~\ref{bergman} which was deduced from Proposition \ref{ptwwave} applied with $\gamma=0$ and $H$ as the Krein extension of $-\Delta$.


\section{Integrated one-term Weyl law}\label{sec:intweyl}

In this section we will use the pointwise bounds on the spectral function and its Riesz means from the previous section to obtain bounds on the Riesz means $\Tr(-\Delta_\Omega^\#-\lambda)_\limminus^\gamma$. The idea of the proof is to write the Riesz mean as an integral of the corresponding integral kernel and then use our pointwise bounds for the kernels appropriately depending on the relative size of $\sqrt{\lambda}\,d_\Omega(x)$. Under the assumption $\gamma>0$ our bounds will have an order-sharp remainder.

It will be convenient to abbreviate
\begin{align*}
	\vartheta_\Omega(l) := \frac{|\{x\in \Omega: d_\Omega(x)\leq l\}|}{l}\,,\qquad
	\Theta_\Omega := \sup_{l>0} \vartheta_\Omega(l)\,,
\end{align*}
where we recall that the distance $d_\Omega$ to the boundary was introduced in \eqref{eq:distance}.

\begin{theorem}
	\label{thm: Sharp Weyld}
	Let $d\geq 1$ and $\gamma>0$. Then there is a constant $C_{\gamma,d}$ such that for any open set $\Omega \subset \R^d$ with $|\Omega|<\infty$ and $\Theta_\Omega <\infty$ and for all $\lambda> 0$ we have
	$$
	\bigl| \Tr(-\Delta_\Omega^{\rm D}-\lambda)_\limminus^\gamma - L_{\gamma,d}^{\rm sc}|\Omega| \lambda^{\gamma+\frac d2} \bigr|\leq  
	C_{\gamma,d} \, \Theta_\Omega \, \lambda^{\gamma+\frac{d-1}2} \,.
	$$
\end{theorem}

\begin{theorem}
	\label{thm: Sharp Weyln}
	Let $d\geq 1$ and $\gamma>0$. Then for any open set $\Omega \subset \R^d$ with $|\Omega|<\infty$, $\Theta_\Omega<\infty$ and the extension property and for all $\lambda_0>0$ there is a constant $C_{\gamma,\Omega,\lambda_0}$ such that for all $\lambda> 0$ we have
	$$
	\bigl| \Tr(-\Delta_\Omega^{\rm N}-\lambda)_\limminus^\gamma - L_{\gamma,d}^{\rm sc}|\Omega| \lambda^{\gamma+\frac d2} \bigr|\leq 
	C_{\gamma,\Omega,\lambda_0} \, \lambda^{\gamma+\frac{d-1}2} \, \max\bigl\{1,(\lambda_0/\lambda)^\frac{d-1}2 \bigr\} \,.
	$$
\end{theorem}

\begin{theorem}
    \label{thm: Sharp Weylnc}
    Let $d\geq 1$ and $\gamma>0$. Then there is a constant $C_{\gamma,d}$ such that for any bounded, convex, open set $\Omega\subset\R^d$ and for all $\lambda> 0$, we have
    $$
    \bigl| \Tr(-\Delta_\Omega^{\rm N}-\lambda)_\limminus^\gamma - L_{\gamma,d}^{\rm sc}|\Omega| \lambda^{\gamma+\frac d2} \bigr|\leq 
	C_{\gamma,d} \, \lambda^{\gamma+\frac{d-1}2} \, \max \bigl\{ 1 , \bigl(r_{\rm in}(\Omega) \sqrt\lambda\bigr)^{-d+1} \bigr\} \mathcal H^{d-1}(\partial\Omega) \,.
    $$
\end{theorem}

\begin{remark} A couple of remarks:
    \begin{enumerate}
        \item The bounds in Theorems \ref{thm: Sharp Weyld}, \ref{thm: Sharp Weyln} and \ref{thm: Sharp Weylnc} are important ingredients in the proofs of our main results, Theorems \ref{thm: Weyl asymptotics Lipschitz} and \ref{thm: Weyl asymptotics convex}. Once these main results have been shown, we deduce conversely that the error bounds in Theorems~\ref{thm: Sharp Weyld}, \ref{thm: Sharp Weyln} and \ref{thm: Sharp Weylnc} are order-sharp as $\lambda\to\infty$.

        \item We find the simple dependence on $\Omega$ of the error bounds in Theorems \ref{thm: Sharp Weyld} and \ref{thm: Sharp Weylnc} quite remarkable.

        \item Theorem \ref{thm: Sharp Weyld} provides an analogue for $\gamma>0$ of Seeley's bound \eqref{eq:seeleyintro} from \cite{Seeley78,Seeley80}. The latter bound was proved under several smoothness conditions on the boundary, while our point is that no such conditions are necessary, as soon as one moves away from $\gamma=0$.

        \item The special case $\gamma=1$ of Theorem \ref{thm: Sharp Weyld} is proved in Safarov's paper \cite{Safarov01}, except for the issue discussed in Remark \ref{safarovrem}.
    \end{enumerate}
\end{remark}


\begin{proof}[Proof of Theorem~\ref{thm: Sharp Weyld}]
    Decomposing $\Omega$ into the two sets $\{x\in \Omega: d_\Omega(x)\geq 1/\sqrt{\lambda}\}$ and $\{x\in \Omega: d_\Omega(x)<1/\sqrt{\lambda}\}$ we have
    \begin{align*}
        \Bigl| \Tr(-\Delta_\Omega^{\rm D} - \lambda)_\limminus^\gamma &- L_{\gamma,d}^{\rm sc} |\Omega| \lambda^{\gamma+\frac{d}2} \Bigr|\\
        &= \biggl| \int_\Omega\Bigl((-\Delta_\Omega^{\rm D} - \lambda)_\limminus^\gamma(x, x) - L_{\gamma,d}^{\rm sc} \lambda^{\gamma+\frac{d}2} \Bigr)dx\biggr|\\
        &\leq
        \int_{\{x\in \Omega: d_\Omega(x)<1/{\sqrt{\lambda}}\}}|(-\Delta_\Omega^{\rm D} - \lambda)_\limminus^\gamma(x, x) - L_{\gamma,d}^{\rm sc} \lambda^{\gamma+\frac{d}2}|\,dx\\
        &\quad +
        \int_{\{x\in \Omega: d_\Omega(x)\geq 1/{\sqrt{\lambda}}\}}|(-\Delta_\Omega^{\rm D} - \lambda)_\limminus^\gamma(x, x) - L_{\gamma,d}^{\rm sc} \lambda^{\gamma+\frac{d}2}|\,dx\,.
    \end{align*}
    The first integral we bound using Proposition~\ref{prop:rieszptwbgd}, for the second we use Theorem~\ref{thm: bulk} with $\kappa=1$. This yields
    \begin{align*}
        \Bigl| \Tr(-\Delta_\Omega^{\rm D} - \lambda)_\limminus^\gamma - L_{\gamma,d}^{\rm sc} |\Omega| \lambda^{\gamma+\frac{d}2} \Bigr|
        &\lesssim_{\gamma,d}
        \lambda^{\frac{\gamma+d-1}{2}}\biggl(\lambda^{\frac{1+\gamma}2}|\{x\in \Omega: d_\Omega(x)<1/\sqrt{\lambda}\}|\\
        &\quad +
        \int_{\{x\in \Omega: d_\Omega(x)\geq 1/\sqrt{\lambda}\}}d_\Omega(x)^{-1-\gamma}\,dx\biggr)\\
        &\lesssim_{\gamma,d}
        \lambda^{\frac{\gamma+d-1}{2}} \int_{\Omega}\min\bigl\{\lambda^{\frac{1+\gamma}2}, d_\Omega(x)^{-1-\gamma}\bigr\}\,dx\,.
    \end{align*}
	It remains to prove an estimate for the integral in the last line.

    By the layer cake representation, Fubini's theorem, and a change of variables
    \begin{equation}\label{eq: layer cake integrated one-term Weyl}
    \begin{aligned}
        \int_\Omega \min\bigl\{\lambda^{\frac{1+\gamma}2}, d_\Omega(x)^{-1-\gamma} \bigr\} \,dx
        &= 
        \int_\Omega \int_0^{\lambda^{\frac{1+\gamma}2}} \1_{\{y\in \Omega:  d_\Omega(y)^{-1-\gamma}\geq t\}}(x)\,dtdx\\
        &= 
        \int_0^{\lambda^{\frac{1+\gamma}2}} |\{x\in \Omega: d_\Omega(x)^{-1-\gamma} \geq t\}|\,dt\\
        &= 
        (1+\gamma)\int_{1/\sqrt{\lambda}}^\infty \frac{|\{x\in \Omega: d_\Omega(x)\leq s \}|}{s^{2+\gamma}}\,ds\,.
    \end{aligned}
    \end{equation}
    Since
    \begin{equation*}
        \int_{1/\sqrt{\lambda}}^\infty \frac{|\{x\in \Omega: d_\Omega(x)\leq s \}|}{s^{2+\gamma}}\,ds = \int_{1/\sqrt{\lambda}}^\infty \frac{\vartheta_\Omega(s)}{s^{1+\gamma}}\,ds \leq \Theta_\Omega \int_{1/\sqrt{\lambda}}^\infty \frac{1}{s^{1+\gamma}}\,ds = \frac{\Theta_\Omega}{\gamma}\lambda^{\frac{\gamma}2}\,,
    \end{equation*}
    this completes the proof of Theorem~\ref{thm: Sharp Weyld}.
\end{proof}

\begin{proof}[Proof of Theorem~\ref{thm: Sharp Weyln}]
   As in the proof of Theorem~\ref{thm: Sharp Weyld} we decompose $\Omega$ into the two sets $\{x\in \Omega:d_\Omega(x)\geq  1/\sqrt{\lambda}\}$ and $\{x\in \Omega:d_\Omega(x)< 1/\sqrt{\lambda}\}$. In the first set we again apply Theorem~\ref{thm: bulk} with $\kappa=1$ in the second we apply Proposition~\ref{prop:rieszptwbgn} with $\lambda_0$ replaced by $\tilde \lambda_0$. Assuming that $r_{\rm in}(\Omega)\geq 1/\sqrt{\lambda}$, this yields
    \begin{align*}
        \Bigl| \Tr(-\Delta_\Omega^{\rm N} - \lambda)_\limminus^\gamma &- L_{\gamma,d}^{\rm sc} |\Omega| \lambda^{\gamma+\frac{d}2} \Bigr| \\
        &\leq 
            \int_\Omega \bigl|(-\Delta_\Omega^{\rm N}-\lambda)_\limminus^\gamma(x, x)-L_{\gamma, d}^{\rm sc}\lambda^{\gamma+\frac{d}2}\bigr|\,dx\\
        &\leq C_{\gamma, \Omega, \tilde\lambda_0}\lambda^{\gamma+\frac{d}2}\bigl(1+(\tilde\lambda_0/\lambda)^{\frac{d}2}\bigr)|\{x\in \Omega: d_\Omega(x)<1/\sqrt{\lambda}\}|\\
        &\quad + C_{\gamma, d} \lambda^{\frac{\gamma+d-1}2}\int_{\{x\in \Omega: d_\Omega(x) \geq 1/\sqrt{\lambda}\}} d_\Omega(x)^{-1-\gamma}\,dx\\
        &\lesssim_{\gamma, \Omega, \tilde \lambda_0}
            \lambda^{\frac{\gamma+d-1}{2}}\bigl(1+(\tilde\lambda_0/\lambda)^{\frac{d}2}\bigr)\int_\Omega \min\bigl\{\lambda^{\frac{1+\gamma}2}, d_\Omega(x)^{-1-\gamma} \bigr\} \,dx\,.
    \end{align*}
    The bound stemming from~\eqref{eq: layer cake integrated one-term Weyl} therefore implies that
    \begin{align*}
        \Bigl| \Tr(-\Delta_\Omega^{\rm N} - \lambda)_\limminus^\gamma - L_{\gamma,d}^{\rm sc} |\Omega| \lambda^{\gamma+\frac{d}2} \Bigr| 
        &\leq C_{\gamma, \Omega, \tilde\lambda_0}\lambda^{\gamma+\frac{d-1}2}\bigl(1+(\tilde\lambda_0/\lambda)^{\frac{d}2}\bigr)\Theta_\Omega\\
        &= C_{\gamma, \Omega, \tilde\lambda_0}\Theta_\Omega\bigl(\lambda^{\gamma+\frac{d-1}2}+ \lambda^{\gamma-\frac{1}2}\tilde\lambda_0^{\frac{d}2}\bigr)\\
        &\leq C_{\gamma, \Omega, \tilde\lambda_0}\Theta_\Omega\bigl(\lambda^{\gamma+\frac{d-1}2}+ \lambda^{\gamma}r_{\rm in}(\Omega)\tilde\lambda_0^{\frac{d}2}\bigr)\,.
    \end{align*}

    If $r_{\rm in}(\Omega)\leq 1/\sqrt{\lambda}$, then the set $\{x\in \Omega: d_\Omega(x)\geq 1/\sqrt{\lambda}\}$ is empty and Proposition~\ref{prop:rieszptwbgn} (with $\lambda_0$ replaced by $\tilde\lambda_0$) yields
    \begin{align*}
        \Bigl| \Tr(-\Delta_\Omega^{\rm N} - \lambda)_\limminus^\gamma - L_{\gamma,d}^{\rm sc} |\Omega| \lambda^{\gamma+\frac{d}2} \Bigr| 
        &\leq C_{\gamma, \Omega, \tilde\lambda_0}\lambda^{\gamma+\frac{d}2}\bigl(1+(\tilde\lambda_0/\lambda)^{\frac{d}2}\bigr)|\Omega|\\
        &\leq C_{\gamma, \Omega, \tilde\lambda_0}\lambda^{\gamma+\frac{d}2}\bigl(1+(\tilde\lambda_0/\lambda)^{\frac{d}2}\bigr)\Theta_\Omega r_{\rm in}(\Omega)\\
        &\leq C_{\gamma, \Omega, \tilde\lambda_0}\Theta_\Omega\bigl(\lambda^{\gamma+\frac{d-1}2}+\lambda^\gamma r_{\rm in}(\Omega) \tilde\lambda_0^{\frac{d}2}\bigr)\,.
    \end{align*}
    Given $\lambda_0$ we choose $\tilde\lambda_0$ such that $r_{\rm in}(\Omega) \tilde\lambda_0^{\frac{d}2} = \lambda_0^\frac{d-1}2$ (note that $r_{\rm in}(\Omega)\lesssim_d |\Omega|^{1/d} < \infty$). This completes the proof of Theorem~\ref{thm: Sharp Weyln}.
\end{proof}

\begin{remark}
    Concerning the assumptions of Theorem \ref{thm: Sharp Weyln}, we note that any open set $\Omega\subset\R^d$ with the extension property and finite measure is bounded. This follows easily from the density bound \cite[Equation (2.64)]{LTbook}; see also the references therein.   
\end{remark}

\begin{proof}[Proof of Theorem \ref{thm: Sharp Weylnc}]
    Similarly to the previous proofs we split the integral and this time apply Proposition \ref{prop:rieszptwbgnc} in the set $\{ x\in\Omega:\ d_\Omega(x)<1/\sqrt\lambda\}$ and Theorem~\ref{thm: bulk} in the set $\{ x\in\Omega:\ d_\Omega(x)\geq 1/\sqrt\lambda\}$ with $\kappa=1$. In this way we find
    \begin{equation}\label{eq:sharpweylncproof}
    \begin{aligned}
        \Bigl| \Tr(-\Delta_\Omega^{\rm N} - \lambda)_\limminus^\gamma - L_{\gamma,d}^{\rm sc} |\Omega| \lambda^{\gamma+\frac{d}2} \Bigr| & \lesssim_{\gamma,d} \lambda^{\frac{\gamma+d-1}2} \int_\Omega \min\bigl\{\lambda^{\frac{1+\gamma}2}, d_\Omega(x)^{-1-\gamma}\bigr\}\,dx \\
        & \quad + \lambda^\gamma \int_{\{x\in \Omega:d_\Omega(x) <1/\sqrt\lambda\}} V_\Omega(x,1/\sqrt\lambda)^{-1} \,dx\,,
    \end{aligned}
    \end{equation}
    where we abbreviate $V_\Omega(x,r):=|\Omega\cap B_r(x)|$.

    In Lemma \ref{lem: volume bdry neighbourhood bound} we shall prove that
    $$
    |\{ x\in\Omega:\ d_\Omega(x)<s\}| \leq s \mathcal H^{d-1}(\partial\Omega) \,.
    $$
    Therefore, using \eqref{eq: layer cake integrated one-term Weyl} 
    we find
    \begin{align*}
        \int_\Omega \min\bigl\{ \lambda^{\frac{1+\gamma}2}, d_\Omega(x)^{-1-\gamma}\bigr\}\,dx 
        & \leq (1+\gamma) \int_{1/\sqrt{\lambda}}^\infty \frac{\mathcal H^{d-1}(\partial\Omega)}{s^{1+\gamma}}\,ds \\
        & = \frac{1+\gamma}{\gamma} \, \lambda^{\frac{\gamma}2}\, \mathcal H^{d-1}(\partial\Omega)  \,.
    \end{align*}
    Thus, the first term on the right side of \eqref{eq:sharpweylncproof} is bounded by the quantity in the right side of the inequality in Theorem~\ref{thm: Sharp Weylnc}. Therefore, the theorem will follow if we can show that for all $\ell > 0$ we have
	\begin{align}
        \label{eq:convexintegratedv}
		\int_{\{x\in \Omega:d_\Omega(x) <\ell\}} \frac{\ell^{d}}{V_\Omega(x,\ell)}\,dx
		\lesssim_{d} \Bigl[\ell+ \ell^{d}r_{\rm in}(\Omega)^{1-d}\Bigr] \Haus^{d-1}(\partial\Omega)\,.
	\end{align}
    Indeed, this inequality with $\ell=1/\sqrt\lambda$ shows that also the second term on the right side of \eqref{eq:sharpweylncproof} bounded by that in the right side of the inequality in Theorem~\ref{thm: Sharp Weylnc}.

    We now turn to the proof of \eqref{eq:convexintegratedv}. Let $\{x_k\}_{k=1}^M\subset \{x\in \Omega:d_\Omega(x) <\ell\}$ satisfy $B_{\ell}(x_k)\cap B_{\ell}(x_j)=\emptyset$ for $k \neq j$ and $\{x\in \Omega:d_\Omega(x) <\ell\}\subset \cup_{k=1}^M B_{2\ell}(x_k)$. The existence of such a collection can be proven by induction as follows. Start by choosing an arbitrary point in $\{x\in \Omega:d_\Omega(x) <\ell\}$; then, given a finite collection of points with pairwise distances $\geq 2\ell$, we can either find a new point in $\{x\in \Omega:d_\Omega(x) <\ell\}$ whose distance to all of the previous points is $\geq 2\ell$ or the collection satisfies that $\{x\in \Omega:d_\Omega(x) <\ell\}\subset \bigcup_k B_{2\ell}(x_k)$. Since $\{x\in \Omega:d_\Omega(x) <\ell\}$ is bounded, the algorithm terminates after a finite number of steps. 
    
    It follows from the Bishop--Gromov comparison theorem that the function
    \begin{equation*}
        r \mapsto \frac{r^d}{V_\Omega(x, r)}
    \end{equation*}
    is nondecreasing for each $x\in \R^d$ (see Lemma~\ref{lem: Bishop-Gromov monotonicity}). Together with the properties of the collection $\{x_k\}_{k =1}^M$ it follows that
 	\begin{align*}
 	 	\int_{\{x\in \Omega:d_\Omega(x) <\ell\}} \frac{\ell^{d}}{V_\Omega(x, \ell)}\,dx 
 	 	&\leq 
 	 	\sum_{k=1}^M \int_{\Omega \cap B_{2\ell}(x_k)}\frac{\ell^{d}}{V_\Omega(x, \ell)}\,dx\\
 	 	&\leq 
 	 	\sum_{k=1}^M \int_{\Omega \cap B_{2\ell}(x_k)}\frac{(4\ell)^{d}}{V_\Omega(x, 4\ell)}\,dx\,.
 	\end{align*} 
    For any $x \in B_{2\ell}(x_k)$, $\Omega \cap B_{2\ell}(x_k)\subset \Omega \cap B_{4\ell}(x)$ and so $V_\Omega(x, 4\ell) \geq V_\Omega(x_k, 2\ell)$. When inserted into the previous bound this yields
    \begin{align*}
 	 	\int_{\{x\in \Omega:d_\Omega(x) <\ell\}} \frac{\ell^{d}}{V_\Omega(x, \ell)}\,dx 
 	 	&\leq 
 	 	\sum_{k=1}^M \int_{\Omega \cap B_{2\ell}(x_k)}\frac{(4\ell)^{d}}{V_\Omega(x_k, 2\ell)}\,dx\\
        &=
         4^d M\ell^{d}\,.
 	\end{align*} 
    It thus remains to bound $M$. In order to do so, we write
 	 \begin{align*}
 	 	\ell^{d}M = \frac{|{\cup_{k=1}^M B_{\ell}(x_k)}|}{|B_1|} &\leq \frac{|\{x\in \Omega: \dist(x, \partial\Omega)<2\ell\}|}{|B_1|}\\
 	 	&\quad +\frac{|\{x\in \Omega^c: \dist(x, \Omega)<\ell\}|}{|B_1|}\,.
 	 \end{align*}
 	 We estimate the two terms on the right side separately.

 	 For the exterior term Proposition~\ref{prop: Minkowski sum bounds} implies
 	 \begin{align*}
 	 	|\{x\in \Omega^c:\dist(x, \Omega)<\ell\}| \lesssim_d \Haus^{d-1}(\partial\Omega)\ell \Bigl[1+ \Bigl(\frac{\ell}{r_{\rm in}(\Omega)}\Bigr)^{d-1}\Bigr]\,.
 	 \end{align*}

 	 For the interior term we bound
 	 \begin{align*}
 	 	|\{x\in \Omega: \dist(x, \partial\Omega)<2\ell\}|
 	 	&= \int_0^{2\ell} \Haus^{d-1}(\partial\{x \in \Omega: \dist(x, \partial\Omega)>\eta\})\,d\eta\\
 	 	&\leq 2\ell\Haus^{d-1}(\partial\Omega)\,.
 	 \end{align*}
 	 The last bound comes from the inequality $\Haus^{d-1}(\partial\{x \in \Omega: \dist(x, \partial\Omega)>\eta\}) \leq \Haus^{d-1}(\partial\Omega)$ (see Lemma~\ref{lem: inner parallel perimeter bounds}).
	
 	 Combining the above bounds proves \eqref{eq:convexintegratedv} and thereby completes the proof.
 \end{proof}

We note that the case $\gamma=0$ is not included in Theorems \ref{thm: Sharp Weyld}, \ref{thm: Sharp Weyln} and \ref{thm: Sharp Weylnc}. Although it is not needed for our later arguments, we state the following result for $\gamma=0$.

\begin{proposition}
	\label{prop: Quant Weyl}
	Let $d\geq 1$ and let $\Omega \subset \R^d$ be an open set with $|\Omega|<\infty$ and $\Theta_\Omega <\infty$.
	Then for all $\lambda> 0$,
	$$
	\bigl| \Tr(-\Delta_\Omega^{\rm D} -\lambda)_\limminus^0 - L_{0,d}^{\rm sc}  |\Omega| \lambda^{\frac{d}2}\bigr|
	\leq C_{d} \Theta_\Omega \lambda^{\frac{d-1}2} \bigl( 1+ \ln_\limplus\bigl(r_{\rm in}(\Omega) \sqrt{\lambda}\bigr) \bigr),
	$$
	where $C_{d}$ is a constant depending only on the dimension. 
    Moreover, if $\Omega$ additionally has the extension property, then for any $\lambda_0>0$ there exists a constant $C_{\Omega, \lambda_0}$ so that for all $\lambda > 0$,
	$$
	\bigl| \Tr(-\Delta_\Omega^{\rm N} -\lambda)_\limminus^0 - L_{0,d}^{\rm sc}  |\Omega| \lambda^{\frac{d}2}\bigr|
	\leq C_{\Omega,\lambda_0} \lambda^{\frac{d-1}2} \bigl(\max\bigl\{1, (\lambda_0/\lambda)^{\frac{d-1}{2}}\bigr\}+ \ln_\limplus\bigl(r_{\rm in}(\Omega) \sqrt{\lambda}\bigr)\bigr)\,.
	$$
\end{proposition}

Proposition~\ref{prop: Quant Weyl} can be proved by following the proofs of Theorems \ref{thm: Sharp Weyld} and \ref{thm: Sharp Weyln} given above. The only change in the argument is that one needs to take into account logarithms appearing upon integration. In this case, the error term obtained is not of the (conjectured) sharp order, but is only off by a logarithmic term. Essentially the same result was obtained by Courant \cite{Courant1920} using a completely different method, based on Dirichlet--Neumann bracketing. A detailed proof in the Dirichlet case appears in~\cite[Theorem 1.8]{netrusov_weyl_2005}; see also \cite[Corollary 3.14]{LTbook} and references therein. We have not found the result in the Neumann case in the literature, even though closely related results appear, for instance, in \cite{netrusov_weyl_2005}.

\begin{remark}\label{rem: WeylBerry}
    In the results above we have assumed that $\Theta_\Omega<\infty$ in order to obtain the order-sharp bounds that are needed for the applications to two-term asymptotics. However, the proof provided establishes Weyl's law for $\Tr(-\Delta_\Omega^\#-\lambda)_\limminus^\gamma$ under the assumptions of Proposition~\ref{prop: Quant Weyl} but in place of $\Theta_\Omega<\infty$ assuming that
\begin{equation*}
    \int_{1/\sqrt{\lambda}}^\infty \frac{|\{x\in \Omega: d_\Omega(x)<s\}|}{s^{2+\gamma}}\,ds =o(\lambda^{\frac{1+\gamma}{2}})\quad \mbox{as } \lambda \to \infty\,.
\end{equation*}
In particular, if $|\{x\in \Omega: d_\Omega(x)<s\}| \sim s^\beta$ for some $0<\beta<1$ the proof reproduces the order of error predicted by the modified Weyl--Berry conjecture; see, e.g., \cite{Lapidus_91}.
\end{remark}


\section{Two-term asymptotics for Riesz means with \texorpdfstring{$\gamma\geq 1$}{gamma>=1}}\label{sec:mainprooflarge}

In this and the following section we will prove Theorems~\ref{thm: Weyl asymptotics Lipschitz} and \ref{thm: Weyl asymptotics convex}. We will distinguish the cases $\gamma\geq 1$ and $\gamma<1$, treating the former case in this section and the latter one in the next. As the main point of our paper is to treat arbitrarily small $\gamma$, the present section is in some sense a preliminary step towards this goal. In fact, the asymptotics for $\gamma=1$ will be one ingredient in proving the asymptotics for $0<\gamma<1$ in the next section. The same applies to the non-asymptotic bounds in the case of convex sets.

The $\gamma\geq 1$ parts of Theorems~\ref{thm: Weyl asymptotics Lipschitz} and \ref{thm: Weyl asymptotics convex} concerning the Dirichlet case are essentially contained in our previous work \cite{FrankLarson_Crelle20}. The results in the Neumann case, however, seem to be new. We have devised an argument that allows us to deduce the asymptotics in the Neumann case in Theorem \ref{thm: Weyl asymptotics Lipschitz} from those in the Dirichlet case via (known) heat trace asymptotics and a (known) Tauberian theorem. A similar argument works for the non-asymptotic bound in Theorem \ref{thm: Weyl asymptotics convex}, except that now neither the non-asymptotic heat trace bound nor the non-asymptotic Tauberian theorem seemed to have been known. We will prove the latter in Part 2 of this paper while the former is proved in the companion paper \cite{FrankLarson_NeumannHeat2025}.

The two subsections in this section are devoted to Theorems~\ref{thm: Weyl asymptotics Lipschitz} and \ref{thm: Weyl asymptotics convex}, respectively, in the case $\gamma\geq 1$.


\subsection{Proof of Theorem \ref{thm: Weyl asymptotics Lipschitz} for \texorpdfstring{$\gamma\geq 1$}{gamma>=1}}

We begin with the Dirichlet case.

\begin{proof}[Proof of Theorem \ref{thm: Weyl asymptotics Lipschitz} for $\gamma\geq 1$, $\#={\rm D}$]
    The assertion of the theorem for $\gamma=1$ is proved in \cite{FrankLarson_Crelle20}. The assertion for $\gamma>1$ can be easily deduced from that in the case $\gamma=1$. Indeed, for any lower semibounded operator $H$, any constant $\lambda$ and any $\gamma>1$ we have
    \begin{equation}\label{eq: Aizenman-Lieb identity}
    \Tr (H-\lambda)_\limminus^\gamma = \gamma(\gamma-1) \int_0^\infty \tau^{\gamma-2} \Tr(H-\lambda+\tau)_\limminus\,d\tau \,.
    \end{equation}
    We apply this formula with $H=-\Delta_\Omega^{\rm D}$. In view of the formula
    \begin{align*}
        L_{\gamma,d}^{\rm sc} \lambda^{\gamma+\frac{d}2} & = \int_{\R^d} (|\xi|^2-\lambda)_\limminus^\gamma \,\frac{d\xi}{(2\pi)^{\frac{d}2}} = \gamma(\gamma-1) \int_0^\infty \tau^{\gamma-2} \int_{\R^d} (|\xi|^2 - \lambda + \tau)_\limminus \,\frac{d\xi}{(2\pi)^d}\,d\tau \\
        & = \gamma(\gamma-1) \int_0^\infty \tau^{\gamma-2} L_{1,d}^{\rm sc} (\lambda-\tau)_\limplus^{1+\frac{d}2} \,d\tau \,,
    \end{align*}
    the asymptotics for $\gamma>1$ follow from those for $\gamma=1$ by a simple limiting argument. For the details of a similar argument (where asymptotics for $\gamma>0$ are deduced from asymptotics for $\gamma=0$) see \cite[Corollary 3.17]{LTbook}. This completes the proof of Theorem~\ref{thm: Weyl asymptotics Lipschitz} for $\# = {\rm D}$ and $\gamma \geq 1$.
\end{proof}

We now turn to the proof in the Neumann case. The basic strategy will be to deduce the asymptotics in the Neumann case from those in the Dirichlet case. To do this, we combine asymptotics for the traces of the associated heat kernels with a Tauberian theorem. The following theorem of Brown \cite{Brown93} provides the relevant asymptotics for the heat trace. 

\begin{theorem}[{\cite{Brown93}}]\label{thm: Weyl asymptotics Lipschitz heat}
	Let $d\geq 2$ and let $\Omega\subset\R^d$ be a bounded open set with Lipschitz boundary. Then, as $t\to 0$,
	\begin{equation*}
		\Tr (e^{t\Delta_\Omega^{\rm D}}) = \frac{1}{(4\pi t)^{\frac d2}}\biggl(|\Omega| - \frac{\sqrt{\pi t}}{2} \mathcal H^{d-1}(\partial\Omega) + o(\sqrt{t})\biggr)
	\end{equation*}
	and
	\begin{equation*}
		\Tr (e^{t\Delta_\Omega^{\rm N}}) = \frac{1}{(4\pi t)^{\frac d2}}\biggl(|\Omega| + \frac{\sqrt{\pi t}}{2} \mathcal H^{d-1}(\partial\Omega) + o(\sqrt{t})\biggr) \,.
	\end{equation*}
\end{theorem}

\begin{proof}[Proof of Theorem \ref{thm: Weyl asymptotics Lipschitz} for $\gamma\geq 1$, $\#={\rm N}$]
	We consider the function
	$$
	f(\lambda) := \Tr(-\Delta_\Omega^{\rm N}-\lambda)_\limminus - \Tr(-\Delta_\Omega^{\rm D}-\lambda)_\limminus \,.
	$$
	We claim that $f$ is nondecreasing. Indeed, since for any lower semibounded operator $H$ we have $\Tr(H-\lambda)_\limminus = \int_{-\infty}^\lambda \Tr(H-\mu)_\limminus^0\,d\mu$, we have
    $$
	f(\lambda) = \int_{-\infty}^\lambda \left( \Tr(-\Delta_\Omega^{\rm N}-\mu)_\limminus^0 - \Tr(-\Delta_\Omega^{\rm D}-\mu)_\limminus^0 \right)d\mu \,.
	$$
	In particular, when $\lambda_1\leq \lambda_2$, then
	$$
	f(\lambda_2) - f(\lambda_1) = \int_{\lambda_1}^{\lambda_2} \left( \Tr(-\Delta_\Omega^{\rm N}-\mu)_\limminus^0 - \Tr(-\Delta_\Omega^{\rm D}-\mu)_\limminus^0 \right)d\mu \,.
	$$
	Since, by the variational principle, $\Tr(-\Delta_\Omega^{\rm N}-\mu)_\limminus^0 - \Tr(-\Delta_\Omega^{\rm D}-\mu)_\limminus^0\geq 0$ for any $\mu$, we deduce the claimed monotonicity.
	
	We note that
	$$
	\int_0^\infty e^{-t\lambda} df(\lambda) = t^{-1} \bigl( \Tr(e^{t\Delta_\Omega^{\rm N}}) - \Tr(e^{t\Delta_\Omega^{\rm D}}) \bigr).
	$$
	Thus, Theorem \ref{thm: Weyl asymptotics Lipschitz heat} implies that
	$$
	\int_0^\infty e^{-t\lambda} df(\lambda) = \frac{1}{2t (4\pi t)^{\frac{d-1}2}} \mathcal H^{d-1}(\partial\Omega)  (1 + o(1))
	\qquad\text{as}\ t\to 0 \,.
	$$
	It follows from the standard Tauberian theorem (see, e.g., \cite[Theorem 10.3]{Simon_FunctionalIntegrationBook} or \cite[Theorem VII.3.2]{Korevaar_TauberianTheory_book}) that
	$$
	f(\lambda) = \frac{1}{2(4\pi)^{\frac{d-1}{2}}\Gamma\bigl(\frac{d+1}{2}+1\bigr)}  \mathcal H^{d-1}(\partial\Omega) \lambda^{\frac{d+1}{2}}(1 + o(1))
	\qquad\text{as}\ \lambda\to\infty \,.
	$$
	We note that $(4\pi)^{-\frac{d-1}{2}}\Gamma(\frac{d+1}{2}+1)^{-1} = L_{1,d-1}^{\rm sc}$. Since the asymptotics in Theorem \ref{thm: Weyl asymptotics Lipschitz} have already been proved in the Dirichlet case for $\gamma=1$, we obtain the corresponding asymptotics in the Neumann case and $\gamma=1$.

    The asymptotics for $\gamma>1$ are deduced from those for $\gamma=1$ by the same argument as in the Dirichlet case. This completes the proof of Theorem~\ref{thm: Weyl asymptotics Lipschitz} for $\# = {\rm N}$ and $\gamma \geq 1$.
\end{proof}


\subsection{Proof of Theorem \ref{thm: Weyl asymptotics convex} for \texorpdfstring{$\gamma\geq 1$}{gamma>=1}}

In this subsection we turn our attention to non-asymptotic bounds when the underlying domain is convex.

\begin{proof}[Proof of Theorem \ref{thm: Weyl asymptotics convex} for $\gamma\geq 1$, $\#={\rm D}$]
    The assertion of the theorem for $\gamma=1$ is proved in~\cite[Theorem 1.2]{FrankLarson_Crelle20}. The bounds for $\gamma>1$ are deduced from that for $\gamma=1$ by means of the integral identity in~\eqref{eq: Aizenman-Lieb identity}. This completes the proof of Theorem~\ref{thm: Weyl asymptotics convex} for $\# = {\rm D}$ and $\gamma \geq 1$.
\end{proof}

It remains to prove Theorem~\ref{thm: Weyl asymptotics convex} with $\# = {\rm N}, \gamma\geq 1$. To this end we follow the same strategy as applied above for the asymptotics in the corresponding case of Theorem~\ref{thm: Weyl asymptotics Lipschitz}. However, each ingredient in the proof is replaced by a quantified version. The Tauberian result applied above is replaced by a Tauberian theorem with a remainder estimate, which is a non-asymptotic version of a result of Ganelius~\cite{Ganelius54} and which will be proved in Section \ref{sec: Laplace Tauber}. Likewise, the heat trace asymptotics of Theorem~\ref{thm: Weyl asymptotics Lipschitz heat} are replaced by the following estimates. If $\Omega \subset \R^d, d\geq 2,$ is a bounded convex set then, for all $t>0$,
\begin{equation}\label{eq: heat trace bounds}
\begin{aligned}
	\biggl|(4\pi t)^{\frac{d}{2}}\Tr(e^{t\Delta_\Omega^{\rm D}})- |\Omega| + \frac{\sqrt{\pi t}}{2}\Haus^{d-1}(\partial\Omega)\biggr| &\!\lesssim_d \!\Haus^{d-1}(\partial\Omega) \sqrt{t}\Bigl(\frac{\sqrt{t}}{r_{\rm in}(\Omega)}\Bigr)^{\frac{1}{11}}\,,\\
	\biggl|(4\pi t)^{\frac{d}{2}}\Tr(e^{t\Delta_\Omega^{\rm N}})-|\Omega| - \frac{\sqrt{\pi t}}{2}\Haus^{d-1}(\partial\Omega)\biggr| &\!\lesssim_d \!\Haus^{d-1}(\partial\Omega) \sqrt{t}\biggl[\Bigl(\frac{\sqrt{t}}{r_{\rm in}(\Omega)}\Bigr)^{\frac{1}{11}}+ \Bigl(\frac{\sqrt{t}}{r_{\rm in}(\Omega)}\Bigr)^{d-1}\biggr].
\end{aligned}
\end{equation}
We emphasize that the implicit constants in these bounds can be chosen depending only on $d$. The bound in~\eqref{eq: heat trace bounds} for the Dirichlet case follows from the bound in~\cite[Theorem 1.2]{FrankLarson_Crelle20} by means of the formula
\begin{equation*}
    \Tr(e^{-tH}) = t^2 \int_{\R} e^{-\lambda t} \Tr(H-\lambda)_\limminus\,d\lambda \,,
\end{equation*}
valid for any semibounded operator $H$ and applied here with $H=-\Delta_\Omega^{\rm D}$. The bound in~\eqref{eq: heat trace bounds} for the Neumann setting is proved in \cite{FrankLarson_NeumannHeat2025}. In passing we note that a combination of our arguments in \cite{FrankLarson_NeumannHeat2025} and those in~\cite[Theorem 1.2]{FrankLarson_Crelle20} gives an alternative proof of \eqref{eq: heat trace bounds} in the Dirichlet case, which is more direct than going via Riesz means \cite[Theorem 1.2]{FrankLarson_Crelle20} and then integrating. Probably this approach would also give a better exponent than $1/11$ in the error term.

\begin{proof}[Proof of Theorem \ref{thm: Weyl asymptotics convex} for $\gamma\geq 1$, $\#={\rm N}$]
	Consider the function
	$$
	g(\lambda) := \Tr(-\Delta_\Omega^{\rm N}-\lambda)_\limminus - \Tr(-\Delta_\Omega^{\rm D}-\lambda)_\limminus - \frac{1}{2}L_{1,d-1}^{\rm sc}\Haus^{d-1}(\partial\Omega) \lambda_\limplus^{1+ \frac{d-1}{2}}\,.
	$$
	In the previous subsection the function $\lambda \mapsto g(\lambda) + \frac{1}{2}L_{1,d-1}^{\rm sc}\Haus^{d-1}(\partial\Omega) \lambda_\limplus^{1+ \frac{d-1}{2}}$ was shown to be nondecreasing. Consequently, the Borel measure $\tilde \mu$ on $[0,\infty)$ defined by
	\begin{equation*}
		\tilde \mu(\omega) := \int_\omega dg(\lambda) + \frac{(d+1)}{4}L_{1,d-1}^{\rm sc}\Haus^{d-1}(\partial\Omega)\int_\omega \lambda^{\frac{d-1}{2}}\,d\lambda
	\end{equation*}
	is nonnegative.
 
    We note that
	$$
	\int_0^\infty e^{-t\lambda} dg(\lambda) = t^{-1} \Bigl( \Tr( e^{t\Delta_\Omega^{\rm N}}) - \Tr (e^{t\Delta_\Omega^{\rm D}}) - (4\pi t)^{- \frac{d}{2}}\sqrt{\pi t}\,\Haus^{d-1}(\partial\Omega)\Bigr)\,,
	$$
	and for $\gamma \geq 1$
	\begin{equation}\label{eq: Riesz rep 1}
    \begin{aligned}
	   \int_0^\lambda (1-v/\lambda)^{\gamma-1}\,dg(v) 
        &= \frac{\lambda^{1-\gamma}}{\gamma} \biggl(\Tr(-\Delta_\Omega^{\rm N}-\lambda)_\limminus^\gamma- \Tr(-\Delta_\Omega^{\rm D}-\lambda)_\limminus^\gamma\\
        &\quad - \frac{1}{2}L_{\gamma,d-1}^{\rm sc}\Haus^{d-1}(\partial\Omega)\lambda^{\gamma+ \frac{d-1}{2}}\biggr)\,.
    \end{aligned}
	\end{equation}

	By~\eqref{eq: heat trace bounds} it holds that
	$$
	\biggl|\int_0^\infty e^{-t\lambda} dg(\lambda) \biggr| \lesssim_d  \frac{\Haus^{d-1}(\partial\Omega)}{t^{1+\frac{d-1}{2}}}\biggl[\biggl(\frac{\sqrt{t}}{r_{\rm in}(\Omega)}\biggr)^{\frac{1}{11}}+\biggl(\frac{\sqrt{t}}{r_{\rm in}(\Omega)}\biggr)^{d-1}\biggr]
	\qquad\text{for all }\ t>0 \,.
	$$

	The Tauberian remainder theorem in Proposition~\ref{prop: Tauberian v2}, applied with $\nu=\frac{d+1}2\geq 1$, implies that there are constants $B, k_0$ depending only on $d, \gamma$ such that for all $k \geq k_0$ and $\lambda >0$
	\begin{align*}
		\Bigl|\int_0^\lambda &(1-v/\lambda)^{\gamma-1}dg(v)\Bigr|\\ 
		&\lesssim_{\gamma, d} \Haus^{d-1}(\partial\Omega) B^k \lambda^{1+ \frac{d-1}{2}} \max_{j=1, \ldots, k}\biggl[j^{-1-\frac{d-1}{2}+ \frac{1}{22}}\bigl(r_{\rm in}(\Omega)\sqrt{\lambda}\bigr)^{-\frac{1}{11}}+j^{-1}\bigl(r_{\rm in}(\Omega)\sqrt{\lambda}\bigr)^{1-d}\biggr]\\
		&\quad + \frac{1}{k^{\gamma}} \Haus^{d-1}(\partial\Omega) \lambda^{1+ \frac{d-1}{2}}\\
		&= \Haus^{d-1}(\partial\Omega) B^k \lambda^{1+ \frac{d-1}{2}} \biggl[\bigl(r_{\rm in}(\Omega)\sqrt{\lambda}\bigr)^{-\frac{1}{11}}+\bigl(r_{\rm in}(\Omega)\sqrt{\lambda}\bigr)^{1-d}\biggr]\\
		&\quad + \frac{1}{k^\gamma} \Haus^{d-1}(\partial\Omega) \lambda^{1+ \frac{d-1}{2}}\,.
	\end{align*}
	Without loss of generality we may assume that $B\geq 2$.

	For $c'>0$ to be chosen we set $$k = \max\bigl\{ k_0, \bigl\lfloor c' \ln\bigl(r_{\rm in}(\Omega)\sqrt{\lambda}\bigr)/\ln(B)\bigr\rfloor \bigr\}.$$ Then if $\lambda \leq r_{\rm in}(\Omega)^{-2}B^{\frac{2k_0}{c'}}$, we find that 
	$$
		\Bigl|\int_0^\lambda (1-v/\lambda)^{\gamma-1}dg(v)\Bigr| \lesssim_{\gamma, d, c'} \frac{\Haus^{d-1}(\partial\Omega)}{r_{\rm in}(\Omega)^{d-1}}\lambda\,.
	$$ 
	If instead $\lambda > r_{\rm in}(\Omega)^{-2}B^{\frac{2k_0}{c'}}$ and we choose $c'< \frac{1}{11}$,
    \begin{align*}
		\Bigl|\int_0^\lambda (1-v/\lambda)^{\gamma-1}dg(v)\Bigr| 
		&\lesssim_{\gamma, d}\Haus^{d-1}(\partial\Omega) \lambda^{1+ \frac{d-1}{2}} \biggl[\bigl(r_{\rm in}(\Omega)\sqrt{\lambda}\bigr)^{-\frac{1}{11}+c'}+\bigl(r_{\rm in}(\Omega)\sqrt{\lambda}\bigr)^{1-d+c'}\biggr]\\
		&\quad + \frac{1}{\bigl(c'\ln\bigl(r_{\rm in}(\Omega)\sqrt{\lambda}\bigr)\bigr)^\gamma} \Haus^{d-1}(\partial\Omega)\lambda^{1+ \frac{d-1}{2}}\\
		&\lesssim_{\gamma, d, c'}\Haus^{d-1}(\partial\Omega)\lambda^{1+ \frac{d-1}{2}}\bigl(1+\ln_\limplus\bigl(r_{\rm in}(\Omega)\sqrt{\lambda}\bigr)\bigr)^{-\gamma}\,.
	\end{align*}
	In conclusion, for any $d\geq 2, \gamma \geq 1$ and all $\lambda>0$,
	\begin{equation}\label{eq: final bound difference of traces 2}
		\Bigl|\int_0^\lambda \! (1-v/\lambda)^{\gamma-1}dg(v)\Bigr|
		\!\lesssim_{\gamma,d} \! \Haus^{d-1}(\partial\Omega) \lambda^{1+ \frac{d-1}{2}} \biggl[\!\bigl(1+\ln_\limplus\bigl(\sqrt{\lambda}{r_{\rm in}(\Omega)}\bigr)\bigr)^{-\gamma}\!+\bigl(r_{\rm in}(\Omega)\sqrt{\lambda}\bigr)^{1-d}\biggr].
	\end{equation}

	Consequently, by~\eqref{eq: Riesz rep 1},~\eqref{eq: final bound difference of traces 2}, and the quantitative two-term asymptotic expansion of $\Tr(-\Delta_\Omega^{\rm D}-\lambda)_\limminus^\gamma$ obtained from~\cite[Theorem 1.2]{FrankLarson_Crelle20} combined with~\eqref{eq: Aizenman-Lieb identity},
	\begin{align*}
		\biggl|\Tr(&-\Delta_\Omega^{\rm N}-\lambda)_\limminus^\gamma - L_{\gamma,d}^{\rm sc}|\Omega|\lambda^{\gamma+ \frac{d}{2}}- \frac{1}{4}L_{\gamma, d-1}^{\rm sc}\Haus^{d-1}(\partial\Omega)\lambda^{\gamma+\frac{d-1}{2}}\biggr|\\
		&=
		\Biggl|\gamma \lambda^{\gamma-1}\int_0^\lambda (1-v/\lambda)^{\gamma-1}\,dg(v)+ \Tr(-\Delta_\Omega^{\rm D}-\lambda)_\limminus^\gamma - L_{\gamma,d}^{\rm sc}|\Omega|\lambda^{\gamma+ \frac{d}{2}}\\
		&\qquad+ \frac{1}{4}L_{\gamma, d-1}^{\rm sc}\Haus^{d-1}(\partial\Omega)\lambda^{\gamma+\frac{d-1}{2}}\Biggr|\\
		&\leq \biggl|\gamma \lambda^{\gamma-1}\int_0^\lambda (1-v/\lambda)^{\gamma-1}\,dg(v)\biggr| \\
		&\quad + \biggl|\Tr(-\Delta_\Omega^{\rm D}-\lambda)_\limminus^\gamma - L_{\gamma,d}^{\rm sc}|\Omega|\lambda^{\gamma+ \frac{d}{2}}+ \frac{1}{4}L_{\gamma, d-1}^{\rm sc}\Haus^{d-1}(\partial\Omega)\lambda^{\gamma+ \frac{d-1}{2}}\biggr|\\
		&\lesssim_{\gamma, d}\! \Haus^{d-1}(\partial\Omega) \lambda^{\gamma+ \frac{d-1}{2}} \biggl[\bigl(1+\ln_\limplus\bigl(\sqrt{\lambda}{r_{\rm in}(\Omega)}\bigr)\bigr)^{-\gamma}\!+\bigl(r_{\rm in}(\Omega)\sqrt{\lambda}\bigr)^{1-d} + \bigl(r_{\rm in}(\Omega)\sqrt{\lambda}\bigr)^{-\frac{1}{11}}\biggr]\\
		&\lesssim_{\gamma, d} \!\Haus^{d-1}(\partial\Omega) \lambda^{\gamma+ \frac{d-1}{2}} \biggl[\bigl(1+\ln_\limplus\bigl(\sqrt{\lambda}{r_{\rm in}(\Omega)}\bigr)\bigr)^{-\gamma}\!+\bigl(r_{\rm in}(\Omega)\sqrt{\lambda}\bigr)^{1-d}\biggr]\,.
	\end{align*}
	This completes the proof of Theorem~\ref{thm: Weyl asymptotics convex} for $\gamma\geq 1$ and $\# = {\rm N}$.
\end{proof}


\section{Two-term asymptotics for Riesz means with \texorpdfstring{$0<\gamma<1$}{0<gamma<1}}
\label{sec:mainproofsmall}

In this section we will complete the proofs of Theorems~\ref{thm: Weyl asymptotics Lipschitz} and ~\ref{thm: Weyl asymptotics convex} by considering the range $0<\gamma<1$.

Our proof of the first theorem will rely on two main ingredients:
\begin{enumerate}
	\item The one-term asymptotics
	\begin{equation}
		\label{eq:seeley}
		\Tr(-\Delta_\Omega^\# -\lambda)_\limminus^\gamma = L_{\gamma,d}^{\rm sc} |\Omega|\lambda^{\gamma+\frac d2} + O(\lambda^{\gamma+\frac{d-1}{2}})
	\end{equation}
	with order-sharp remainder, which are valid for arbitrary small $\gamma>0$; see Theorems \ref{thm: Sharp Weyld} and \ref{thm: Sharp Weyln}.
	\item The two-term asymptotics
	\begin{equation}
		\label{eq:fl}
		\Tr(-\Delta_\Omega^\# -\lambda)_\limminus = L_{1,d}^{\rm sc} |\Omega|\lambda^{1+ \frac d2} \mp \frac14 L_{\gamma,d-1}^{\rm sc} \mathcal H^{d-1}(\partial\Omega) \lambda^{1+\frac{d-1}{2}} + o(\lambda^{1+\frac{d-1}{2}}) \,,
	\end{equation}
	valid if $\Omega$ is open, bounded, and has Lipschitz regular boundary. This is a special case of Theorem \ref{thm: Weyl asymptotics Lipschitz} that was proved in the previous subsection.
\end{enumerate}
The important tool that we will use in order to tie these two ingredients together is a convexity result for Riesz means, due to Riesz, which we will discuss in Section \ref{sec: Riesz convexity}.

To prove Theorem~\ref{thm: Weyl asymptotics convex} we follow the same overall idea. However, now we rely in addition on quantitative estimates for the error terms in the asymptotic expansions~\eqref{eq:seeley} and~\eqref{eq:fl}.

\begin{proof}[{Proof of Theorem \ref{thm: Weyl asymptotics Lipschitz} for $0<\gamma<1$}]
	We set $\tau_{\rm D}=-1, \tau_{\rm N}=1$ and
	$$
	f(\lambda) := L_{0,d}^{\rm sc} |\Omega|\lambda^{\frac d2} + \frac{\tau_\#}4 L_{0,d-1}^{\rm sc} \mathcal H^{d-1}(\partial\Omega) \lambda^{\frac{d-1}{2}} 
	- \Tr(-\Delta_\Omega^\# -\lambda)_\limminus^0\,.
	$$
	Set $f^{(0)} :=f$ and for $\kappa >0$ define $f^{(\kappa)}$ by
	\begin{equation*}
		f^{(\kappa)}(\lambda) := \frac{1}{\Gamma(\kappa)}\int_0^\lambda (\lambda-\mu)^{\kappa-1}f(\mu)\,d\mu\,.
	\end{equation*}
	Note that
	$$
	f^{(\gamma)}(\lambda) = \frac{1}{\Gamma(\gamma+1)} \Bigl(L_{\gamma,d}^{\rm sc} |\Omega|\lambda^{\gamma+\frac d2} + \frac{\tau_\#}4 L_{\gamma,d-1}^{\rm sc} \mathcal H^{d-1}(\partial\Omega) \lambda^{\gamma+\frac{d-1}{2}} 
	-  \Tr(-\Delta_\Omega -\lambda)_\limminus^\gamma\Bigr) \,.
	$$
	Therefore, the claim follows if we show that $f^{(\gamma)}(\lambda)= o(\lambda^{\gamma+\frac{d-1}{2}})$.
	
	We fix $0<\kappa<\gamma$ and apply Proposition \ref{prop: Riesz logconvexity} to the function $f^{(\kappa)}$, $\gamma$ replaced by $1-\kappa$, and $\sigma=\gamma-\kappa$. By the semigroup property of Riesz means (see~\eqref{eq: semigroup Riesz prop} below), we then obtain
	\begin{equation}
	    \label{eq:rieszinterpolapplied}
     |f^{(\gamma)}(\lambda)| \leq C \biggl( \sup_{\Lambda\in [0, \lambda]} |f^{(\kappa)}(\Lambda)| \biggr)^{1-\frac{\gamma-\kappa}{1-\kappa}} \biggl( \sup_{\Lambda\in [0, \lambda]} |f^{(1)}(\Lambda)| \biggr)^{\frac{\gamma-\kappa}{1-\kappa}}.
	\end{equation}
	
	By Theorems~\ref{thm: Sharp Weyld} and \ref{thm: Sharp Weyln}, we have
	$$
	\sup_{\Lambda\leq\lambda} |f^{(\kappa)}(\Lambda)| = O( \lambda^{\kappa+\frac{d-1}{2}}) \,,
	$$
    (Here, in the Neumann case we use the well known fact that bounded Lipschitz domains have the extension property; see, e.g., \cite[Theorem 2.92]{LTbook}.) By the $\gamma=1$-case of Theorem \ref{thm: Weyl asymptotics Lipschitz} proved in the previous section, we have
	$$
	\sup_{\Lambda\leq\lambda} |f^{(1)}(\Lambda)| = o(\lambda^{1+\frac{d-1}{2}}) \,.
	$$
	The assertion now follows from the fact that
	$$
	\biggl( 1 - \frac{\gamma-\kappa}{1-\kappa}\biggr) \biggl( \kappa+\frac{d-1}{2} \biggr) + \frac{\gamma-\kappa}{1-\kappa} \biggl( 1 +\frac{d-1}{2} \biggr) = \gamma +\frac{d-1}{2} \,.
	$$
	This completes the proof.
\end{proof}

The proof of Theorem~\ref{thm: Weyl asymptotics convex} is almost identical to that of Theorem~\ref{thm: Weyl asymptotics Lipschitz} presented above, but in each step we rely on quantitative asymptotic inequalities instead of the non-quantitative versions used above.

\begin{proof}[{Proof of Theorem \ref{thm: Weyl asymptotics convex} for $0<\gamma<1$}]
	We define $f, f^{(\kappa)}$ as in the previous proof and recall inequality \eqref{eq:rieszinterpolapplied}. We need to bound the two suprema on the right side.

    We begin with the bounds for $f^{(\kappa)}$ with $\kappa>0$. In the Dirichlet case we have
	$$
	\sup_{\Lambda\leq\lambda} |f^{(\kappa)}(\Lambda)| \leq C_{\gamma, d} \Haus^{d-1}(\partial\Omega) \lambda^{\kappa + \frac{d-1}{2}}\,.
	$$
    This follows from Theorem \ref{thm: Sharp Weyld} since the convexity assumption implies $\Theta_\Omega \leq \Haus^{d-1}(\partial\Omega)$; see Lemma \ref{lem: volume bdry neighbourhood bound}. In the Neumann case we have
    $$
        \sup_{\Lambda\leq\lambda} |f^{(\kappa)}(\Lambda)| \leq  
	C_{\gamma,d} \mathcal H^{d-1}(\partial\Omega) \Bigl( \lambda^{\kappa+\frac{d-1}2} + \lambda^{\kappa} \,r_{\rm in}(\Omega)^{1-d} \Bigr) \,.
    $$
	This follows from Theorem~\ref{thm: Sharp Weylnc}.

    Next, we turn to the bounds for $f^{(1)}$. In the Dirichlet case we have
	$$
	\sup_{\Lambda\leq\lambda} |f^{(1)}(\Lambda)| \leq C_{\gamma, d}\Haus^{d-1}(\partial\Omega) \lambda^{1+\frac{d-1}{2}}\bigl(r_{\rm in}(\Omega)\sqrt{\lambda}\bigr)^{-\frac{1}{11}}
	$$
    and in the Neumann case we have
    $$
    \sup_{\Lambda\leq\lambda} |f^{(1)}(\Lambda)| \leq C_{\gamma, d}\Haus^{d-1}(\partial\Omega) \lambda^{1+\frac{d-1}{2}}\Bigl[\bigl(1+\ln_\limplus\bigl(r_{\rm in}(\Omega)\sqrt{\lambda}\bigr)\bigr)^{-1}+ \bigl(r_{\rm in}(\Omega)\sqrt{\lambda}\bigr)^{1-d}\Bigr] \,.
    $$
    Both bounds follow from the $\gamma=1$ case of Theorem \ref{thm: Weyl asymptotics convex}, which has already been proved in the previous section.
 
	The assertion of Theorem \ref{thm: Weyl asymptotics convex} now follows as in the previous proof, together with the fact that $(0, \gamma) \ni \kappa \mapsto \frac{\gamma-\kappa}{1-\kappa}$ is a continuous decreasing function with $\lim_{\kappa \to 0^\limplus} \frac{\gamma-\kappa}{1-\kappa}=\gamma$ and $\lim_{\kappa \to \gamma^\limminus} \frac{\gamma-\kappa}{1-\kappa}=0$.
\end{proof}


\section{Variations on the theme}\label{sec:variations}

The method that we have applied in the previous section to prove Theorems \ref{thm: Weyl asymptotics Lipschitz} and \ref{thm: Weyl asymptotics convex} for $\gamma<1$ is rather general and can be used in other situations as well. In this section we sketch four further sample applications.


\subsection{Two-term Riesz means asymptotics from two-term asymptotics for the heat trace}

In this subsection we discuss a closely related but slightly different route to get around the poor performance of Tauberian theorems for the Laplace transform in regards to two-term asymptotics. The upshot of this method is that one avoids comparing the Neumann and Dirichlet Riesz means as in our proof of Theorems~\ref{thm: Weyl asymptotics Lipschitz} and \ref{thm: Weyl asymptotics convex}, and uses only information for the operator of interest as input. The downside is that the proof relies more crucially on the fact that we are considering the Dirichlet or Neumann Laplace operators, and generalizing this approach to a more general context appears difficult.

\begin{theorem}\label{thm: KrogerBerezin extrapolation}
    Let $d\geq 2$, $\#\in \{\rm D, \rm N\}$. Let $\Omega\subset \R^d$ be an open set satisfying that, as $t \to 0$,
    \begin{equation}\label{eq: two-term heat Kroger Berezin}
        \Tr(e^{t\Delta_\Omega^{\#}}) = \frac{1}{(4\pi t)^{\frac{d}2}}\biggl(|\Omega| +\tau_\# \frac{\sqrt{\pi t}}{2}\Haus^{d-1}(\partial\Omega)+o(\sqrt{t})\biggr) \,.
    \end{equation}
    Then, for all $\gamma \geq 2$,
    \begin{equation}\label{eq: two-term Riesz Kroger Berezin}
        \Tr(-\Delta_\Omega^{\#}-\lambda)_\limminus^\gamma 
        =
        L_{\gamma,d}^{\rm sc}|\Omega|\lambda^{\gamma + \frac{d}2}+ \frac{\tau_\#}4 L_{\gamma,d-1}^{\rm sc}\Haus^{d-1}(\partial\Omega)\lambda^{\gamma+ \frac{d-1}{2}} +o(\lambda^{\gamma+ \frac{d-1}{2}}) \,,
    \end{equation}
    as $\lambda \to \infty$. 
    
    If, in addition, there is a $\gamma_0 \in [0, 2)$ so that
    \begin{equation*}
        \Tr(-\Delta_\Omega^{\#}-\lambda)_\limminus^{\gamma_0} = L_{\gamma_0,d}^{\rm sc}|\Omega|\lambda^{\gamma_0 + \frac{d}2}+ O(\lambda^{\gamma_0+ \frac{d-1}{2}}) \,,
    \end{equation*}
    as $\lambda \to \infty$, then \eqref{eq: two-term Riesz Kroger Berezin} holds for all $\gamma >\gamma_0$.
\end{theorem}

\begin{proof} 
    By the Berezin--Li--Yau \cite{Berezin,LiYau_83} (for $\#={\rm D}$) and Kr\"oger's inequality \cite{Kroger92} (for $\#={\rm N}$) it holds that, for all $\lambda \geq 0$,
    \begin{equation*}
        \Tr(-\Delta_\Omega^{\rm D}-\lambda)_\limminus \leq L_{1,d}^{\rm sc}|\Omega|\lambda^{1+\frac{d}2} \leq \Tr(-\Delta_\Omega^{\rm N}-\lambda)_\limminus\,.
    \end{equation*}
    It thus follows from \eqref{eq: Aizenman-Lieb identity} that 
    \begin{equation*}
        f(\lambda):= \Tr(-\Delta_\Omega^{\#}-\lambda)^2_\limminus-L_{2,d}^{\rm sc}|\Omega|\lambda^{2+\frac{d}2}
    \end{equation*}
    is monotone (increasing for $\#={\rm N}$, decreasing for $\#={\rm D}$).

    By \eqref{eq: two-term heat Kroger Berezin},
    \begin{equation*}
        \int_0^\infty e^{-t\lambda}\,df(\lambda)
        =
        \frac{2}{t^{2}}\Bigl(\Tr(e^{t\Delta_\Omega^\#})- (4\pi t)^{-\frac{d}2}|\Omega|\Bigr) 
        =
        \frac{\tau_\#}{2t^{2}(4\pi t)^{\frac{d-1}2}}\Haus^{d-1}(\partial\Omega)(1+o(1))\,.
    \end{equation*}
    The standard Tauberian theorem (see, e.g., \cite[Theorem 10.3]{Simon_FunctionalIntegrationBook} or \cite[Theorem VII.3.2]{Korevaar_TauberianTheory_book}) implies that
	$$
	f(\lambda) =  \frac{\tau_\#}{4}L_{2,d-1}^{\rm sc}\mathcal H^{d-1}(\partial\Omega) \lambda^{2+\frac{d-1}{2}}(1 + o(1))
	\qquad\text{as}\ \lambda\to\infty \,.
	$$
    This proves \eqref{eq: two-term Riesz Kroger Berezin} for $\gamma=2$. The claim for $\gamma >2$ follows from the fact that, for all $\lambda \geq 0$ and $\gamma>2$,
    \begin{equation*}
        \Tr(-\Delta_\Omega^\#-\lambda)_\limminus^\gamma = \frac{2}{\gamma(\gamma-1)(\gamma-2)} \int_0^\infty \tau^{\gamma-3}\Tr(-\Delta_\Omega^\#-\lambda +\tau)_\limminus^2\,d\tau\,.
    \end{equation*}

    To extend the two-term asymptotics to $\gamma \in (\gamma_0, 2)$ we argue as in Section \ref{sec:mainproofsmall}, but now we interpolate between Riesz means of orders $2$ and $\gamma_0$ instead of between Riesz means of order $1$ and some $0<\kappa<\gamma$. This completes the proof of Theorem~\ref{thm: KrogerBerezin extrapolation}.
\end{proof}

We have chosen to formulate Theorem~\ref{thm: KrogerBerezin extrapolation} in the case where the second term in the heat trace asymptotics is of the form $C t^{-\frac{d-1}2}$. In principle, the argument might be also applicable when the second term is of the form $C t^{-\frac{d-\beta}2}$ for some $0<\beta<1$. Such a situation can occur if the boundary of $\Omega$ is fractal-like and $d-\beta$ is the dimension, in some sense, of $\partial\Omega$. The heat trace asymptotics then give two-term asymptotics for Riesz means with $\gamma\geq 2$. If, at the same time, $N(\lambda,-\Delta_\Omega^\#) = L_{0,d}^{\rm sc}  |\Omega| \lambda^{\frac{d}2} + O(\lambda^{\frac{d-\beta}2})$, then this gives asymptotics for Riesz means with any $\gamma>0$.  Examples of heat traces that show this behaviour can be found, for instance, in~\cite{LapidusPomerance_96}. Moreover, in the Dirichlet case a sufficient condition for the required bound on $N(\lambda,-\Delta_\Omega^\#) - L_{0,d}^{\rm sc}  |\Omega| \lambda^{\frac{d}2}$ is given in Remark \ref{rem: WeylBerry}. We refrain from pursuing this further.

\subsection{Three-term asymptotics in planar polygons}

Let $\Omega \subset \R^2$ be a bounded, open set whose boundary is polygonal; that is, $\partial\Omega$ consists of a finite number of line segments $\{e_1, \ldots, e_n\}$ that meet only at their endpoints $\{p_1, \ldots, p_n\}$ and each endpoint is shared by exactly two of the segments. The interior angle at the vertex $p_i$ is defined in the natural manner and here denoted by $\alpha_i \in (0, 2\pi]$.

The following theorem provides a non-asymptotic bound on $\Tr(-\Delta_\Omega^{\rm D}-\lambda)_\limminus^{\gamma}$, which depends in a rather explicit way on the geometric characteristics of $\Omega$.

The geometric characteristics in question are as follows. For each vertex $p_i$ of $\Omega$ let $W_i$ denote the infinite wedge with vertex $p_i$ and opening angle $\alpha_i$ with the property that the two edges of $\Omega$ adjacent to $p_i$ are contained in $\partial W_i$. Define
\begin{align*}
    \alpha & := \min_{i=1, \ldots, n}\alpha_i\,,\\
    W_i(r) & := W_i \cap B_r(p_i)\,,\\
    R & := \frac{1}{2}\sup\Bigl\{r>0 : W_i(r)\cap W_j(r) = \emptyset \mbox{ for all }i\neq j,\ \bigcup_{i=1}^n W_i(r)\subset \Omega\Bigr\}\,.
\end{align*}

The following is our result concerning non-asymptotic bounds for planar polygons.

\begin{theorem}\label{thm: three term asymptotics in polygons}
 	Let $\Omega\subset\R^2$ be a bounded, open set whose boundary is polygonal with interior angles $\{\alpha_i\}_{i=1}^n$. Then, for all $\gamma \geq0$ and $\lambda >0$,
 	\begin{align*}
 		\biggl|\Tr(-\Delta_\Omega^{\rm D}-&\lambda)_\limminus^{\gamma}- L_{\gamma,2}^{\rm sc}|\Omega|\lambda^{\gamma+1} + \frac{1}{4}L_{\gamma,1}^{\rm sc}\Haus^{1}(\partial\Omega)\lambda^{\gamma+\frac12}- \lambda^{\gamma}\sum_{i=1}^n \frac{\pi^2-\alpha_i^2}{24\pi \alpha_i}\biggr|\\
        &\leq C_\gamma (1+R^2\sin(\alpha/2)^2\lambda)^{- \frac{\gamma+1}{2}}\Bigl( |\Omega|\lambda^{\gamma+1} + \Bigl(n+ \frac{|\Omega|}{R^2}\Bigr)\frac{1}{\alpha^2} \lambda^\gamma \Bigr)\,,
 	\end{align*}
    where the constant $C_\gamma$ depends only on $\gamma$. 
\end{theorem} 

Besides the non-asymptotic bound in the Dirichlet case, we also have an asymptotic result in the Neumann case.

\begin{theorem}\label{thm: three term asymptotics in polygons, Neumann}
 	Let $\Omega\subset\R^2$ be a bounded, open set whose boundary is polygonal with interior angles $\{\alpha_i\}_{i=1}^n$. If $\gamma \geq0$, then, as $\lambda\to\infty$,
 	\begin{align*}
 		\Tr(-\Delta_\Omega^{\rm N}-\lambda)_\limminus^{\gamma}= L_{\gamma,2}^{\rm sc}|\Omega|\lambda^{\gamma+1} + \frac{1}{4}L_{\gamma,1}^{\rm sc}\Haus^{1}(\partial\Omega)\lambda^{\gamma+\frac12}+ \lambda^{\gamma}\sum_{i=1}^n \frac{\pi^2-\alpha_i^2}{24\pi \alpha_i} + O(\lambda^{\frac{\gamma+1}{2}})\,.
 	\end{align*}
\end{theorem} 

We emphasize that the bound in Theorem \ref{thm: three term asymptotics in polygons} is non-asymptotic, reminiscent of that in Theorem \ref{thm: Weyl asymptotics convex}, but here we do not assume convexity. Asymptotically, as $\lambda\to\infty$, the following accuracy of the bound should be emphasized:
\begin{itemize}
    \item[(a)] For $\gamma=0$ the bound reproduces the correct leading order term and gives an order-sharp remainder (to be compared with the logarithmic excess factor for more general sets in Proposition \ref{prop: Quant Weyl}). This was previously proved in \cite{BaileyBrownell62,Fedosov63}. 
    \item[(b)] For $0<\gamma\leq 1$ the bound correctly reproduces the first two asymptotic terms and gives an order-sharp remainder term. (This is reminiscent of Theorems~\ref{thm: Weyl asymptotics Lipschitz} and \ref{thm: Weyl asymptotics convex}, but here in a uniform way without assuming convexity.) Similar, but different bounds of a comparable accuracy appear in \cite{Fedosov63}. (The difference is that there the author considers Riesz means with respect to the variable $\sqrt\lambda$, while we consider Riesz means with respect to the variable $\lambda$.)
    \item[(c)] For $\gamma>1$ the bound correctly reproduces the first three asymptotic terms. As far as we know, such asymptotics have not appeared in the literature before.
\end{itemize}
Corresponding remarks apply to Theorem~\ref{thm: three term asymptotics in polygons, Neumann}.

We have restricted ourselves to the two-dimensional case. Some asymptotic results for the case of $d\geq 3$ can be found in \cite{Fedosov64}. 

The non-asymptotic bound in Theorem \ref{thm: three term asymptotics in polygons} might be useful in spectral shape optimization problems for polygons, just like our Theorem \ref{thm: Weyl asymptotics convex} will be useful in such problems for convex sets. In fact, the motivation for the present paper came from a question, asked by J.~Lagac\'e at an Oberwolfach conference in summer of 2023, concerning spectral shape optimization problems for polygons. We are grateful to J.~Lagac\'e for a stimulating discussion.

The input in the proof of Theorems \ref{thm: three term asymptotics in polygons} and \ref{thm: three term asymptotics in polygons, Neumann} are corresponding asymptotic expansions for the heat traces with exponentially small remainder. In the context of the heat trace it has been known for some time that the presence of corners influences the short time asymptotics, see e.g.~\cite{BaileyBrownell62,Kac,McKeanSinger,vdBergSrisatkunarajah88}. Motivated by Kac's celebrated question `Can one hear the shape of a drum?', recently there have been several works pertaining to `hearing corners' by means of heat trace asymptotics~\cite{Hezari_etal_17,LuRowlett_16,Nursultanov_etal_19,Hezari_etal_21,Nursultanov_etal_23}.

In the Dirichlet case the required estimate for the heat trace is provided by a theorem of van den Berg and Srisatkunarajah~\cite{vdBergSrisatkunarajah88} which states that:

\begin{theorem}[{\cite[Theorem 1]{vdBergSrisatkunarajah88}}]\label{thm: bound vdBergSrisatkunarajah}
 	Let $\Omega$ be a bounded, open set whose boundary is polygonal. Then, for all $t>0$,
 	\begin{equation*}
 		\biggl|\Tr(e^{t\Delta_\Omega^{\rm D}})- \frac{|\Omega|}{4\pi t}+ \frac{\Haus^1(\partial\Omega)}{8\sqrt{\pi t}}- \sum_{i=1}^n \frac{\pi^2-\alpha_i^2}{24\pi \alpha_i}\biggr| \leq \Bigl(5n+ \frac{20|\Omega|}{R^2}\Bigr)\frac{1}{\alpha^2} e^{-\frac{1}{16t}R^2\sin(\alpha/2)^2}\,.
 	\end{equation*}
\end{theorem} 
 
\begin{proof}[Proof of Theorem \ref{thm: three term asymptotics in polygons}]
Define a Borel measure on $\R$ by
\begin{align*}
	\mu(\omega) & := \sum_{k=1}^\infty \delta_{\lambda_k(-\Delta_\Omega^{\rm D})}(\omega) - L_{0, 2}^{\rm sc}|\Omega| \int_{\omega\cap (0, \infty)}d\lambda\\
    &\quad+ \frac{1}{8}L_{0, 1}^{\rm sc}\Haus^1(\partial\Omega) \int_{\omega\cap (0, \infty)}\lambda^{-\frac{1}2}\,d\lambda- \delta_0(\omega)\sum_{i=1}^n \frac{\pi^2-\alpha_i^2}{24\pi \alpha_i} \,,
\end{align*}
where $\delta_x$ denotes the Dirac measure at $x$.
Note that the Borel measure $\tilde \mu$ defined by
\begin{align*}
    \tilde \mu(\omega)
    & :=	\mu(\omega) + \delta_0(\omega)\sum_{i=1}^n \frac{\pi^2-\alpha_i^2}{24\pi \alpha_i} +L_{0, 2}^{\rm sc}|\Omega|\int_{\omega\cap (0, \infty)}d\lambda \\
    &=  \sum_{k=1}^\infty \delta_{\lambda_k(-\Delta_\Omega^{\rm D})}(\omega) + \frac{1}{8}L_{0, 1}^{\rm sc}\Haus^1(\partial\Omega) \int_{\omega\cap(0, \infty)}\lambda^{-\frac{1}2}\,d\lambda
\end{align*}
is a nonnegative measure.

Recalling that $L_{0,2}^{\rm sc}= \frac{1}{4\pi}, L_{0,1}^{\rm sc}= \frac{1}{\pi}$ we calculate
\begin{align*}
	\int_{[0, \infty)} e^{-t\lambda}d\mu(\lambda) = \Tr(e^{t\Delta_\Omega^{\rm D}})- \frac{|\Omega|}{4\pi t}+ \frac{\Haus^1(\partial\Omega)}{8\sqrt{\pi t}}- \sum_{i=1}^n \frac{\pi^2-\alpha_i^2}{24\pi \alpha_i}\,.
\end{align*}
Similarly, using the properties of $L_{\gamma, d}^{\rm sc}$,
\begin{align*}
	\int_{[0, \lambda)} (1-v/\lambda)^{\gamma}d\mu(v) &=\lambda^{-\gamma}\biggl[\Tr(-\Delta_\Omega^{\rm D}-\lambda)_\limminus^{\gamma}- L_{\gamma,2}^{\rm sc}|\Omega|\lambda^{\gamma+1}\\
	&\qquad + \frac{1}{4}L_{\gamma,1}^{\rm sc}\Haus^{1}(\partial\Omega)\lambda^{\gamma+\frac{1}2}- \lambda^{\gamma}\sum_{i=1}^n \frac{\pi^2-\alpha_i^2}{24\pi \alpha_i}\biggr]\,,
\end{align*}
for $\gamma>0$.

By Theorem~\ref{thm: bound vdBergSrisatkunarajah} we see that for all $t>0$
\begin{align*}
	\biggl|\int_0^\infty e^{-t\lambda}d\mu(\lambda)\biggr| \leq \Bigl(5n+ \frac{20|\Omega|}{R^2}\Bigr)\frac{1}{\alpha^2} e^{-\frac{1}{16t}R^2\sin(\alpha/2)^2}\,.
\end{align*}
Therefore, Corollary~\ref{cor: Tauberian cor} yields for any $\gamma \geq 0, \lambda>0$ that
\begin{align*}
	\biggl|\Tr(-\Delta_\Omega^{\rm D}-\lambda)_\limminus^{\gamma}&- L_{\gamma,2}^{\rm sc}|\Omega|\lambda^{\gamma+1} + \frac{1}{4}L_{\gamma,1}^{\rm sc}\Haus^{1}(\partial\Omega)\lambda^{\gamma+\frac{1}2}- \lambda^{\gamma}\sum_{i=1}^n \frac{\pi^2-\alpha_i^2}{24\pi \alpha_i}\biggr|\\
	&\lesssim_{\gamma} \lambda^{\gamma}\Bigl(1+\frac{R^2\sin(\alpha/2)^2}{16}\lambda\Bigr)^{- \frac{\gamma+1}{2}}\Bigl(\Bigl(5n+ \frac{20|\Omega|}{R^2}\Bigr)\frac{1}{\alpha^2} + \frac{|\Omega|}{4\pi}\lambda\Bigr)\\
	&\lesssim_{\gamma} \Bigl(1+\frac{R^2\sin(\alpha/2)^2}{16}\lambda\Bigr)^{- \frac{\gamma+1}{2}}\Bigl(\Bigl(5n+ \frac{20|\Omega|}{R^2}\Bigr)\frac{1}{\alpha^2}\lambda^{\gamma} + \frac{|\Omega|}{4\pi}\lambda^{\gamma+1}\Bigr)\,.
\end{align*}
This completes the proof. 
\end{proof}

\begin{proof}[Proof of Theorem \ref{thm: three term asymptotics in polygons, Neumann}]
    This theorem can be proved in complete analogy. The use of van den Berg and Srisatkunarajah's uniform bound is replaced by the asymptotic bound
\begin{equation*}
    \Tr(e^{t\Delta_\Omega^{\rm N}}) =  \frac{|\Omega|}{4\pi t}+ \frac{\Haus^1(\partial\Omega)}{8\sqrt{\pi t}}+ \sum_{i=1}^n \frac{\pi^2-\alpha_i^2}{24\pi \alpha_i} + O(e^{-\frac{c}t})
\end{equation*}
for some $c$ depending on the polygon $\Omega$, which appears, for instance, in \cite{Hezari_etal_17}. The asymptotic expansion is also proved in \cite{BaileyBrownell62} but without an explicit expression for the third term.    
\end{proof}


\subsection{Remainder terms for more regular boundaries}

We return our attention to the asymptotics in Theorem \ref{thm: Weyl asymptotics Lipschitz}. As we have already emphasized, the remainder term $o(\lambda^{\gamma+\frac{d-1}2})$ cannot be improved within the class of Lipschitz domains. In this subsection we briefly discuss possible improvements for sets $\Omega$ whose boundary has a limited amount of smoothness.

\begin{theorem}
    Let $d\geq 2$, $0<\gamma\leq 1$, $0<\alpha<1$ and let $\Omega\subset\R^d$ be a bounded open set with boundary of class $C^{1,\alpha}$. Then, as $\lambda\to\infty$,
    \begin{equation*}
		\Tr(-\Delta_\Omega^{\rm D} -\lambda)_\limminus^\gamma = L_{\gamma,d}^{\rm sc} |\Omega| \lambda^{\gamma+\frac d2} - \frac14 L_{\gamma,d-1}^{\rm sc} \mathcal H^{d-1}(\partial\Omega) \lambda^{\gamma+\frac{d-1}{2}} + O(\lambda^{\gamma+\frac{d-1}{2}-\frac{\gamma\alpha}{2+\alpha}}(\ln\lambda)^{1-\gamma}) \,.
	\end{equation*}
\end{theorem}

\begin{proof}
    We only give a sketch of the proof. We define $f$ and $f^{(\kappa)}$ in the same way as in the proof of Theorem \ref{thm: Weyl asymptotics Lipschitz} and use inequality \eqref{eq:rieszinterpolapplied} with $\kappa=0$. To bound the term involving $f^{(1)}$ we use the result of \cite{FrankGeisinger11} (which is the assertion of the theorem in the special case $\gamma=1$). To bound the term involving $f^{(0)}$ we use Proposition \ref{prop: Quant Weyl}.
\end{proof}


\subsection{Two-term asymptotics for Robin Laplacians}

In this subsection we show that our method is applicable to the case of Robin boundary conditions as well. This paper already being long, we have restrict ourselves to stating a result under assumptions that are stronger than necessary, but we will summarize the main steps of a possible strategy to extend this result. We note that after the first version of this paper appeared, in \cite{FrankLarson_Robin2024} we have proved an analogue of Theorem~\ref{thm: Weyl asymptotics Lipschitz} for Robin boundary conditions through a different approach than that we shall explain here. However, it should be mentioned that the approach followed in \cite{FrankLarson_Robin2024} takes Theorem \ref{thm: Weyl asymptotics Lipschitz} with $\#= {\rm N}$ as its starting point.

Note that in the case of Robin boundary conditions the first two terms in the asymptotic expansion are expected to be the same as in the Neumann case. The dependence on the function appearing in the Robin boundary condition is only expected to enter in the third term in the asymptotics, provided of course an asymptotic expansion with so many terms exist (which should require $\gamma> 1$ or at least $\gamma\geq1$). This expectation comes from the heat kernel asymptotics that appear, for instance, in \cite{BrGi}.

We denote the function that appears in the Robin boundary condition by $\sigma$ with the sign convention that positive values of $\sigma$ correspond to a repulsion from the boundary. Thus, if $-\Delta_\Omega^{(\sigma)}$ denotes the corresponding selfadjoint realization in $L^2(\Omega)$, its quadratic form is
$$
\int_\Omega |\nabla u(x)|^2\,dx + \int_{\partial\Omega} \sigma(x) |u(x)|^2\,d\mathcal H^{d-1}(x)\,.
$$

\begin{theorem}\label{robin}
    Let $d\geq 2$ and $\gamma>0$. Let $\Omega\subset\R^d$ be a bounded open set with boundary of class $C^1$ and let $0\leq\sigma\in L^\infty(\partial\Omega)$. Then, as $\lambda\to\infty$,
    \begin{equation*}
		\Tr(-\Delta_\Omega^{(\sigma)} -\lambda)_\limminus^\gamma = L_{\gamma,d}^{\rm sc} |\Omega| \lambda^{\gamma+\frac d2} + \frac14 L_{\gamma,d-1}^{\rm sc} \mathcal H^{d-1}(\partial\Omega) \lambda^{\gamma+\frac{d-1}{2}} + o(\lambda^{\gamma+\frac{d-1}{2}}) \,.
	\end{equation*}
\end{theorem}

\begin{proof}
    The result for $\gamma=1$ appears in \cite[Theorem 1.2]{FrankGeisinger12}. (This is where the $C^1$ assumption on the boundary comes in. Note also that \cite[Theorem 1.2]{FrankGeisinger12} does not have a restriction on the sign of $\sigma$.) The result for $\gamma>1$ follows from that for $\gamma=1$ by the same integration method as in Section \ref{sec:mainprooflarge}. To prove the result for $0<\gamma<1$ we argue as in Section \ref{sec:mainproofsmall} and reduce matters to proving the order sharp bound
    \begin{equation}
        \label{eq:robinordersharp}    
        \Tr(-\Delta_\Omega^{(\sigma)} -\lambda)_\limminus^\gamma = L_{\gamma,d}^{\rm sc} |\Omega| \lambda^{\gamma+\frac d2} + O(\lambda^{\gamma+\frac{d-1}{2}}) \,.
    \end{equation}
    In our proof of the latter, the sign assumption on $\sigma$ comes in. (Note however that the proof is valid assuming only Lipschitz regularity of the boundary.)
 
    To prove \eqref{eq:robinordersharp}, we note that the assumption $\sigma\geq 0$ implies the operator inequalities $-\Delta_\Omega^{\rm N}\leq-\Delta_\Omega^{(\sigma)}\leq-\Delta_\Omega^{\rm D}$, which in turn implies that
    $$
    \Tr(-\Delta_\Omega^{\rm D} -\lambda)_\limminus^\gamma
    \leq \Tr(-\Delta_\Omega^{(\sigma)} -\lambda)_\limminus^\gamma
    \leq \Tr(-\Delta_\Omega^{\rm N} -\lambda)_\limminus^\gamma \,.
    $$
    Therefore \eqref{eq:robinordersharp} is an immediate consequence of Theorems \ref{thm: Sharp Weyld} and \ref{thm: Sharp Weyln}.
\end{proof}

We believe that Theorem \ref{robin} extends to the case of Lipschitz boundaries and to not necessarily nonnegative functions $\sigma$ and we expect that our methodology should allow to prove this. There are, however, several steps that need additional work.

To obtain two-term asymptotics for $\gamma\geq 1$ in the Lipschitz case one can try to follow the method in Section \ref{sec:mainprooflarge}, comparing again with the Dirichlet Laplacian. (To achieve a higher accuracy one might also compare first $-\Delta_\Omega^{(\sigma)}$ with $-\Delta_\Omega^{(\sigma_\pm)}$ and then compare $-\Delta_\Omega^{(\sigma_\pm)}$ with $\Delta_\Omega^{\rm N}$. This higher accuracy might be relevant for three-term asymptotics.) This comparison reduces the problem to proving an analogue of Brown's result (Theorem \ref{thm: Weyl asymptotics Lipschitz heat}) with Robin boundary conditions. We have not pursued this.

To obtain two-term asymptotics for $0<\gamma< 1$ one can try to follow the method in Section~\ref{sec:mainproofsmall}. Given the asymptotics for $\gamma=1$, one needs to find an order-sharp remainder term as in Theorem \ref{thm: Sharp Weyln}. In the case $\sigma\geq 0$ this was achieved in \eqref{eq:robinordersharp} in the proof above. For general $\sigma$ one could try to follow the proof of Proposition \ref{thm: Sharp Weyln}, based on pointwise bounds. A pointwise bound in the bulk can probably be obtained via wave equation methods. (Some modification to what we explained in Subsection \ref{sec: Wave method} is necessary when the operator is no longer nonnegative.) One should also be able to deduce an analogue of the heat kernel bound in Lemma~\ref{heatnapriori}, which would then lead to a pointwise a priori bound that can be used close to the boundary.

This concludes our outline of how we believe one can extend Theorem \ref{robin}. Finally, we mention that it is also interesting to consider the Robin problem in a semiclassical form where the boundary condition depends on the semiclassical parameter; see \cite{FrankGeisinger12}. In this setting the contribution of the boundary condition can appear already in the second term of the asymptotics. Since one is no longer dealing with the spectrum of a single operator, but rather with the spectra of parameter-dependent operators, it is less clear to which extent the methods of the present paper can be applied. We consider this a worthwhile problem to study.


\part{Tauberian theorems}


\section{A convexity theorem of Riesz}
\label{sec: Riesz convexity}

An important role in our analysis is played by a result of Riesz stating the log-convexity of Riesz means \cite{Riesz1922}. 
For $\lambda > 0$, $\phi$ a bounded measurable function on $[0, \lambda]$, $\kappa >0$, and $\Lambda \in [0, \lambda]$ we define
\begin{equation*}
	\phi^{(\kappa)}(\Lambda) := \frac{1}{\Gamma(\kappa)}\int_0^\Lambda(\Lambda-\mu)^{\kappa-1}\phi(\mu)\,d\mu\,,
\end{equation*}
and set $\phi^{(0)}(\Lambda) := \phi(\Lambda)$. There is no canonical choice for the normalization of the Riesz means. The above choice, which is the one from \cite{Riesz1922}, leads to the semigroup property
\begin{equation}\label{eq: semigroup Riesz prop}
	(\phi^{(\kappa_1)})^{(\kappa_2)}(\Lambda) = \phi^{(\kappa_1+\kappa_2)}(\Lambda)\,.
\end{equation}
The result of Riesz that we shall rely on can now be stated as follows. For the sake of completeness we reproduce Riesz's original proof below; this proof also appears in~\cite{Chandrasekharan_TypicalMeans}. For an alternative proof see \cite{HormanderRieszMeans}.

\begin{proposition}\label{prop: Riesz logconvexity}
	Fix $\gamma>0$. There is a constant $C<\infty$ such that for all $\lambda>0$, every bounded, measurable function $\phi$ on $[0,\lambda]$ and every $0<\sigma<\gamma$, one has
	$$
	\sup_{\Lambda\in [0, \lambda]}|\phi^{(\sigma)}(\Lambda)| \leq C \biggl( \sup_{\Lambda\in[0,\lambda]} |\phi(\Lambda)| \biggr)^{1-\frac\sigma\gamma} \biggl( \sup_{\Lambda\in[0,\lambda]} |\phi^{(\gamma)}(\Lambda)| \biggr)^\frac{\sigma}{\gamma}\,.
	$$
\end{proposition}

Before giving the proof of the proposition we state and prove an auxiliary assertion.
\begin{lemma}\label{lem: Riesz lemma}
	Let $0<\gamma\leq 1$, $0<\lambda_1\leq\lambda$ and let $\phi$ be a bounded, measurable function on $[0,\lambda_1]$. Then there is a $\lambda_0\in[0,\lambda_1]$ such that
	$$
	\int_0^{\lambda_1} (\lambda-\mu)^{\gamma-1} \phi(\mu) \,d\mu = \int_0^{\lambda_0} (\lambda_0-\mu)^{\gamma-1} \phi(\mu) \,d\mu \,.
	$$
\end{lemma}

\begin{proof}[Proof of Lemma \ref{lem: Riesz lemma}]
	For $\gamma=1$ the claimed equality follows by choosing $\lambda_0=\lambda_1$.
	
	For fixed $0<\gamma<1$ and $0<\lambda_1\leq \lambda$ we introduce the function
	\begin{equation}
		\label{eq:defm}
		M(\Lambda) := \frac{1}{\Gamma(\gamma)\,\Gamma(1-\gamma)} \int_\Lambda^{\lambda_1} (\lambda-\mu)^{\gamma-1} (\mu-\Lambda)^{-\gamma}\,d\mu
		\qquad\text{for}\  0\leq\Lambda\leq\lambda_1 \,.
	\end{equation}
	From the identity (see, e.g.,~\cite[eq.~5.12.1]{NIST})
	\begin{equation*}
		\frac{1}{\Gamma(\gamma)\Gamma(1-\gamma)}\int_\Lambda^\lambda (\lambda-\mu)^{\gamma-1}(\mu-\Lambda)^{-\gamma}\,d\mu = 1
	\end{equation*}
	it follows that
	\begin{equation}
		\label{eq:malt}
		M(\Lambda) = 1 - \frac{1}{\Gamma(\gamma)\,\Gamma(1-\gamma)} \int_{\lambda_1}^\lambda (\lambda-\mu)^{\gamma-1} (\mu-\Lambda)^{-\gamma}\,d\mu \,.
	\end{equation}
	From this formula it follows that $M$ is differentiable. We claim that
	\begin{equation}
		\label{eq:lemmaclaim}
		\int_\mu^{\lambda_1} (\Lambda-\mu)^{\gamma-1} M'(\Lambda)\,d\Lambda = -(\lambda-\mu)^{\gamma-1}
		\qquad\text{for all}\ 0\leq\mu<\lambda_1 \,.
	\end{equation}
	
	To prove this, note that
	$$
	\int_\mu^{\lambda_1} (\Lambda-\mu)^{\gamma-1} M'(\Lambda)\,d\Lambda
	= -\frac1\gamma \frac{d}{d\mu} \biggl(\int_\mu^{\lambda_1} (\Lambda-\mu)^{\gamma} M'(\Lambda)\,d\Lambda \biggr)\,.
	$$
	By integration by parts, the definition of $M$, and interchanging the order of integration, one computes
	\begin{align*}
		& \int_\mu^{\lambda_1} (\Lambda-\mu)^{\gamma} M'(\Lambda)\,d\Lambda
		= - \gamma \int_\mu^{\lambda_1} (\Lambda-\mu)^{\gamma-1} M(\Lambda)\,d\Lambda \\
		& \quad = - \frac{\gamma}{\Gamma(\gamma)\,\Gamma(1-\gamma)} \int_\mu^{\lambda_1} \int_\Lambda^{\lambda_1} (\Lambda-\mu)^{\gamma-1} (\lambda-\mu')^{\gamma-1} (\mu'-\Lambda)^{-\gamma}\,d\mu' \,d\Lambda \\
		& \quad = - \frac{\gamma}{\Gamma(\gamma)\,\Gamma(1-\gamma)} \int_\mu^{\lambda_1} \int_\mu^{\mu'} (\Lambda-\mu)^{\gamma-1} (\lambda-\mu')^{\gamma-1} (\mu'-\Lambda)^{-\gamma} \,d\Lambda \,d\mu'  \\
		& \quad = - \gamma \int_\mu^{\lambda_1} (\lambda-\mu')^{\gamma-1}\,d\mu' \\
		& \quad = (\lambda-\lambda_1)^\gamma - (\lambda-\mu)^\gamma \,.
	\end{align*}
	Taking the derivative with respect to $\mu$ one arrives at the claimed formula \eqref{eq:lemmaclaim}.

	Integrating \eqref{eq:lemmaclaim} against $\phi$ and interchanging the order of integration, we obtain
	$$
	\int_0^{\lambda_1} (\lambda-\mu)^{\gamma-1} \phi(\mu) \,d\mu = - \int_0^{\lambda_1} \int_0^\Lambda (\Lambda-\mu)^{\gamma-1} \phi(\mu)\,d\mu\, M'(\Lambda)\,d\Lambda \,.
	$$
	Noting that $M'\leq 0$ by \eqref{eq:malt}, we deduce from the first mean value theorem that there is a $0\leq\lambda'\leq\lambda_1$ such that
	$$
	\int_0^{\lambda_1} (\lambda-\mu)^{\gamma-1} \phi(\mu) \,d\mu = - \int_0^{\lambda'} (\lambda'-\mu)^{\gamma-1} \phi(\mu)\,d\mu \ \int_0^{\lambda_1} \, M'(\Lambda)\,d\Lambda \,.
	$$
	Since $M(\lambda_1)=0$, this is the same as
	\begin{equation}\label{eq: mean value theorem riesz lemma}
		\int_0^{\lambda_1} (\lambda-\mu)^{\gamma-1} \phi(\mu) \,d\mu = M(0) \int_0^{\lambda'} (\lambda'-\mu)^{\gamma-1} \phi(\mu)\,d\mu \,.
	\end{equation}
	By \eqref{eq:defm} and \eqref{eq:malt} $0<M(0)<1$. The function $s\mapsto \int_0^s (s-\mu)^{\gamma-1} \phi(\mu)\,d\mu$ is continuous and vanishing at $s=0$. Therefore, the intermediate value theorem implies that there is a $0<\lambda_0<\lambda'$ such that
	$$
	\int_0^{\lambda_0} (\lambda_0-\mu)^{\gamma-1} \phi(\mu)\,d\mu
	= M(0) \int_0^{\lambda'} (\lambda'-\mu)^{\gamma-1} \phi(\mu)\,d\mu \,,
	$$
	which combined with~\eqref{eq: mean value theorem riesz lemma} proves the assertion of the lemma.
\end{proof}

\begin{proof}[Proof of Proposition \ref{prop: Riesz logconvexity}]
	Set for $\kappa \geq 0$
	\begin{equation*}
		\Phi_\kappa := \sup_{\Lambda \in [0, \lambda]}|\phi^{(\kappa)}(\Lambda)|\,.
	\end{equation*}
	The inequality claimed in the proposition then becomes
	\begin{equation*}
		\Phi_\sigma \leq C \Phi_0^{1-\frac\sigma\gamma} \Phi_\gamma^\frac\sigma\gamma\,.
	\end{equation*}
	
	\medskip
	
	\noindent\emph{Step 1.} We first assume $\gamma\leq 1$. For $0\leq\Lambda_1\leq \Lambda\leq \lambda$ to be determined, we split
	$$
	\phi^{(\sigma)}(\Lambda) =\frac{1}{\Gamma(\sigma)} \int_0^\Lambda (\Lambda - \mu)^{\sigma-1} \phi(\mu)\,d\mu = I + II
	$$
	where
	\begin{align*}
		I & := \frac{1}{\Gamma(\sigma)} \int_0^{\Lambda_1} (\Lambda - \mu)^{\sigma-1} \phi(\mu)\,d\mu \,, \\
		II & := \frac{1}{\Gamma(\sigma)} \int_{\Lambda_1}^\Lambda (\Lambda - \mu)^{\sigma-1} \phi(\mu)\,d\mu \,.
	\end{align*}
	
	We bound
	$$
	|II| \leq \frac{1}{\Gamma(\sigma)} \int_{\Lambda_1}^\Lambda (\Lambda - \mu)^{\sigma-1} \phi(\mu)\,d\mu \leq \Phi_0 \, \frac{1}{\Gamma(\sigma)} \int_{\Lambda_1}^\Lambda (\Lambda - \mu)^{\sigma-1} \,d\mu = \frac{1}{\sigma\Gamma(\sigma)}\Phi_0 \, (\Lambda - \Lambda_1)^\sigma \,.
	$$
	
	On the other hand, by the second mean value theorem, for some $0\leq\Lambda_1'\leq\Lambda_1$,
	\begin{align*}
		I & = \frac{1}{\Gamma(\sigma)} \int_0^{\Lambda_1} (\Lambda-\mu)^{\gamma-1} (\Lambda-\mu)^{\sigma-\gamma}\phi(\mu)\,d\mu \\
		& = \frac{1}{\Gamma(\sigma)} (\Lambda - \Lambda_1)^{\sigma-\gamma} \int_{\Lambda_1'}^{\Lambda_1} (\Lambda-\mu)^{\gamma-1} \phi(\mu)\,d\mu \,.
	\end{align*}
	By Lemma~\ref{lem: Riesz lemma}, there are $0\leq\Lambda_0\leq\Lambda_1$ and $0\leq\Lambda_0'\leq\Lambda_1'$ such that
	\begin{align*}
		\int_{0}^{\Lambda_1} (\Lambda-\mu)^{\gamma-1} \phi(\mu)\,d\mu
		& = \int_0^{\Lambda_0} (\Lambda_0-\mu)^{\gamma-1} \phi(\mu)\,d\mu \,, \\
		\int_{0}^{\Lambda_1'} (\Lambda-\mu)^{\gamma-1} \phi(\mu)\,d\mu
		& = \int_0^{\Lambda_0'} (\Lambda_0'-\mu)^{\gamma-1} \phi(\mu)\,d\mu \,.
	\end{align*}
	In particular, we find that
	$$
	\left| \int_{\Lambda_1'}^{\Lambda_1} (\Lambda-\mu)^{\gamma-1} \phi(\mu)\,d\mu \right| \leq 2\Gamma(\gamma) \Phi_\gamma
	$$
	and therefore
	$$
	|I| \leq \frac{2\Gamma(\gamma)}{\Gamma(\sigma)}\, \Phi_\gamma \, (\Lambda - \Lambda_1)^{\sigma-\gamma} \,.
	$$
	
	To summarize, we have shown that, for any $0\leq\Lambda_1\leq\Lambda\leq \lambda$,
	$$
	|\phi^{(\sigma)}(\Lambda)| \leq  \frac{1}{\sigma\Gamma(\sigma)}\Phi_0 \,(\Lambda - \Lambda_1)^\sigma + \frac{2\Gamma(\gamma)}{\Gamma(\sigma)}\, \Phi_\gamma \, (\Lambda - \Lambda_1)^{\sigma-\gamma}  \,.
	$$
	If $\Lambda \geq \bigl( 2(\gamma-\sigma)\Gamma(\gamma) \, \frac{\Phi_\gamma}{\Phi_0} \bigr)^{\frac{1}\gamma}$ we can choose
	$$
	\Lambda_1 = \Lambda - \Bigl(2(\gamma-\sigma)\Gamma(\gamma) \, \frac{\Phi_\gamma}{\Phi_0} \Bigr)^{\frac{1}\gamma}
	$$
	and obtain
	$$
	|\phi^{(\sigma)}(\Lambda)| \leq
	2^\frac\sigma\gamma	\frac{\Gamma(\gamma+1)^{\frac\sigma\gamma}}{\Gamma(\sigma+1)}\Bigl(1-\frac{\sigma}\gamma\Bigr)^{-1+\frac{\sigma}{\gamma}} \Phi_0^{1-\frac\sigma\gamma}\Phi_\gamma^{\frac\sigma\gamma}
	\qquad \text{if}\ \Lambda \geq \Bigl( 2(\gamma-\sigma)\Gamma(\gamma) \, \frac{\Phi_\gamma}{\Phi_0} \Bigr)^{\frac{1}\gamma}\,.
	$$
	At the same time, we have for all $\Lambda \in [0, \lambda]$ the trivial bound
	$$
	|\phi^{(\sigma)}(\Lambda)|\leq \left| \frac{1}{\Gamma(\sigma)} \int_0^\Lambda (\Lambda - \mu)^{\sigma-1} \phi(\mu)\,d\mu \right| \leq \frac{1}{\sigma \Gamma(\sigma)}\Phi_0 \Lambda^\sigma \,,
	$$
	which corresponds to taking $\Lambda_1=0$, so that the term $I$ is absent.
	
	Note that as long as $\Lambda<(1-\frac{\sigma}{\gamma})^{-\frac{1}\sigma} \bigl( 2(\gamma-\sigma)\Gamma(\gamma) \, \frac{\Phi_\gamma}{\Phi_0} \bigr)^{\frac{1}\gamma}$, the trivial bound is better than the first bound. Since $(1-\frac{\sigma}{\gamma})^{-\frac{1}\sigma}>1$, this covers the whole parameter regime and therefore
	\begin{equation*}
		|\phi^{(\sigma)}(\Lambda)| \leq
		2^\frac\sigma\gamma	\frac{\Gamma(\gamma+1)^{\frac\sigma\gamma}}{\Gamma(\sigma+1)}\Bigl(1-\frac{\sigma}\gamma\Bigr)^{-1+\frac{\sigma}{\gamma}}\Phi_0^{1-\frac\sigma\gamma}\Phi_\gamma^{\frac\sigma\gamma}
		\qquad \text{for all}\ \Lambda \in [0, \lambda]\,.
	\end{equation*}
	Furthermore, since $\Gamma(1+\kappa)\in [1/2, 1]$ for $\kappa \in [0, 1]$ we have
	$$
		\frac{\Gamma(\gamma+1)^{\frac\sigma\gamma}}{\Gamma(\sigma+1)} \leq 2
	$$
	for all $0<\sigma<\gamma \leq 1$ and thus
	\begin{equation}\label{eq: final inequality step 1 Riesz prop}
		|\phi^{(\sigma)}(\Lambda)| \leq
		2^{1+\frac\sigma\gamma}\Bigl(1-\frac{\sigma}\gamma\Bigr)^{-1+\frac{\sigma}{\gamma}}\Phi_0^{1-\frac\sigma\gamma}\Phi_\gamma^{\frac\sigma\gamma}
		\qquad \text{for all}\ \Lambda \in [0, \lambda]\,.
	\end{equation}
	By maximizing with respect to $\sigma/\gamma \in (0, 1)$ one observes that $$2^{1+\frac{\sigma}\gamma}\biggl( 1- \frac\sigma\gamma \biggr)^{-1+\frac{\sigma}{\gamma}}\leq 4e^{\frac{1}{2e}}\,,$$
	for all $0<\sigma<\gamma$. 
	
	\medskip
	
	\noindent\emph{Step 2.} We now extend the previous result to arbitrary $\gamma> 1$. 
	
	Fix $\gamma>1$ and choose $N\in\N$ with $N\geq 2$ such that $\frac\gamma N\leq \frac12$. (To make $N$ unique, we could assume $\frac\gamma{N-1}>\frac12$, but this is not necessary.) Set $\gamma_n:=\frac{n\gamma}N$ for $n=0,\ldots,N$. Our first goal is to prove the inequality
	\begin{equation}\label{eq: Goal step 2 Riesz prop}
		\Phi_{\gamma_n} \leq C \Phi_0^{1-\frac{\gamma_n}{\gamma}} \Phi_\gamma^{\frac{\gamma_n}{\gamma}}
		\qquad\text{for all}\ n=1,\ldots,N-1 \,,
	\end{equation}
	with a constant $C$ which is only allowed to depend on $\gamma$.
	
	To prove~\eqref{eq: Goal step 2 Riesz prop}, we use the semigroup property~\eqref{eq: semigroup Riesz prop}.
	Applying the bound~\eqref{eq: final inequality step 1 Riesz prop} from Step 1 to the function $\phi^{(\gamma_{n-1})}$ with $\gamma$ replaced by $\gamma_{n+1}-\gamma_{n-1}= \frac{2\gamma}{N}\leq 1$ and $\sigma$ by $\gamma_{n}-\gamma_{n-1} = \frac{\gamma}{N}$, we find that
	\begin{align*}
		\Phi_{\gamma_n} \leq  4 \left(\Phi_{\gamma_{n-1}}  \Phi_{\gamma_{n+1}} \right)^\frac12 \,.
	\end{align*}
	
	It follows now by iteration that
	$$
	\Phi_{\gamma_n} \leq 4^{n(N-n)} \Phi_{0}^{1-\frac{\gamma_n}{\gamma}} \Phi_\gamma^{\frac{\gamma_n}{\gamma}} \,.
	$$
	Indeed, this inequality is obtained by multiplying the inequalities
	$$
	\Phi_{\gamma_m}^{m(N-n)} \leq 4^{m(N-n)} \left( \Phi_{\gamma_{m-1}}\Phi_{\gamma_{m+1}}\right)^\frac{m(N-n)}{2}
	\qquad\text{for}\ m=1,\ldots, n \,,
	$$
	and
	$$
	\Phi_{\gamma_m}^{n(N-m)} \leq 4^{n(N-m)} \left( \Phi_{\gamma_{m-1}}\Phi_{\gamma_{m+1}} \right)^\frac{n(N-m)}{2}
	\qquad\text{for}\ m=n+1,\ldots, N-1 \,.
	$$
	
	Estimating $4^{n(N-n)}\leq 4^{\frac{N^2}4}$ and noting that $N$ only depends on $\gamma$, we find the inequality we set out to prove. 
	
	This proves the inequality in the proposition if $\sigma =\gamma_n$ for some $n=0,\ldots,N$. When $\sigma$ is not of this form, we choose $n=1,\ldots,N$ such that $\gamma_{n-1}<\sigma<\gamma_n$ and we apply the bound from Step 1 to $\phi^{(\gamma_{n-1})}$, using the semigroup property. We obtain
	$$
	\Phi_\sigma \leq 4e^{\frac{1}{2e}} \Phi_{\gamma_{n-1}}^{\frac{\gamma_n-\sigma}{\gamma_n-\gamma_{n-1}}} \Phi_{\gamma_n}^{\frac{\sigma-\gamma_{n-1}}{\gamma_n - \gamma_{n-1}}} \,.
	$$
	Applying the interpolation inequality for $\gamma_{n-1}$ and $\gamma_n$ proved above we finally obtain the inequality claimed in the proposition. This completes the proof.
\end{proof}


\section{A Tauberian theorem for the Laplace transform}\label{sec: Laplace Tauber}

This section is devoted to Tauberian theorems for the Laplace transform. The results as well as their proofs essentially go back to the work of Ganelius \cite{Ganelius54}. However, in the applications of these results in the present paper it is necessary to have statements that are uniform and not only asymptotic in nature. For this reason we provide complete proofs.

Asymptotic results that are related to those in \cite{Ganelius54} have been announced in \cite{Avakumovic52} for $\gamma=0$ and in \cite{Ganelius56} for general $\gamma\geq 0$ without proofs. See also \cite[Section 2]{FrSa} for a related, but different non-asymptotic version of a Tauberian theorem of Ingham.

\begin{proposition}\label{prop: Tauberian v2}
	Let $\mu$ be a locally finite signed Borel measure on $\R$. Assume that there exist finite collections $\{K_i\}_{i= 0}^N \subset [0, \infty)$ and $\{\nu_i\}_{i= 1}^N \subset (0, \infty)$ such that the Borel measure $\tilde \mu$ on $\R$, defined by 
	$$
		\tilde\mu(\omega) := \mu(\omega)+ K_0\delta_0(\omega)+\sum_{i=1}^N\nu_i K_i \int_{\omega\cap (0, \infty)} u^{\nu_i-1}du \,,
	$$ 
	is nonnegative.  
	Then, for all $\gamma \geq 0$, all integers $k \geq k_0$, and all $u>0$,
	\begin{align*}
		\biggl|\int_{[0, u)} (1-v/u)^\gamma d\mu(v)\biggr| &\leq B^k \sup_{1\leq s \leq k}\biggl|\int_{[0,\infty)} e^{-\frac{s v}u}\,d\mu(v)\biggr| + \frac{C}{k^{\gamma+1}}\sum_{i: \nu_i \geq 1} |K_i| u^{\nu_i}\\
		&\qquad +\frac{D}{k^{\gamma}}\sum_{i: \nu_i < 1} |K_i| \Bigl(\frac{u}k\Bigr)^{\nu_i}\,,
	\end{align*}
	with $B, C, D, k_0$ depending only on $\gamma, \{\nu_1, \ldots, \nu_N\}$.
\end{proposition}

As we shall multiple times need to use Proposition~\ref{prop: Tauberian v2} in a setting where the Laplace transform satisfies an exponential bound we record it as a separate corollary. 

\begin{corollary}\label{cor: Tauberian cor}
	Let $\mu$ be a locally finite signed Borel measure on $\R$. Assume that there exist $\epsilon, c_1, c_2, A>0$ and finite collections $\{K_i\}_{i= 0}^N \subset [0, \infty)$ and $\{\nu_i\}_{i= 1}^N \subset (0, \infty)$ such that 
	\begin{enumerate}
		\item the Borel measure $\tilde \mu$ on $\R$ defined by 
	$$
		\tilde\mu(\omega) := \mu(\omega)+K_0\delta_0(\omega)+\sum_{i=1}^N\nu_i K_i \int_{\omega \cap (0, \infty)} u^{\nu_i-1}\,du
	$$ 
	is nonnegative, and
		\item for all $t\leq c_2 A$
		\begin{equation}\label{eq: exp Laplace transform bound}
			\biggl|\int_{[0, \infty)} e^{-tv}\,d\mu(v)\biggr| \leq c_1 e^{-(c_2/t)^\epsilon}\,.
		\end{equation}
	\end{enumerate}
	Then, for all $\gamma \geq 0$ there exist constants $C, D$ depending only on $\gamma, \epsilon, \{\nu_1, \ldots, \nu_N\}$ so that
	\begin{equation}\label{eq: pow Riesz bound}
		\biggl|\int_{[0, u)} (1-v/u)^\gamma d\mu(v)\biggr| \leq C (1+c_2 u)^{- \frac{(1+\gamma)\epsilon}{1+\epsilon}}\Bigl(c_1 + \sum_{i:\nu_i\geq 1} |K_i|u^{\nu_i}+ \sum_{i:\nu_i<1} |K_i|u^{\nu_i}(c_2u)^{\frac{(1-\nu_i)\epsilon}{1+\epsilon}}\Bigr)
	\end{equation}
    for all $u \geq \frac{D}{c_2A}(1+A^{-\epsilon})$. In particular, if~\eqref{eq: exp Laplace transform bound} holds for all $t>0$ then~\eqref{eq: pow Riesz bound} holds for all $u>0$.
\end{corollary}

The key ingredient in our proof is the following polynomial approximation result \cite[Theorem VII.3.4]{Korevaar_TauberianTheory_book}.

\begin{theorem}\label{thm: Pol approximation}
	Fix $m\in \N \cup \{0\}$ and $\alpha, \beta \in \R$. Set
	\begin{equation*}
		G_m(t) = \begin{cases}
			(1-t)^m & \mbox{for }0\leq t \leq 1\,,\\
			0 & \mbox{for }1<t<\infty\,.
		\end{cases}
	\end{equation*}
	There exist constants $B_1>0, B_2>1$ and $k_0 \in \N$ depending only on $m, \beta, \alpha$ such that the following holds. For every integer $k \geq k_0$, there are polynomials in the variable $e^{-t}$ of degree $\leq k$:
	\begin{equation*}
		p_k(t) = \sum_{j=1}^k a_{jk}e^{-jt} \,, \quad P_k(t) = \sum_{j=1}^k b_{kj}e^{-jt}\,,
	\end{equation*}
	which satisfy
	\begin{align*}
		&p_k(t)\leq G_m(t)\leq P_k(t) \quad \mbox{for } 0<t<\infty\,, \quad p_k(0)=P_k(0)=1\,,\\
		&\int_0^\infty e^{\alpha t}t^{-\beta }[P_k(t)-p_k(t)]\,dt \leq \frac{B_1}{k^{m+1}}\,,\\
		&\sum_{j=1}^k |a_{jk}| \leq B_2^k\,, \quad \sum_{j=1}^k |b_{jk}|\leq B_2^k\,.
	\end{align*}
\end{theorem}

Before we prove the full statement of Proposition~\ref{prop: Tauberian v2} let us show how the polynomial approximation in Theorem~\ref{thm: Pol approximation} can be used to prove a bound when $\gamma \in \N$.

\begin{lemma}\label{lem: Tauberian v1}
	Let $\mu$ be a locally finite signed Borel measure on $\R$. Assume that there exist finite collections $\{K_i\}_{i= 0}^N \subset [0, \infty)$ and $\{\nu_i\}_{i= 1}^N \subset (0, \infty)$ such that the Borel measure $\tilde \mu$ on $\R$, defined by 
	$$
		\tilde\mu(\omega) := \mu(\omega)+ K_0 \delta_0(\omega)+\sum_{i=1}^N\nu_i K_i \int_{\omega\cap (0, \infty)} u^{\nu_i-1}\,du \,,
	$$ 
	is nonnegative. 
	Then, for all $\gamma \in \N\cup \{0\}$, all integers $k \geq k_0$, and all $u>0$,
	\begin{equation*}
		\biggl|\int_{[0, u)} (1-v/u)^\gamma d\mu(v)\biggr| \leq B^k \max_{j=1, \ldots, k}\biggl|\int_{[0,\infty)} e^{-\frac{j v}u}\,d\mu(v)\biggr| + \frac{C}{k^{\gamma+1}}\sum_{i=1}^N |K_i| u^{\nu_i}\,,
	\end{equation*}
	with $B, C, k_0$ depending only on $\gamma, \{\nu_1, \ldots, \nu_N\}$.
\end{lemma}
\begin{remark}
    We note that in this lemma we do not need to distinguish between $\nu_i<1$ or $\nu_i\geq 1$. It is not clear to us if the distinction can be avoided also in Proposition~\ref{prop: Tauberian v2}. However, in this paper we shall only need to apply the proposition in situations where $\nu_i \geq 1$ for all $i$.
\end{remark}

\begin{proof}[Proof of Lemma~\ref{lem: Tauberian v1}]
	The proof follows that of Ganelius in~\cite{Ganelius54}.

	Define 
	\begin{equation*}
		Q(t) := \int_{[0, \infty)} e^{-tu}\,d\mu(u)\,.
	\end{equation*}

	Let $p_k, P_k$ be the polynomials provided by Theorem~\ref{thm: Pol approximation} with $m=\gamma, \alpha=1, \beta = 1-\min_{i= 1, \ldots, N}\nu_i$ and $k\geq k_0$ to be chosen below. Then with $\mu= \mu^\limplus- \mu^\limminus$ denoting the Hahn--Jordan decomposition of $\mu$,
	\begin{align*}
		\int_{[0, u)} (1-v/u)^\gamma \,d\mu(v) &= \int_{[0, u)} (1-v/u)^m\,d\mu^\limplus(v)-\int_{[0, u)} (1-v/u)^m\,d\mu^\limminus(v)\\
		&\leq \int_{[0, \infty)} P_k(v/u)\,d\mu^\limplus(v)-\int_{[0, \infty)} p_k(v/u)\,d\mu^\limminus(v)\\
		&= \int_{[0, \infty)} P_k(v/u)\,d\mu(v)+\int_{[0, \infty)} [P_k(v/u)-p_k(v/u)]\,d\mu^\limminus(v)\,.
	\end{align*}
	In the same manner,
	\begin{align*}
		\int_{[0, u)} (1-v/u)^\gamma \,d\mu(v) 
		&\geq \int_{[0, \infty)} p_k(v/u)\,d\mu(v)-\int_{[0, \infty)} [P_k(v/u)-p_k(v/u)]\,d\mu^\limminus(v)\,.
	\end{align*}
	Note that the second integral is nonnegative as $P_k\geq p_k$ and $d\mu^\limminus$ is a nonnegative measure.

	By the properties of $p_k, P_k$, we have
	\begin{align*}
		\biggl|\int_{[0, \infty)} p_k(v/u)\,d\mu(v)\biggr| &= \biggl|\sum_{j=1}^k a_{jk}Q(j/u)\biggr| \leq B_2^k\max_{j=1, \ldots, k}|Q(j/u)|\,,\\
		\biggl|\int_{[0, \infty)} P_k(v/u)\,d\mu(v)\biggr| &= \biggl|\sum_{j=1}^k b_{jk}Q(j/u)\biggr| \leq B_2^k\max_{j=1, \ldots, k}|Q(j/u)|\,.
	\end{align*}
	
	To bound the remaining integral we use the assumption that $\tilde\mu$ is nonnegative (which is equivalent to the measure $-\mu^\limminus(\omega)+ K_0\delta_0(\omega) + \sum_{i=1}^N \nu_i K_i \int_{\omega\cap (0, \infty)} u^{\nu_i-1}\,du$ being nonnegative). Since $P_k(0)=p_k(0)=1$, the assumed nonnegativity implies that
	\begin{align*}
		\int_{[0, \infty)}[P_k(v/u)-p_k(v/u)]\,d\mu^\limminus(v) &\leq  \sum_{i=1}^NK_i|\nu_i|\int_0^\infty[P_k(v/u)-p_k(v/u)]v^{\nu_i-1}\,dv\\
		&= \sum_{i=1}^NK_i|\nu_i| u^{\nu_i}\int_0^\infty[P_k(s)-p_k(s)]s^{\nu_i-1}\,ds\,.
	\end{align*}
	Next note that $\max_{i=1, \ldots, N} s^{\nu_i-1} \leq C_{\nu} e^s s^{-\beta}$ with $C_\nu$ depending only on $\{\nu_1, \ldots, \nu_N\}$. When combined with the fact that $P_k(s)-p_k(s) \geq 0$ and the estimate above, one obtains
	\begin{align*}
		\biggl|\int_{[0, \infty)}[P_k(v/u)-p_k(v/u)]\,d\mu^\limminus(v)\biggr| &\leq C_\nu\sum_{i=1}^N|K_i\nu_i| u^{\nu_i}\int_0^\infty e^s s^{-\beta}[P_k(s)-p_k(s)]\,ds\\
		& \leq \frac{B_1 C_\nu}{k^{\gamma+1}}\sum_{i=1}^N|K_i\nu_i| u^{\nu_i}\,.
	\end{align*}

	Upon combining the above we reach the conclusion that for every $k \geq k_0$ and $u> 0$,
	\begin{equation*}
		\biggl|\int_{[0, u)} (1-v/u)^\gamma d\mu(v)\biggr| \leq B_2^k \max_{j=1, \ldots, k}\biggl|\int_{[0, \infty)} e^{-\frac{jv}u}\,d\mu(v)\biggr| + \frac{B_1 C_\nu}{k^{\gamma+1}}\sum_{i=1}^N |K_i \nu_i| u^{\nu_i}\,.
	\end{equation*}
	This completes the proof Lemma~\ref{lem: Tauberian v1}.
	\end{proof}

	\begin{proof}[Proof of Proposition~\ref{prop: Tauberian v2}]
	To extend Lemma~\ref{lem: Tauberian v1} to arbitrary $\gamma \geq 0$ we utilize the semigroup property of Riesz means. Set
	\begin{equation*}
		F_\gamma(u) = \frac{1}{\Gamma(\gamma+1)}\int_{[0, u)} (u-v)^\gamma\, d\mu(v) = \frac{u^\gamma}{\Gamma(\gamma+1)}\int_{[0, u)}(1-v/u)^\gamma \,d\mu(v)\,.
	\end{equation*}
	Observe that for $\gamma \geq 1$ $F_\gamma$ is Lipschitz continuous and $F_\gamma'(u) = F_{\gamma-1}(u)$ (weakly if $\gamma=1$).

	If $m\in \N$, $m-1 <\gamma <m$ and $\theta = \gamma+1-m$, then an application of Fubini's theorem yields
	\begin{equation*}
		F_{\gamma}(u) =  \frac{1}{\Gamma(\theta)}\int_0^u (u-v)^{\theta-1}F_{m-1}(v)\,dv = \frac{1}{\Gamma(\theta)}\int_0^u (u-v)^{\theta-1}F'_{m}(v)\,dv\,.
	\end{equation*}
	Assume that $0\leq u_1 \leq u$, for $u_1$ to be chosen. Then, by splitting the integral and an integration by parts (using that $F_m$ is continuous and $F_m(0)=0$),
	\begin{equation}\label{eq: Tauberian integration by parts bound}
	\begin{aligned}
		|\Gamma(\theta)F_{\gamma}(u)| &\leq \biggl|\int_0^{u_1} (u-v)^{\theta-1}F'_{m}(v)\,dv\biggr| +\biggl|\int_{u_1}^u (u-v)^{\theta-1}F_{m-1}(v)\,dv\biggr|\\
		&\leq |(u-u_1)^{\theta-1}F_m(u_1)|+(1-\theta)\biggl|\int_0^{u_1}(u-v)^{\theta-2}F_m(v)\,dv\biggr|\\
		&\quad  + \biggl|\int_{u_1}^u (u-v)^{\theta-1}F_{m-1}(v)\,dv\biggr|\\
		&\leq (u-u_1)^{\theta-1}|F_m(u_1)|+(1-\theta)\biggl|\int_0^{u_1}(u-v)^{\theta-2}\,dv\biggr|\sup_{v\in [0, u_1]}|F_m(v)|\\
		&\quad  + \biggl|\int_{u_1}^u (u-v)^{\theta-1}\,dv\biggr|\sup_{v \in [u_1, u]}|F_{m-1}(v)|\\
		&= u^{\theta-1}(1-u_1/u)^{\theta-1}|F_m(u_1)|+u^{\theta-1}((1-u_1/u)^{\theta-1}-1)\sup_{v\in [0, u_1]}|F_m(v)| \hspace{-12pt}\\
		&\quad  + \frac{u^{\theta}(1-u_1/u)^\theta}{\theta}\sup_{v \in [u_1, u]}|F_{m-1}(v)|\,.
	\end{aligned}
	\end{equation}

	By the nonnegativity of $\tilde \mu$ the function
	\begin{equation*}
		u \mapsto \frac{1}{\Gamma(1+\gamma)}\int_{[0, u)} (1-v/u)^\gamma \,d\tilde\mu(v) = u^{-\gamma} F_\gamma(u) + \frac{K_0}{\Gamma(1+\gamma)}+ \sum_{i=1}^N\frac{K_i \nu_i \Gamma(\nu_i)}{\Gamma(\gamma+\nu_i+1)}u^{\nu_i} 
	\end{equation*}
	is nondecreasing and nonnegative. Therefore, $u \mapsto F_\gamma(u)+ \sum_{i=1}^N\frac{K_i \nu_i \Gamma(\nu_i)}{\Gamma(\gamma+\nu_i+1)}u^{\nu_i+\gamma}$ is nondecreasing and nonnegative (as the product of two nonnegative nondecreasing functions). Thus, we can estimate for any $0\leq u_2 \leq u_1$
	\begin{align*}
		\sup_{v \in [0, u_1]}|F_{m}(v)| &\leq \sup_{v\in [0, u_2]}\Bigl|F_{m}(v)+\sum_{i=1}^N\frac{K_i\nu_i\Gamma(\nu_i)}{\Gamma(m+\nu_i+1)}v^{\nu_i+m}-\sum_{i=1}^N\frac{K_i\nu_i\Gamma(\nu_i)}{\Gamma(m+\nu_i+1)}v^{\nu_i+m}\Bigr|\\
		&\quad + \sup_{v \in [u_2, u_1]}|F_{m}(v)|\\
		&\leq \sup_{v\in [0, u_2]}\Bigl|F_{m}(v)+\sum_{i=1}^N\frac{K_i\nu_i\Gamma(\nu_i)}{\Gamma(m+\nu_i+1)}v^{\nu_i+m}\Bigr|\\
        &\quad+\sup_{v\in [0, u_2]}\Bigl|\sum_{i=1}^N\frac{K_i\nu_i\Gamma(\nu_i)}{\Gamma(m+\nu_i+1)}v^{\nu_i+m}\Bigr| +\sup_{v \in [u_2, u_1]}|F_{m}(v)|\\
		&= F_{m}(u_2)+2\sum_{i=1}^N\frac{K_i\nu_i\Gamma(\nu_i)}{\Gamma(m+\nu_i+1)}u_2^{\nu_i+m}+\sup_{v \in [u_2, u_1]}|F_{m}(v)|\\
		&\leq 2\sum_{i=1}^N\frac{K_i\nu_i\Gamma(\nu_i)}{\Gamma(m+\nu_i+1)}u_2^{\nu_i+m}+2\sup_{v \in [u_2, u_1]}|F_{m}(v)|\,.
	\end{align*}
	Inserted into~\eqref{eq: Tauberian integration by parts bound}, this yields
	\begin{equation*}
	\begin{aligned}
		|\Gamma(\theta)F_{\gamma}(u)| &\leq u^{\theta-1}(1-u_1/u)^{\theta-1}|F_m(u_1)|+2u^{\theta-1}((1-u_1/u)^{\theta-1}-1)\sup_{v\in [u_2, u_1]}|F_m(v)|\\
		&\quad  + \frac{u^{\theta}(1-u_1/u)^\theta}{\theta}\sup_{v \in [u_1, u]}|F_{m-1}(v)|\\
		&\quad + 2u^{\theta-1}((1-u_1/u)^{\theta-1}-1)\sum_{i=1}^N\frac{K_i\nu_i\Gamma(\nu_i)}{\Gamma(m+\nu_i+1)}u_2^{\nu_i+m}\\
		&\lesssim_{\gamma}
		u^{\theta-1}(1-u_1/u)^{\theta-1}\sup_{v\in [u_2, u_1]}|F_m(v)| + u^{\theta}(1-u_1/u)^\theta\sup_{v \in [u_1, u]}|F_{m-1}(v)|\\
		&\quad + u^{\theta-1}(1-u_1/u)^{\theta-1}\sum_{i=1}^N|K_i|u_2^{\nu_i+m}\,.
	\end{aligned}
	\end{equation*}

	Since
	\begin{equation*}
		F_\kappa(u) = \frac{u^\kappa}{\Gamma(1+\kappa)}\int_{[0, u)} (1-v/u)^{\kappa}\,d\mu(v)
	\end{equation*}
	we can write this as
	\begin{align*}
		\biggl|\int_{[0, u)} (1-v/u)^\gamma\,d\mu(v)\biggr| &\lesssim_\gamma u^{-m}(1-u_1/u)^{\theta-1}\sup_{v\in [u_2, u_1]}\Bigl|v^m\int_{[0, v)}(1-s/v)^m\,d\mu(s)\Bigr| \\
		&\quad + u^{-m+1}(1-u_1/u)^\theta\sup_{v \in [u_1, u]}\Bigl|v^{m-1}\int_{[0, v)}(1-s/v)^{m-1}\,d\mu(s)\Bigr|\\
		&\quad + u^{-m}(1-u_1/u)^{\theta-1}\sum_{i=1}^N|K_i|u_2^{\nu_i+m}
	\end{align*}

	By Lemma~\ref{lem: Tauberian v1} applied with $\gamma=m$ and $\gamma=m-1$, there exist $B_m, B_{m-1}, k_m, k_{m-1}$ so that for all $k_1\geq k_m, k_2\geq k_{m-1}$,
	\allowdisplaybreaks
	\begin{align*}
		\biggl|\int_{[0, u)}& (1-v/u)^\gamma\,d\mu(v)\biggr|\\
		 &\lesssim_\gamma u^{-m}(1-u_1/u)^{\theta-1}\sup_{v\in [u_2, u_1]}v^m\biggl(B_m^{k_1}\max_{j=1, \ldots, k_1} |Q(j/v)|+ \frac{1}{k_1^{m+1}}\sum_{i=1}^N |K_i|v^{\nu_i}\biggr) \\
		&\quad + u^{-m+1}(1-u_1/u)^\theta\sup_{v \in [u_1, u]}v^{m-1}\biggl(B_{m-1}^{k_2}\max_{j=1, \ldots, k_2} |Q(j/v)|+ \frac{1}{k_2^{m}}\sum_{i=1}^N |K_i|v^{\nu_i}\biggr)\\
		&\quad + u^{-m}(1-u_1/u)^{\theta-1}\sum_{i=1}^N|K_i|u_2^{\nu_i+m}\\
		&\lesssim_\gamma (1-u_1/u)^{\theta-1}\sup_{\delta_1\in [\frac{u_2}u, \frac{u_1}u]}\delta_1^m\biggl(B_m^{k_1}\max_{j=1, \ldots, k_1} |Q(j/(\delta_1 u))|+ \frac{1}{k_1^{m+1}}\sum_{i=1}^N |K_i|(\delta_1u)^{\nu_i}\biggr) \\
		&\quad + (1-u_1/u)^\theta\sup_{\delta_2 \in [\frac{u_1}u, 1]}\delta_2^{m-1}\biggl(B_{m-1}^{k_2}\max_{j=1, \ldots, k_2} |Q(j/(\delta_2u))|+ \frac{1}{k_2^{m}}\sum_{i=1}^N |K_i|(\delta_2u)^{\nu_i}\biggr)\\
		&\quad + (1-u_1/u)^{\theta-1}\Bigl(\frac{u_2}{u}\Bigr)^m\sum_{i=1}^N|K_i|u_2^{\nu_i}\,.
	\end{align*}

	Choosing $k_1 = \delta_1 k$ and $k_2 = \delta_2 k$ with $k \geq \max\{k_m/\delta_1, k_{m-1}/\delta_2\}$ for each $\delta_1, \delta_2$ in the appropriate ranges we arrive at
	\begin{align*}
		\biggl|&\int_{[0, u)} (1-v/u)^\gamma\,d\mu(v)\biggr|\\
		 &\quad\lesssim_\gamma 
		 (1-u_1/u)^{\theta-1}\sup_{\delta_1\in [\frac{u_2}u, \frac{u_1}u]}\delta_1^m\biggl(B_m^{\delta_1k}\sup_{\delta_1^{-1}\leq s\leq k} |Q(s/u)|+ \frac{1}{\delta_1^{m+1}k^{m+1}}\sum_{i=1}^N |K_i|(\delta_1u)^{\nu_i}\biggr) \\
		&\qquad + (1-u_1/u)^\theta\sup_{\delta_2 \in [\frac{u_1}u, 1]}\delta_2^{m-1}\biggl(B_{m-1}^{\delta_2k}\sup_{\delta_2^{-1}\leq s\leq k} |Q(s/u)|+ \frac{1}{\delta_2^mk^{m}}\sum_{i=1}^N |K_i|(\delta_2u)^{\nu_i}\biggr)\\
		&\qquad + (1-u_1/u)^{\theta-1}\Bigl(\frac{u_2}{u}\Bigr)^m\sum_{i=1}^N|K_i|u_2^{\nu_i}\,.
	\end{align*}
	Further choosing $u_1 = u- u/k$ and $u_2=k_m u/k$ which is ok if $k \geq k_{m-1}+1$,  we find
	\begin{align*}
		\biggl|\int_{[0, u)}& (1-v/u)^\gamma\,d\mu(v)\biggr|\\
		 &\lesssim_\gamma 
		 \sup_{\delta_1\in [k_m/k, 1-1/k]}\biggl(k^{1-\theta}\delta_1^mB_m^{\delta_1k}\sup_{\delta_1^{-1}\leq s\leq k} |Q(s/u)|+ \frac{1}{\delta_1k^{\gamma+1}}\sum_{i=1}^N |K_i|(\delta_1u)^{\nu_i}\biggr) \\
		&\quad + \sup_{\delta_2 \in [1-1/k, 1]}\biggl(k^{-\theta}\delta_2^{m-1}B_{m-1}^{\delta_2k}\sup_{\delta_2^{-1}\leq s\leq k} |Q(s/u)|+ \frac{1}{\delta_2k^{\gamma+1}}\sum_{i=1}^N |K_i|(\delta_2u)^{\nu_i}\biggr)\\
		&\quad + \frac{1}{k^\gamma}\sum_{i=1}^N|K_i|(u/k)^{\nu_i}\\
		&\lesssim_\gamma 
		 \Bigl[k^{1-\theta}B_m^{k}+k^{-\theta}B_{m-1}^{k}\Bigr]\sup_{1\leq s\leq k} |Q(s/u)|+ \frac{1}{k^{\gamma+1}}\sum_{i:\nu_i \geq 1} |K_i|u^{\nu_i}\\
		 &\quad +\frac{1}{k^{\gamma}}\sum_{i: \nu_i< 1} |K_i|(u/k)^{\nu_i}\\
		 &\lesssim_\gamma 
		 B^k\sup_{1\leq s\leq k} |Q(s/u)|+ \frac{1}{k^{\gamma+1}}\sum_{i:\nu_i \geq 1} |K_i|u^{\nu_i}+\frac{1}{k^{\gamma}}\sum_{i: \nu_i< 1} |K_i|(u/k)^{\nu_i}
	\end{align*}
	This completes the proof of Proposition~\ref{prop: Tauberian v2}.
\end{proof}

\begin{proof}[Proof of Corollary~\ref{cor: Tauberian cor}]
	By Proposition~\ref{prop: Tauberian v2} there exist constants $B, C, k_0$ such that for all $k \geq k_0$ and $u>0$
	\begin{align*}
		\biggl|\int_{[0, u)} (1-v/u)^\gamma d\mu(v)\biggr| &\leq B^k \sup_{1\leq s\leq k}\biggl|\int_{[0, \infty)} e^{-\frac{s v}u}\,d\mu(v)\biggr| + \frac{C}{k^{\gamma+1}}\sum_{i: \nu_i\geq 1} |K_i| u^{\nu_i}\\
        &\quad + \frac{C}{k^{\gamma}}\sum_{i: \nu_i< 1} |K_i| (u/k)^{\nu_i}\,.
    \end{align*}
    If we choose $k$ satisfying $k/u \leq c_2A$ then~\eqref{eq: exp Laplace transform bound} implies that
    \begin{equation}\label{eq: exp Laplace transform bound 2}
    \begin{aligned}
		\biggl|\int_{[0, u)} (1-v/u)^\gamma d\mu(v)\biggr| &\leq B^k c_1\sup_{1\leq s\leq k} e^{-(c_2u/s)^\epsilon} + \frac{C}{k^{\gamma+1}}\sum_{i: \nu_i\geq 1} |K_i| u^{\nu_i}\\
        &\quad+ \frac{C}{k^{\gamma}}\sum_{i: \nu_i< 1} |K_i| (u/k)^{\nu_i}\\
        &= B^k c_1 e^{-(c_2u/k)^\epsilon} + \frac{C}{k^{\gamma+1}}\sum_{i: \nu_i\geq 1} |K_i| u^{\nu_i}\\
        &\quad + \frac{C}{k^{\gamma}}\sum_{i: \nu_i< 1} |K_i| (u/k)^{\nu_i}\,,
	\end{aligned}
    \end{equation}
	where we used the fact that $(0, \infty) \ni x \mapsto e^{-(c_2 x)^\epsilon}$ is a decreasing function and so the supremum is attained at $s =k$. 
 
    To prove the claimed bound we want to choose $k = \max\{k_0, \lfloor (2\ln(B)^{-\frac{1}{1+\epsilon}}) (c_2u)^{\frac{\epsilon}{1+\epsilon}}\rfloor\}$ for an appropriately chosen $c'$. To justify this particular choice of $k$ we note that 
    \begin{equation*}
        \frac{k}{u} \leq \max\biggl\{\frac{k_0}{u}, \frac{c_2^{\frac{\epsilon}{1+\epsilon}}}{(2\ln(B))^{\frac{1}{1+\epsilon}}u^{\frac{1}{1+\epsilon}}}\biggr\}
    \end{equation*}
    and so $k/u \leq c_2A$ if 
    \begin{equation*}
        u \geq \frac{1}{c_2A}\max\Bigl\{k_0, \frac{1}{2A^{\epsilon}\ln(B)}\Bigr\}\,.
    \end{equation*}
    Thus the choice is justified for all $u \geq \frac{D}{c_2A}(1+A^{-\epsilon})$ for any $D \geq \max\bigl\{k_0, \frac{1}{2\ln(B)}\bigr\}$. In particular, if the bound on the Laplace transform of $\mu$ is valid for all $t>0$ our choice of $k$ is justified for all $u>0$.
    
    Plugging the choice of $k$ into~\eqref{eq: exp Laplace transform bound 2} yields,
	\begin{align*}
		\biggl|\int_{[0, u)} (1-v/u)^\gamma d\mu(v)\biggr| &\lesssim_{\epsilon, \gamma, \{\nu_i\}} (1+c_2 u)^{-\frac{(1+\gamma)\epsilon}{1+\epsilon}}\Bigl(c_1+ \sum_{i: \nu_i \geq 1}|K_i|u^{\nu_i}\\
  &\quad + \sum_{i:\nu_i<1} |K_i|u^{\nu_i}(c_2u)^{\frac{(1-\nu_i)\epsilon}{1+\epsilon}}\Bigr)\,,
	\end{align*}
    for all $u > \frac{D}{c_2A}(1+A^{-\epsilon})$.
\end{proof}


\section{A Tauberian theorem for the Fourier transform}\label{sec: Fourier tauber}

Let $\mu$ be a locally finite signed Borel measure on $\R$. Assume that $T_\mu$, the distribution associated with $\mu$ through
\begin{equation*}
	T_\mu(\psi) = \int_\R \psi(\tau)\,d\mu(\tau), \quad \psi\in \mathcal{S}(\R),
\end{equation*}
is a temperate distribution. There exists a unique function $N_\mu \in BV_{loc}(\R)$ satisfying that
\begin{enumerate}
 	\item the distributional derivative of $N_\mu$ is $T_\mu$,
 	\item $N_\mu(0)=0$, and 
 	\item for all $\tau \in \R$
 	\begin{equation*}
	\lim_{\epsilon \to 0} \frac{N_\mu(\tau+\epsilon)+N_\mu(\tau-\epsilon)}{2} = N_\mu(\tau)\,.
\end{equation*}
\end{enumerate}

Explicitly, we can write the \emph{regular distribution function} $N_\mu$ as
\begin{equation*}
	N_\mu(\tau) = 
    \begin{cases}
        \frac{\mu([0, \tau])+\mu([0, \tau))}{2} & \text{if}\ \tau> 0 \,,\\
        0   & \text{if}\ \tau=0 \,,\\
        - \frac{\mu([\tau, 0])+\mu((\tau, 0])}{2} & \text{if}\ \tau< 0 \,.
    \end{cases}
\end{equation*}

We are interested in the asymptotic behaviour of $N_\mu(\tau)$ as $\tau \to \infty$ and also in the asymptotics of the quantities
\begin{equation*}
	R_\mu^\gamma(\tau) := \frac{2\gamma}{\tau}\int_0^\tau \Bigl(1- \frac{\sigma^2}{\tau^2}\Bigr)^{\gamma-1}\frac{\sigma}{\tau} N_\mu(\sigma)\,d\sigma
\end{equation*}
for $\gamma >0$. Note that we can also write this as
\begin{equation*}
	R_\mu^\gamma(\tau)=\int_{[0, \tau]} \Bigl(1- \frac{\sigma^2}{\tau^2}\Bigr)^\gamma \,d\mu(\sigma)\,.
\end{equation*}
If $f(x) = N_\mu(\sqrt{x})$ for $x\geq 0$, then $R_\mu^\gamma(\tau)$ is the $\gamma$-Riesz mean of $f$ evaluated at $\tau^2$, up to multiplication by an explicit factor; specifically, $R_\mu^\gamma(\tau) = \Gamma(\gamma+1)\tau^{-2\gamma}f^{(\gamma)}(\tau^2)$ in the notation of Section \ref{sec: Riesz convexity}.

Our aim is to show that the asymptotic behaviour of $N_\mu$ or $R_\mu^\gamma$ can be computed in terms of the corresponding quantity for a second measure $\nu$ under the assumption that the distributional Fourier transforms of $T_\mu$ and $T_\nu$ agree on an interval around zero, or more generally, that their difference is small on this interval.

Fix a function $\phi \in \mathcal{S}(\R)$ such that
\begin{enumerate}
	\item $\phi \geq 0$,
	\item $\phi$ is even,
	\item $\supp \hat \phi \subseteq [-1, 1]$,
	\item $\int_\R \phi(\tau)\,d\tau = 1$.
\end{enumerate}
For $a>0$ define $\phi_a$ by $\phi_a(\tau) := a^{-1}\phi(\tau/a)$. Note that $\hat \phi_a(t) = \hat \phi(a t)$ and consequently $\supp \hat \phi_a \subset [-1/a, 1/a]$. In particular, if $\hat T_\mu = \hat T_\nu$ on $(-1/a, 1/a)$, then the above assumptions on $\phi$ imply that $\phi_a*T_\mu = \phi_a*T_\nu$.

Our aim in this section is to prove the following result:

\begin{theorem}\label{thm: Fourier Tauberian}
	Let $\mu, \nu$ be two signed measures satisfying that
	\begin{enumerate}
		\item $N_\mu$ is nondecreasing,
		\item the distribution functions $N_\mu, N_\nu$ are odd,
		\item $|T_\nu(\psi)| \leq M_0 \int_\R |\psi(\tau)|(a_0+|\tau|)^{\alpha}\,d\tau$ for all $\psi\in C_0^\infty(\R)$,
	 	\item $|\phi_a*(T_\mu-T_\nu)(\tau)|\leq M_1(a_1+|\tau|)^\beta$ for all $\tau\in \R$, and
	\end{enumerate} 
    for some $\alpha \geq \beta \geq 0$ and $a_0, a_1 \geq a$.
	If $m \in \N$ there exists a constant $C$ depending only on $m, \phi, \alpha, \beta$ such that for all $\tau \geq 0$
    \begin{align*}
	   \bigl|R_{\mu}^m(\tau) -R_{\nu}^m(\tau)\bigr| \leq C(M_0a^{1+m}(a_0+|\tau|)^{\alpha-m}+M_1 (a+|\tau|)(a_1+|\tau|)^{\beta})\,.
    \end{align*}
\end{theorem}
\begin{remark}
    The case $m=0$ of Theorem~\ref{thm: Fourier Tauberian} is the content of \cite[Lemma 17.5.6]{Hormander_bookIII}. The special case where $m=1$ and where $\nu$ is absolutely continuous with a monomial density is proved in \cite{Safarov01}, disregarding the issue discussed in Remark \ref{safarovrem}.
\end{remark}

\begin{remark}
	For us the important case is when $\hat T_\mu = \hat T_\nu$ on $(-1/a,1/a)$. In this case assumption (4) in Theorem \ref{thm: Fourier Tauberian} holds with $M_1=0$.
\end{remark}

By combining Theorem \ref{thm: Fourier Tauberian} with Riesz's concavity theorem (Proposition~\ref{prop: Riesz logconvexity}) we can also get a bound for non-integer $\gamma$.

\begin{corollary}\label{cor: Fourier Tauberian general gamma}
    Let $\mu, \nu$ be as in Theorem~\ref{thm: Fourier Tauberian}. For any $\gamma\geq 0$ there exists $C$ depending only on $\gamma, \phi, \alpha, \beta$ such that for all $\tau \geq 0$
    \begin{align*}
	   \bigl|R_{\mu}^\gamma(\tau) &-R_{\nu}^\gamma(\tau)\bigr|\\ &\leq C\Bigl(M_0a^{1+\gamma}(a_0+|\tau|)^{\alpha-\gamma}+M_1 (a+|\tau|)(a_1+|\tau|)^{\beta}\\
        &\quad + (M_0a^{1+\gamma}(a_0+\tau)^{\alpha-\gamma})^{1-\{\gamma\}} (M_1 (a+\tau)(a_1+\tau)^\beta)^{\{\gamma\}}(a^{-1}(a_0+\tau))^{\{\gamma\}(1-\{\gamma\})}\Bigr)\,,
    \end{align*}
    where $\{\gamma\}$ denotes the fractional part of $\gamma$, that is the unique number in $[0, 1)$ so that $\gamma -\{\gamma\} \in \Z$.
\end{corollary}
\begin{remark}
    If $\gamma \notin (\lfloor \alpha-\beta-1\rfloor, \lceil \alpha-\beta-1\rceil)$ then the third term in the bound is insignificant in the limit $\tau \to \infty$. While we believe that the first two terms are necessary, it is unclear to us if one might argue differently and prove a bound for $\gamma \notin\N$ without this additional third term (corresponding to the bound for $\gamma \in \N$ of Theorem~\ref{thm: Fourier Tauberian}). Since in this paper we will apply Corollary \ref{cor: Fourier Tauberian general gamma} only with $M_1=0$, where the third term vanishes, we did not explore this further.
\end{remark}

\begin{proof}[Proof of Corollary~\ref{cor: Fourier Tauberian general gamma}] 
    If $\gamma \in \N\cup \{0\}$ the claimed bound is exactly Theorem~\ref{thm: Fourier Tauberian}. Therefore, without loss of generality we assume that $\gamma \notin \N$.

    Set $f(s) := N_\mu(\sqrt{s})-N_\nu(\sqrt{s})$. As observed above, a change of variables in the integrals defining $R_\mu^\gamma$ and $R_\nu^\gamma$ implies that
    \begin{equation*}
        R^\gamma_\mu(\tau)-R_\nu^\gamma(\tau) = \frac{\Gamma(\gamma+1)}{\tau^{2\gamma}} f^{(\gamma)}(\tau^2)\,.
    \end{equation*}
    In particular, for any $m \in \N$ Theorem~\ref{thm: Fourier Tauberian} implies that
    \begin{equation*}
        |f^{(m)}(s)| \leq C s^m (M_0 a^{1+m}(a_0+\sqrt{s})^{\alpha-m}+M_1 (a+\sqrt{s})(a_1+\sqrt{s})^\beta)
    \end{equation*}
    with $C$ depending only on $m, \phi, \alpha, \beta$. 
    
    Write $\gamma = m + \kappa$ with $m \in \N \cup \{0\}, \kappa \in [0, 1)$ (i.e.\ $\kappa = \{\gamma\}$). By Proposition~\ref{prop: Riesz logconvexity} and the semigroup property of Riesz means
    \begin{align*}
        |R^\gamma_\mu(\tau)-&R_\nu^\gamma(\tau)|\\ &=  \frac{\Gamma(\gamma+1)}{\tau^{2\gamma}}|f^{(\gamma)}(\tau^2)|\\
        &\leq  \frac{C}{\tau^{2\gamma}}\biggl[\sup_{0\leq s\leq \tau^2}|f^{(m)}(s)|\biggr]^{1-\kappa}\biggl[\sup_{0\leq s\leq \tau^2}|f^{(m+1)}(s)|\biggr]^{\kappa}\\
        &\leq   \frac{C}{\tau^{2\gamma}}\biggl[\sup_{0\leq s\leq \tau^2}\Bigl(s^{m} (M_0 a^{1+m}(a_0+\sqrt{s})^{\alpha-m}+M_1 (a+\sqrt{s})(a_1+\sqrt{s})^\beta)\Bigr)\biggr]^{1-\kappa}\\
        &\quad \times \biggl[\sup_{0\leq s\leq \tau^2}\Bigl(s^{m+1} (M_0 a^{m+2}(a_0+\sqrt{s})^{\alpha-m-1}+M_1 (a+\sqrt{s})(a_1+\sqrt{s})^\beta)\Bigr)\biggr]^{\kappa}\\
        &= C\Bigl(M_0 a^{1+m}(a_0+\tau)^{\alpha-m}+M_1 (a+\tau)(a_1+\tau)^\beta\Bigr)^{1-\kappa}\\
        &\quad \times \Bigl( M_0 a^{2+m}(a_0+\tau)^{\alpha-m-1}+M_1 (a+\tau)(a_1+\tau)^\beta\Bigr)^{\kappa}\,,
    \end{align*}
    where $C$ depends on $\gamma, \phi, \alpha, \beta$ and the final step uses the fact that $s \mapsto s^k (a_0+\sqrt{s})^{\alpha-k}$ is monotone increasing if $k \in \N\cup \{0\}$ and $\alpha \geq 0$. 
    
    In particular, the previous bound implies that
    \begin{align*}
        |R^\gamma_\mu(\tau)-R_\nu^\gamma(\tau)| 
        &\leq  C \max\Bigl\{M_0 a^{1+m}(a_0+\tau)^{\alpha-m}, M_1 (a+\tau)(a_1+\tau)^\beta\Bigr\}^{1-\kappa}\\
        &\quad \times \max\Bigl\{ M_0 a^{1+m}(a_0+\tau)^{\alpha-m}a (a_0+\tau)^{-1}, M_1 (a+\tau)(a_1+\tau)^\beta\Bigr\}^{\kappa}\,.
    \end{align*}
    Since $a_0\geq a, \tau \geq 0$ the factor $a(a_0+\tau)^{-1}\leq 1$ and so if the second maxima is attained by the first quantity the same holds for the first maxima. Therefore, 
    \begin{align*}
        |R^\gamma_\mu(\tau)-R_\nu^\gamma(\tau)| 
        &\leq  C \Bigl(M_0 a^{1+\gamma}(a_0+\tau)^{\alpha-\gamma}+ M_1 (a+\tau)(a_1+\tau)^\beta\\
        & +(M_0a^{1+\gamma}(a_0+\tau)^{\alpha-\gamma})^{1-\kappa} (M_1 (a+\tau)(a_1+\tau)^\beta)^{\kappa}(a^{-1}(a_0+\tau))^{\kappa(1-\kappa)} \Bigr).
    \end{align*}
    This completes the proof of the corollary.
\end{proof}

We now turn our attention to the proof of Theorem~\ref{thm: Fourier Tauberian}. One of the main ingredients in its proof is an integral identity that is recorded in Lemma~\ref{lem: Integration by parts identity}. Our aim is to apply this identity to the integral defining $R_\mu^\gamma$. In order to do this, we need to first introduce some auxiliary functions that appear in the analysis.

Fix $\chi \in C_0^\infty(\R)$ with $\chi \geq 0$, $\chi$ even, $\int_\R \chi(\tau)\,d\tau =1$, and $\supp \chi \subseteq [-1, 1]$. For $\epsilon>0$ define
\begin{equation*}
	\chi_\epsilon(\tau) := \epsilon^{-1}\chi\Bigl(\frac{\tau}{\epsilon}\Bigr)\,.
\end{equation*}

For $\epsilon>0$ define
\begin{equation*}
	\phi_{0, \epsilon}(\tau) := \phi(\tau)\,.
\end{equation*}
For $k \geq 1, \epsilon>0$ define $\phi_{k, \epsilon}\in \mathcal{S}(\R)$ by the property
\begin{equation*}
 	\phi_{k, \epsilon}'(\tau) := \phi_{k-1, \epsilon}(\tau) - \Bigl(\int_\R \phi_{k-1,\epsilon}(\sigma)\,d\sigma \Bigr) \chi_\epsilon(\tau) \,,
 \end{equation*} 
which is equivalent to setting
\begin{equation*}
	\phi_{k, \epsilon}(\tau) := \int_{-\infty}^\tau \Bigl(\phi_{k-1, \epsilon}(\sigma)-\Bigl(\int_\R \phi_{k-1, \epsilon}(\sigma')\,d\sigma' \Bigr)\chi_\epsilon(\sigma)\Bigr)\,d\sigma \,.
\end{equation*}

Since $\phi, \chi$ are assumed to be even, it follows that $\phi_{k, \epsilon}$ is odd for odd $k$ and even for even $k$. In particular, for even $k$ the defining equation simplifies to $\phi_{k,\epsilon}' = \phi_{k-1, \epsilon}$.

With these definitions we record the following lemma. When $m=1$ the statement (after taking the limit $\epsilon\to 0^\limplus$) corresponds to \cite[Theorem 1.4]{Safarov01}. Our new observation is that this idea can be iterated.

\begin{lemma}\label{lem: Integration by parts identity}
	Let $\{\phi_{k,\epsilon}\}_{k\geq 0}, \chi_\epsilon$ be defined as above. Define recursively
	\begin{equation*}
		b_0 := 1 \quad \mbox{and} \quad b_m := (-1)^{m-1} \sum_{\substack{j=0\\j \text{ even}}}^{m-1} b_j \int_\R \phi_{m-j}(\sigma)\,d\sigma \quad \mbox{for }m\in \N\,.
	\end{equation*}
    Then $b_m=0$ when $m$ is odd. Moreover, the following holds for every $m\in\N_0$ and $\tau>0$. For any $u \in \mathcal{S}'(\R)$ and $f \in C^m(0, \tau)$ for which the limits
	\begin{equation*}
	f^{(j)}(0^+):=\lim_{\sigma \to 0^+}f^{(j)}(\sigma) \quad \mbox{and}\quad f^{(j)}(\tau^-):=\lim_{\sigma \to \tau^-}f^{(j)}(\sigma)
	\end{equation*}
	exist for all $0 \leq j \leq m-1$, it holds that
	\begin{align*}
		\int_0^\tau f(\sigma)&\chi_\epsilon*u(\sigma)\,d\sigma\\ 
		&=
		\sum_{\substack{j=0\\j \text{ even}}}^m b_j \int_0^\tau f^{(j)}(\sigma)\phi_{0,\epsilon}*u(\sigma)\,d\sigma\\
		&\quad 
		-(-1)^{m}\sum_{\substack{j=0\\j \text{ even}}}^{m} b_{j} \int_0^\tau f^{(m)}(\sigma)\phi_{m+1-j,\epsilon}*u'(\sigma)\,d\sigma\\
		&\quad 
		- \sum_{k=0}^{m-1}\sum_{\substack{j=0\\j \text{ even}}}^{k}(-1)^{k}b_{j} \Bigl[f^{(k)}(\tau^-)\phi_{k+1-j,\epsilon}*u(\tau)-f^{(k)}(0^+)\phi_{k+1-j,\epsilon}*u(0)\Bigr]\,.
	\end{align*}
\end{lemma}

In the statement of the lemma, we use the convention that $\sum_0^{m-1}\ldots = 0$ when $m=0$. Similarly, the assumption on $f^{(j)}$ for $0\leq j\leq m-1$ is void in this case.

\begin{remark}\label{rem:bkexplicit}
    The recursion defining $b_m$ can be solved and we find
    $$
    b_m := \sum_{j=1}^m (-1)^j \sum_{\substack{(k_1, \ldots, k_j)\in \N^j\\ \sum_{i=1}^j k_i = m}}\prod_{i=1}^j\biggl(\int_\R \phi_{k_i, \epsilon}(\sigma)\,d\sigma \biggr)\quad \mbox{for }m\in \N\,.
    $$
\end{remark}

We defer the proof of Lemma~\ref{lem: Integration by parts identity} to the following subsection and continue towards the proof of Theorem \ref{thm: Fourier Tauberian}.

Define
$$
G_\gamma(\tau) := (1-\tau^2)^{\gamma-1}\tau 
\qquad\text{for}\ \tau\in(-1,1) \,.
$$
Note that $G_\gamma \colon (-1, 1)\to \R$ is an odd function that is continuously differentiable to all orders on $(-1, 1)$ with
\begin{equation*}
	G_\gamma(0)=0 \quad \mbox{and} \quad G_\gamma^{(2j)}(0) = 0 \quad \mbox{for all }j \in \N\,,
\end{equation*}
and furthermore
\begin{equation*}
	 \lim_{\tau \to 1^-}G_\gamma^{(j)}(\tau) =0 \quad \mbox{for all }j \in \N \cap \bigl[0, \gamma-1\bigr)\,.
\end{equation*}
For higher derivatives there is a singularity as $\tau \to 1^-$ unless $\gamma$ is an integer.
For $j = \gamma-1 \in \N$ we have $\lim_{\tau\to 1^-}G_\gamma^{(j)}(\tau)=(-2)^{j}j! = (-2)^{\gamma-1}(\gamma-1)!$.

The relevance of the function $G$ for us is that
\begin{equation*}
	R_\mu^\gamma(\tau) = \frac{2\gamma}{\tau}\int_0^\tau G_\gamma(\sigma/\tau)N_\mu(\sigma)\,d\sigma\,.
\end{equation*}
For technical reasons we will rather work with the smoothed quantities
\begin{equation*}
	R_{\mu,\epsilon}^\gamma(\tau) := \frac{2\gamma}{\tau}\int_0^\tau G_\gamma(\sigma/\tau)\chi_\epsilon*N_\mu(\sigma)\,d\sigma\,.
\end{equation*}
Note that in the sense of distributions $\chi_\epsilon \to \delta_0$ as $\epsilon \to 0$. In particular, the fact that $\lim_{\epsilon\to 0}\chi_\epsilon *N_\mu(\tau) = N_\mu(\tau)$ for all $\tau$, implies
$\lim_{\epsilon\to 0} R_{\mu,\epsilon}^\gamma(\tau) = R_{\mu}^\gamma(\tau)$.

The following identity will follow from Lemma \ref{lem: Integration by parts identity}, applied to the integral defining $R_{\mu, \epsilon}^m$.

\begin{proposition}\label{prop: itterated Safarov identity}
	Let $\mu$ be a signed measures such that the distribution function $N_\mu$ is odd. Then, with $\{\phi_{k, \epsilon}\}_{k\geq 0}$ defined as above and $m \in \N$,
	\begin{align*}
	R_{\mu, \epsilon}^m(\tau)
		&= 
		2m\sum_{\substack{j=0\\j \text{ even}}}^m b_j \int_0^\tau \tau^{-j-1}G_m^{(j)}(\sigma/\tau)\phi_{0,\epsilon}*N_\mu(\sigma)\,d\sigma\\
	&\quad - 2m\sum_{\substack{j=0\\j \text{ even}}}^{m} (-1)^{m}b_{j} \int_0^\tau \tau^{-m-1}G_m^{(m)}(\sigma/\tau)\phi_{m+1-j,\epsilon}*T_\mu(\sigma)\,d\sigma\\
		&\quad - \sum_{\substack{j=0\\ j\text{ even}}}^{m-1}2^{m}m!b_j \tau^{-m}\phi_{m-j,\epsilon}*N_\mu(\tau) \,.
\end{align*}
\end{proposition}

\begin{proof}[Proof of Proposition~\ref{prop: itterated Safarov identity}]
Fix $\gamma =m \in \N$, $\epsilon, \tau>0$. By Lemma~\ref{lem: Integration by parts identity}, applied with $f(\sigma)=G_m(\sigma/\tau)$ and $u(\sigma)=N_\mu(\sigma)$, we have
\begin{align*}
	\frac{\tau}{2m}R_{\mu, \epsilon}^m(\tau) &=
	\sum_{\substack{j=0\\j \text{ even}}}^m b_j \int_0^\tau \tau^{-j}G_m^{(j)}(\sigma/\tau)\phi_{0,\epsilon}*N_\mu(\sigma)\,d\sigma\\
		&\quad 
		- \sum_{\substack{j=0\\j \text{ even}}}^{m} (-1)^{m}b_{j} \int_0^\tau \tau^{-m}G_m^{(m)}(\sigma/\tau)\phi_{m+1-j,\epsilon}*T_\mu(\sigma)\,d\sigma\\
		&\quad - \sum_{\substack{j=0\\ j\text{ even}}}^{m-1}2^{m-1}(m-1)!b_j \tau^{-m+1}\phi_{m-j,\epsilon}*N_\mu(\tau) \\
        &\quad 
		+ \sum_{\substack{k=1\\k \text{ odd}}}^{m-1}\sum_{\substack{j=0\\j \text{ even}}}^{k}(-1)^{k} b_{j} \tau^{-k}G_m^{(k)}(0^+)\phi_{k+1-j,\epsilon}*N_\mu(0) \,.
\end{align*}
Here, we have simplified the boundary terms using the properties of $G_m^{(k)}(1^-)$ and $G_m^{(k)}(0)$ that we listed before.

The identity that we have shown is essentially the claimed identity in the proposition, except for the double sum. We claim that each term in this double sum vanishes. This is seen as follows: As $N_\mu$ is an odd function, the fact that $\phi_{l, \epsilon}$ is an even function for $l \in 2\N$ yields 
\begin{equation*}
	\phi_{l,\epsilon}*N_\mu(0) = \int_\R \phi_{l,\epsilon}(-\sigma)N_\mu(\sigma)\,d\sigma = 0
\end{equation*}
whenever $l \in 2\N$. In the double sum we have $k$ odd and $j$ even, so $l=k+1-j \in 2\N$. This completes the proof of the proposition.
\end{proof}

To prove Theorem~\ref{thm: Fourier Tauberian} we shall apply the above proposition for the two measures $\mu$ and $\nu$ and use the assumptions to control the differences of the terms in the corresponding representations. To prove something more concrete about the convolutions that appear in this approach requires bounds for the functions $|\phi_{k, \epsilon}(\sigma)|$ that are uniform in $\epsilon$.  

\begin{lemma}\label{lem: psik bound}
Define $\psi_{-1}:=\psi_0:=\phi$ and for $k \geq 1$
\begin{equation*}
	\psi_k(\tau) := 
		\int_\tau^\infty \sigma \psi_{k-2}(\sigma)\,d\sigma \,.
\end{equation*}
Then, for each $k\geq -1$, $\psi_k$ belongs to $\mathcal{S}(\R)$ and is even and nonnegative. Moreover, for each $k \geq 0$ there is a constant $c_k$ so that for all $\tau \in \R$ and $\epsilon\in (0, 1]$ we have
$$|\phi_{k, \epsilon}(\tau)|\leq c_k\psi_k(\tau) \,.$$ 
\end{lemma}

\begin{proof}
We begin by proving the claimed properties of $\psi_k$. We argue by induction. Since $\psi_0=\psi_{-1}=\varphi$ these properties hold for $k\in\{0,-1\}$ by the assumptions on $\varphi$.

For $k\geq 1$ assume that we have proven the claimed properties for $\psi_{k-2}$. The fact that $\psi_{k-2}\in \mathcal{S}(\R)$ ensures on the one hand that $\psi_k$ is well defined. The fact that $\psi_{k-2}$ is even leads to
\begin{equation*}
    \psi_k(-\tau) = \int_{-\infty}^\infty \sigma\psi_{k-2}(\sigma)\,d\sigma -\psi_k(\tau) = \psi_k(\tau) \,,
\end{equation*}
since $\sigma\mapsto \sigma \psi_{k-2}(\sigma)$ is odd, that is, $\psi_k$ is also even. Since $\psi_{k-2}$ is nonnegative, it follows directly from the definition that $\psi_k$ is positive for $\tau \geq 0$. That $\psi_k$ is positive for $\tau<0$ follows as we have already proved that $\psi_k$ is even. That $\psi_k \in \mathcal{S}(\R)$ can be readily verified by using the entailed decay estimates ensured for $\psi_{k-2}$ and its derivatives by the fact that $\psi_{k-2}\in \mathcal{S}(\R)$.

Since $\phi_{k, \epsilon}$ is either odd or even depending on the parity of $k$, it suffices to prove bounds for $\tau \geq 0$.

We begin by proving that $|\phi_{k,\epsilon}(\tau)|\leq \psi_k(\tau)$ for $\tau \geq 1$. We argue by induction.

Clearly $|\phi_{0,\epsilon}(\tau)|\leq \psi_0(\tau)$ for each $\tau\in \R$ by definition. Since $\supp \chi_\epsilon \subseteq [-\epsilon, \epsilon]\subseteq [-1, 1]$ we have for all $\tau \geq 1$ that
\begin{align*}
	|\phi_{1, \epsilon}(\tau)| &= \biggl|\int_{\tau}^\infty \phi_{0, \epsilon}(\sigma)\,d\sigma \biggr|\\
	&\leq \frac{1}{\tau}\biggl|\int_{\tau}^\infty \sigma \phi_{0, \epsilon}(\sigma)\,d\sigma\biggr|\\
	&= \frac{\psi_1(\tau)}{\tau} \leq \psi_1(\tau)\,.
\end{align*}

Since $\supp \chi_\epsilon \subseteq [-\epsilon, \epsilon]\subseteq [-1, 1]$ we have for any $k \geq 2$ and all $\tau \geq 1$ by the induction hypothesis that
\begin{align*}
	|\phi_{k, \epsilon}(\tau)| &= \biggl|\int_{\tau}^\infty \int_s^\infty \phi_{k-2, \epsilon}(t)\,dtds\biggr|\\
	& = \biggl|\int_\tau^\infty (t-\tau)\phi_{k-2, \epsilon}(t)\,dt\biggr| \leq \int_\tau^\infty t|\phi_{k-2, \epsilon}(t)|\,dt \\
	&\leq \int_\tau^\infty t\psi_{k-2}(t)\,dt \\
	&= \psi_{k}(\tau)\,.
\end{align*}

Thus we have proved that for all $k\geq 0, \epsilon \in (0, 1], |\tau|\geq 1$ it holds that
\begin{equation*}
	|\phi_{k, \epsilon}(\tau)|\leq \psi_k(\tau)\,.
\end{equation*}

For $\tau\in (-1, 1)$ we again argue by induction. By definition the bound $|\phi_{0, \epsilon}(\tau)|\leq \psi_0(\tau)$ holds for all $\tau \in \R$ independently of $\epsilon$. Assuming that $|\phi_{k-1, \epsilon}(\tau)|\leq c_{k-1}\psi_{k-1}(\tau)$ for all $\tau \in \R$, then by the bound $|\phi_{k,\epsilon}(\tau)|\leq \psi_k(\tau)$ for $\tau \geq 1$ we also have that, for any $\tau \in (-1, 1)$,
\begin{align*}
	|\phi_{k, \epsilon}(\tau)|
	&=
	\biggl|\phi_{k,\epsilon}(1)- \int_\tau^1 \phi_{k, \epsilon}'(\sigma)\,d\sigma\biggr|\\
	&=
	\biggl|\phi_{k,\epsilon}(1)- \int_\tau^1 \Bigl(\phi_{k-1, \epsilon}(\sigma)-\Bigl(\int_\R \phi_{k-1,\epsilon}(\sigma')\,d\sigma'\Bigr)\chi_\epsilon(\sigma)\Bigr)\,d\sigma\biggr|\\
	&\leq
	|\phi_{k,\epsilon}(1)| + \int_\tau^1 |\phi_{k-1,\epsilon}(\sigma)|\,d\sigma + \biggl(\int_\R |\phi_{k-1,\epsilon}(\sigma')|\,d\sigma'\biggr)\int_\tau^1 \chi_\epsilon(\sigma)\,d\sigma\\
	&\leq
	\psi_k(1) +2c_{k-1}\int_\R \psi_{k-1}(\sigma)\,d\sigma \,,
\end{align*}
where we used $\chi_\epsilon\geq 0$ and $\int_\R \chi_\epsilon(\sigma)\,d\sigma =1$. We argued above that $\psi_k(\tau)>0$ for all $\tau \in \R$ and $k\geq 1$. Therefore, if we define
\begin{equation*}
	c_k := \frac{\psi_{k}(1)+2c_{k-1}\int_\R \psi_{k-1}(\sigma)\,d\sigma}{\inf_{\sigma\in [-1, 1]}\psi_k(\sigma)} >1 \,,
\end{equation*}
then, for all $\tau \in \R$,
\begin{equation*}
	|\phi_{k,\epsilon}(\tau)| \leq c_k\psi_k(\tau)\,.
\end{equation*}
This completes the proof of the lemma.
\end{proof}

We are now ready to prove Theorem~\ref{thm: Fourier Tauberian}. 
\begin{proof}[Proof of Theorem~\ref{thm: Fourier Tauberian}]
    We begin by reducing to the case $a=1$. Assume that $\mu, \nu$ satisfy the assumptions in the theorem with the parameters $a=a', a_0=a_0', a_1=a_1', M_0=M_0', M_1=M_2', \alpha', \beta'$. Define measures $\tilde\mu, \tilde\nu$ by
    \begin{equation*}
        \tilde\mu(\omega) := \frac{1}{a'}\mu(a'\omega) \quad \mbox{and} \quad \tilde\nu(\omega) := \frac{1}{a'}\nu(a'\omega)\,.
    \end{equation*}
    Then $\tilde \mu, \tilde \nu$ satisfy the assumptions of the theorem with parameters $a=1, a_0=a_0'/a', a_1=a_1'/a', M_0=(a')^{\alpha}M_0', M_1=(a')^\beta M_1', \alpha = \alpha', \beta= \beta'$. Indeed, one can check that
    \begin{equation*}
       \phi*(T_{\tilde\mu}-T_{\tilde\nu})(\tau) =\phi_{a'}*(T_\mu-T_\nu)(a'\tau) \quad \mbox{and} \quad  T_{\tilde \nu}(\psi)= \frac{1}{a'}T_{\nu}(\psi(\cdot/a')) \quad \forall \psi \in C_0^\infty(\R) \,,
    \end{equation*}
    and so the claimed bounds for $\tilde \mu, \tilde \nu$ follow from those assumed for $\mu, \nu$. Furthermore, the definition of $\tilde \mu, \tilde \nu$ together with a change of variables shows that
    \begin{equation*}
        R_{\tilde \mu}^\gamma(\tau) = \frac{1}{a'}R_\mu^\gamma(a'\tau) \quad \mbox{and}\quad R_{\tilde \nu}^\gamma(\tau) = \frac{1}{a'}R_\nu^\gamma(a'\tau)\,.
    \end{equation*}
    Consequently, the claimed bound for $\mu, \nu$ follows from the theorem when applied to $\tilde \mu, \tilde \nu$. Therefore, we may without loss of generality assume that $a=1$.

    The case $\gamma =0$ is the content of~\cite[Lemma~17.5.6]{Hormander_bookIII}, which states that there exists a constant $C$, depending only on $\alpha, \beta$, so that
    \begin{equation}\label{eq: Hormanders bound}
        |N_\mu(\tau)-N_\nu(\tau)|\leq C(M_0(a_0+|\tau|)^\alpha + M_1(1+|\tau|)(|\tau|+a_1)^\beta)\,.
    \end{equation}

	By applying Proposition~\ref{prop: itterated Safarov identity} with both measures $\mu, \nu$, we find
	\begin{align*}
	\bigl|R_{\mu, \epsilon}^m(\tau) -R_{\nu, \epsilon}^m(\tau)\bigr|
		&\leq 
		2m\sum_{\substack{j=0\\j \text{ even}}}^m |b_j| \biggl|\int_0^\tau \tau^{-j-1}G_m^{(j)}(\sigma/\tau)\phi_{0,\epsilon}*(N_\mu-N_\nu)(\sigma)\,d\sigma\biggr|\\
	&\quad +2m\sum_{\substack{j=0\\j \text{ even}}}^{m}|b_{j}| \biggl|\int_0^\tau \tau^{-m-1}G_m^{(m)}(\sigma/\tau)\phi_{m+1-j,\epsilon}*(T_\mu-T_\nu)(\sigma)\,d\sigma\biggr|\\
		&\quad + \sum_{\substack{j=0\\ j\text{ even}}}^{m-1}2^{m}m!|b_j| \tau^{-m}|\phi_{m-j,\epsilon}*(N_\mu-N_\nu)(\tau)|\\
		&\leq
		2m\sum_{\substack{j=0\\j \text{ even}}}^m |b_j| \tau^{-j}\|G_m^{(j)}\|_{\infty} \sup_{\sigma \in [0, \tau]}|\phi_{0,\epsilon}*(N_\mu-N_\nu)(\sigma)|\\
	&\quad +2m\sum_{\substack{j=0\\j \text{ even}}}^{m}|b_{j}| \tau^{-m}\|G_m^{(m)}\|_\infty\sup_{\sigma \in [0, \tau]}|\phi_{m+1-j,\epsilon}*(T_\mu-T_\nu)(\sigma)|\\
		&\quad + \sum_{\substack{j=0\\ j\text{ even}}}^{m-1}2^{m}m!|b_j| \tau^{-m}|\phi_{m-j,\epsilon}*(N_\mu-N_\nu)(\tau)|\,.
\end{align*}

Since $\phi_{0, \epsilon}=\phi$ and $\phi*(N_\mu-N_\nu)(0)=0$ by symmetry, the assumptions yield that
\begin{equation*}
	|\phi_{0, \epsilon}*(N_\mu-N_\nu)(\sigma)| = \biggl|\int_0^\sigma \phi_{0, \epsilon}*(T_\mu-T_\nu)(s)\,ds\biggr| \leq M_1 (a_1+|\sigma|)^{\beta+1} \,.
\end{equation*}
This allows us to bound the terms in the first sum.

By an integration by parts and the estimate $|\phi_{k, \epsilon}(\sigma)|\leq c_k \psi_k(\sigma)$ from Lemma \ref{lem: psik bound}, we obtain
\begin{align*}
	|\phi_{m+1-j, \epsilon}&*(T_\mu-T_\nu)(\sigma)|\\ &= \biggl|\int_\R\phi'_{m+1-j,\epsilon}(\sigma')(N_\mu(\sigma-\sigma')-N_\nu(\sigma-\sigma'))\,d\sigma'\biggr|\\
	&= \biggl|\int_\R\phi_{m-j,\epsilon}(\sigma')(N_\mu(\sigma-\sigma')-N_\nu(\sigma-\sigma'))\,d\sigma'\\
	&\quad 
	-\Bigl(\int_\R \phi_{m-j, \epsilon}(\sigma'')\,d\sigma'' \Bigr) \int_\R\chi_\epsilon(\sigma')(N_\mu(\sigma-\sigma')-N_\nu(\sigma-\sigma'))\,d\sigma'\biggr|\\
	& 
	\leq 
	c_{m-j}\psi_{m-j}*|N_\mu-N_\nu|(\sigma)\\
	&\quad 
	+\Bigl|\int_\R \phi_{m-j, \epsilon}(\sigma'')\,d\sigma''\Bigr| \chi_\epsilon*|N_\mu-N_\nu|(\sigma)\,,
\end{align*}
and
\begin{equation*}
	|\phi_{m-j,\epsilon}*(N_\mu-N_\nu)(\tau)| \leq c_{m-j}\psi_{m-j}*|N_\mu-N_\nu|(\tau)\,.
\end{equation*}

To control the remaining convolutions we utilize~\eqref{eq: Hormanders bound} together with the elementary inequality $(a_0+|\sigma-\sigma'|) \leq (a_0+|\sigma|+|\sigma'|)\leq (a_0+|\sigma|)(1+|\sigma'|)$. This yields the bound
\begin{align*}
	\psi_k\!*\!|N_\mu-N_\nu|(\sigma) 
	&= \int_\R \psi_k(\sigma')|N_\mu(\sigma-\sigma')-N_\nu(\sigma-\sigma')|\,d\sigma'\\
	&\leq C\!\!\int_\R \!\psi_k(\sigma')\bigl(M_0 (a_0\!+\!|\sigma\!-\!\sigma'|)^{\alpha}\!+ \!M_1(1\!+\!|\sigma\!-\!\sigma'|)(a_1\!+\!|\sigma\!-\!\sigma'|)^{\beta}\bigr)\,d\sigma' \\
	&\leq C M_0 (a_0+|\sigma|)^{\alpha}\int_\R \psi_k(\sigma')(1+|\sigma'|)^{\alpha}\,d\sigma'\\
    &\quad +CM_1(1+|\sigma|)(a_1+|\sigma|)^{\beta}\int_\R \psi_k(\sigma')(1+|\sigma'|)^{\beta+1}\,d\sigma'\\
    &\leq C_{k, \varphi, \alpha, \beta} \bigl(M_0(a_0+|\sigma|)^\alpha+ M_1(1+|\sigma|)(a_1+|\sigma|)^\beta\bigr)\,.
\end{align*}
Similarly, for any $\epsilon\in (0, 1]$
\begin{align*}
	\chi_\epsilon*|N_\mu&-N_\nu|(\sigma) \\
	&\leq \|\chi\|_\infty\int_{-1}^1 C(M_0(a_0+|\sigma-\epsilon \sigma'|)^{\alpha}+M_1(1+|\sigma-\epsilon \sigma'|)(a_1+|\sigma- \epsilon \sigma'|)^{\beta})\,d\sigma'\\
	&\leq\|\chi\|_\infty C_{\alpha, \beta}\bigl(M_0(a_0+|\sigma|)^{\alpha}+M_1(1+|\sigma|)(a_1+|\sigma|)^{\beta}\bigr)\,.
\end{align*}

Putting all this back into the bound for the difference of the Riesz means yields
\begin{align*}
	\bigl|R_{\mu, \epsilon}^m(\tau) &-R_{\nu, \epsilon}^m(\tau)\bigr|\\
		&\leq
		2m\sum_{\substack{j=0\\j \text{ even}}}^m |b_j| \tau^{-j}\|G_m^{(j)}\|_{\infty} M_1 (a_1+|\sigma|)^{\beta+1}\\
	&\quad +2m\sum_{\substack{j=0\\j \text{ even}}}^{m} \tau^{-m}\|G_m^{(m)}\|_\infty C'_{m-j, \phi, \alpha, \beta}\bigl(M_0 (a_0+|\tau|)^{\alpha}+ M_1(1+|\tau|)(a_1+|\tau|)^{\beta}\bigr)\\
		&\quad + \sum_{\substack{j=0\\ j\text{ even}}}^{m-1}2^{m}m!|b_j| \tau^{-m}C''_{m-j,\phi,\alpha,\beta}\bigl(M_0 (a_0+|\tau|)^{\alpha}+ M_1(1+|\tau|)(a_1+|\tau|)^{\beta}\bigr)\,.
\end{align*}
Since $G_m$ is a polynomial whose coefficients depend only on $m$, we conclude that there exists a constant $C_*$ depending only on $\phi, \alpha, \beta, m$ such that
\begin{align*}
	\bigl|R_{\mu, \epsilon}^m(\tau) -R_{\nu, \epsilon}^m(\tau)\bigr|
		&\leq
		C_* (M_0(a_0+|\tau|)^{\alpha-m}+M_1 (1+|\tau|)(a_1+|\tau|)^{\beta})\,.
\end{align*}
Taking the limit $\epsilon \to 0$ completes the proof.
\end{proof}


\subsection{Proof of Lemma \ref{lem: Integration by parts identity}}

We finally derive the identity whose proof we deferred.

\begin{proof}[Proof of Lemma~\ref{lem: Integration by parts identity}]
{\allowdisplaybreaks
    \emph{Step 1.} In this step we will prove two key identities on which the proof hinges. The first one is
    \begin{equation}\label{eq: identity 1}
	   \chi_\epsilon* u(\sigma) = \phi*u(\sigma) - \phi_{1,\epsilon}*u'(\sigma)\,,
    \end{equation}
    and the second one is, for any $k\in\N_0$,
    \begin{equation}\label{eq: identity 3}
    \begin{aligned}
    	\int_0^\tau f(\sigma)\phi_{k, \epsilon}*u'(\sigma)\,d\sigma 
	   &=
	   -\int_0^\tau f'(\sigma)\phi_{k+1,\epsilon}*u'(\sigma) d\sigma\\
    	& \quad +f(\tau^\limminus)\phi_{k,\epsilon}*u(\tau)-f(0^\limplus)\phi_{k,\epsilon}*u(0)\\
    	&\quad -\Bigl(\int_\R \phi_{k,\epsilon}(\sigma')\,d\sigma' \Bigr)\int_0^\tau f'(\sigma) \phi_{0,\epsilon}*u(\sigma) \,d\sigma\\
    	&\quad +\Bigl(\int_\R \phi_{k,\epsilon}(\sigma')\,d\sigma' \Bigr)\int_0^\tau f'(\sigma) \phi_{1,\epsilon}*u'(\sigma)\,d\sigma \,.
    \end{aligned}
    \end{equation}

    To prove these identities, we recall the definition of $\phi_{k,\epsilon}$ and integrate by parts to obtain
    \begin{align*}
	   \phi_{k,\epsilon} * u(\sigma) 
	   &= \int_\R \phi_{k,\epsilon}(\sigma-\sigma')u(\sigma')\,d\sigma'\\
	   &= \int_\R \Bigl(\phi_{k,\epsilon}(\sigma-\sigma')-\Bigl(\int_\R\phi_{k,\epsilon}(\sigma'')\,d\sigma''\Bigr)\chi_\epsilon(\sigma-\sigma')\Bigr)u(\sigma')\,d\sigma' \\
	   &\quad + \Bigl(\int_\R \phi_{k,\epsilon}(\sigma'')\,d\sigma'' \Bigr)\Bigl(\int_\R \chi_\epsilon(\sigma-\sigma')u(\sigma')\,d\sigma'\Bigr)\\
	   &= \phi_{k+1, \epsilon}'*u(\sigma) + \Bigl(\int_\R \phi_{k,\epsilon}(\sigma'')\,d\sigma'' \Bigr)\chi_\epsilon*u(\sigma)\\
	   &= \phi_{k+1, \epsilon}*u'(\sigma)+ \Bigl(\int_\R \phi_{k,\epsilon}(\sigma'')\,d\sigma'' \Bigr)\chi_\epsilon*u(\sigma)\,.
    \end{align*}
    That is, for any $k\geq 0$,
    \begin{equation}\label{eq: identity 0}
	   \Bigl(\int_\R \phi_{k,\epsilon}(\sigma'')\,d\sigma'' \Bigr)\chi_\epsilon* u(\sigma) = \phi_{k,\epsilon}*u(\sigma) - \phi_{k+1,\epsilon}*u'(\sigma)\,.
    \end{equation}
    In particular, for $k=0$ since $\phi_{0,\epsilon}=\phi$ and $\int_\R \phi(\tau)\,d\tau=1$, we obtain the claimed identity \eqref{eq: identity 1}.
    
    By using~\eqref{eq: identity 0} and an integration by parts
    \begin{equation}\label{eq: midstep}
    \begin{aligned}
    	\int_0^\tau f(\sigma)&\phi_{k, \epsilon}*u'(\sigma)\,d\sigma\\ 
    	&= \int_0^\tau f(\sigma)\Bigl(\phi_{k+1,\epsilon}*u'(\sigma)+\Bigl(\int_\R \phi_{k,\epsilon}(\sigma'')\,d\sigma'' \Bigr)\chi_\epsilon* u(\sigma)\Bigr)'d\sigma\\
    	&=
    	\int_0^\tau f(\sigma)\Bigl(\phi_{k+1,\epsilon}*u'(\sigma)\Bigr)'d\sigma\\
    	&\quad+\Bigl(\int_\R \phi_{k,\epsilon}(\sigma'')\,d\sigma'' \Bigr)\int_0^\tau f(\sigma)\chi_\epsilon* u'(\sigma)d\sigma\\
    	&=
    	-\int_0^\tau f'(\sigma)\phi_{k+1,\epsilon}*u'(\sigma) d\sigma\\
    	&\quad +f(\tau^\limminus)\phi_{k+1,\epsilon}*u'(\tau)-f(0^\limplus)\phi_{k+1,\epsilon}*u'(0)\\
    	&\quad+\Bigl(\int_\R \phi_{k,\epsilon}(\sigma'')\,d\sigma'' \Bigr)\int_0^\tau f(\sigma)\chi_\epsilon* u'(\sigma)d\sigma\\
    \end{aligned}
    \end{equation}

    Again by integration by parts it also holds that
    \begin{align*}
	   \int_0^\tau f(\sigma)\chi_\epsilon* u'(\sigma)\,d\sigma = - \int_0^\tau f'(\sigma) \chi_\epsilon*u(\sigma)\,d\sigma + f(\tau^-)\chi_\epsilon*u(\tau)-f(0^+)\chi_\epsilon*u(0)\,.
    \end{align*}
    Plugging this into the last integral in~\eqref{eq: midstep} and using \eqref{eq: identity 0} to rewrite the boundary terms we have proved
    \begin{align*}
    	\int_0^\tau f(\sigma)\phi_{k, \epsilon}*u'(\sigma)\,d\sigma 
    	&=
    	-\int_0^\tau f'(\sigma)\phi_{k+1,\epsilon}*u'(\sigma) d\sigma\\
    	&\quad +f(\tau^\limminus)\Bigl(\phi_{k,\epsilon}*u(\tau)-\Bigl(\int_\R \phi_{k,\epsilon}(\sigma')\,d\sigma'\Bigr)\chi_\epsilon*u(\tau)\Bigr)\\
    	&\quad-f(0^\limplus)\Bigl(\phi_{k,\epsilon}*u(0)-\Bigl(\int_\R \phi_{k,\epsilon}(\sigma')\,d\sigma'\Bigr)\chi_\epsilon*u(0)\Bigr)\\
    	&\quad+\Bigl(\int_\R \phi_{k,\epsilon}(\sigma')\,d\sigma' \Bigr)\Bigl(- \int_0^\tau f'(\sigma) \chi_\epsilon*u(\sigma)\,d\sigma\\
	   &\quad + f(\tau^-)\chi_\epsilon*u(\tau)-f(0^+)\chi_\epsilon*u(0)\Bigr)\\
	   &=
	   -\int_0^\tau f'(\sigma)\phi_{k+1,\epsilon}*u'(\sigma) d\sigma\\
    	&\quad +f(\tau^\limminus)\phi_{k,\epsilon}*u(\tau)-f(0^\limplus)\phi_{k,\epsilon}*u(0)\\
    	&\quad-\Bigl(\int_\R \phi_{k,\epsilon}(\sigma')\,d\sigma' \Bigr)\int_0^\tau f'(\sigma) \chi_\epsilon*u(\sigma)\,d\sigma
    \end{align*}
    Finally, we use again~\eqref{eq: identity 1} to rewrite the remaining integral containing $\chi_\epsilon *u$ and obtain the claimed identity \eqref{eq: identity 3}. This concludes the first step of the proof.
    }
    
    \medskip

    \emph{Step 2.} With the identities from Step 1 at hand, we now turn to the main part of the proof of Lemma~\ref{lem: Integration by parts identity}. The idea is to apply identity \eqref{eq: identity 3} iteratively.

    Before carrying this out, however, it is imporant to note that the numbers $b_m$ vanish when $m$ is odd. This follows easily by induction, observing that $\phi_{j,\epsilon}$ is odd when $j$ is odd, so $\int_\R \phi_{j,\epsilon}(\tau)\,d\tau=0$ in this case.

    We will prove the identity in the lemma by induction on $m\in\N_0$. The base case $m=0$ with $b_0=1$ follows immediately by integrating identity \eqref{eq: identity 1} against $f$ over the interval $[0,\tau]$.

    Now let $m\geq 1$ and assume that the identity claimed in the lemma holds when $m=M-1$, that is,
    \begin{align*}
		\int_0^\tau &f(\sigma)\chi_\epsilon*u(\sigma)\,d\sigma \\
		&=
		\sum_{\substack{j=0\\j \text{ even}}}^{M-1} b_{j} \int_0^\tau f^{(j)}(\sigma)\phi_{0,\epsilon}*u(\sigma)\,d\sigma\\
		&\quad 
		-(-1)^{M-1} \sum_{\substack{j=0\\j \text{ even}}}^{M-1} b_{j} \int_0^\tau f^{(M-1)}(\sigma)\phi_{M-j,\epsilon}*u'(\sigma)\,d\sigma\\
		&\quad 
		- \sum_{k=0}^{M-2}\sum_{\substack{j=0\\j \text{ even}}}^{k} (-1)^k b_j \Bigl[f^{(k)}(\tau^-)\phi_{k+1-j,\epsilon}*u(\tau)-f^{(k)}(0^+)\phi_{k+1-j,\epsilon}*u(0)\Bigr] \,.
    \end{align*}
    We rewrite the terms in the second sum using identity~\eqref{eq: identity 3} with $k=M-j$ and with $f$ replaced by $f^{(M-1)}$ and find
    \begin{align*}
        \int_0^\tau \!\!f^{(M-1)}(\sigma)\phi_{M-j,\epsilon}*u'(\sigma)\,d\sigma\!
	    &=
        -\Bigl(\int_\R \phi_{M-j,\epsilon}(\sigma')\,d\sigma' \Bigr)\int_0^\tau f^{(M)}(\sigma) \phi_{0,\epsilon}*u(\sigma) \,d\sigma\\    
        &\quad 
	   -\int_0^\tau f^{(M)}(\sigma)\phi_{M+1-j,\epsilon}*u'(\sigma) d\sigma\\
       & \quad +\Bigl(\int_\R \phi_{M-j,\epsilon}(\sigma')\,d\sigma' \Bigr)\int_0^\tau f^{(M)}(\sigma) \phi_{1,\epsilon}*u'(\sigma)\,d\sigma \\
       & \quad +f^{(M-1)}(\tau^\limminus)\phi_{M-j,\epsilon}*u(\tau)\!-\!f^{(M-1)}(0^\limplus)\phi_{M-j,\epsilon}*u(0)\,.
    \end{align*}
    Inserting this formula into the previous one and reallocating the terms in the sum yields
	\begin{align*}
		\int_0^\tau &f(\sigma)\chi_\epsilon*u(\sigma)\,d\sigma \\
		&=
		\sum_{\substack{j=0\\j\text{ even}}}^{M-1} b_{j}\int_0^\tau f^{(j)}(\sigma)\phi_{0,\epsilon}*u(\sigma)\,d\sigma \\
        &\quad
        + (-1)^{M-1}\sum_{\substack{j=0\\j\text{ even}}}^{M-1} b_{j} \Bigl(\int_\R \phi_{M-j,\epsilon}(\sigma')\,d\sigma' \Bigr)\int_0^\tau f^{(M)}(\sigma) \phi_{0,\epsilon}*u(\sigma) \,d\sigma\\
		&\quad 
		- (-1)^{M} \sum_{\substack{j=0\\j\text{ even}}}^{M-1} b_{j}
        \int_0^\tau f^{(M)}(\sigma)\phi_{M+1-j,\epsilon}*u'(\sigma) d\sigma
        \\
        &\quad -(-1)^{M-1} \sum_{\substack{j=0\\j\text{ even}}}^{M-1} b_{j} \Bigl(\int_\R \phi_{M-j,\epsilon}(\sigma')\,d\sigma' \Bigr)\int_0^\tau f^{(M)}(\sigma) \phi_{1,\epsilon}*u'(\sigma)\,d\sigma \\
		&\quad 
		- \sum_{k=0}^{M-2}\sum_{\substack{j=0\\j\text{ even}}}^{k} (-1)^k b_j \Bigl[f^{(k)}(\tau^-)\phi_{k+1-j,\epsilon}*u(\tau)-f^{(k)}(0^+)\phi_{k+1-j,\epsilon}*u(0)\Bigr]\\
        & \quad - (-1)^{M-1} \sum_{\substack{j=0\\j\text{ even}}}^{M-1} b_j \Bigl[ f^{(M-1)}(\tau^\limminus)\phi_{M-j,\epsilon}*u(\tau)-f^{(M-1)}(0^\limplus)\phi_{M-j,\epsilon}*u(0) \Bigr].
    \end{align*}
    Recalling the recursive definition of $b_M$, we can write this as
    \begin{align*}
		\int_0^\tau &f(\sigma)\chi_\epsilon*u(\sigma)\,d\sigma \\
		&=
		\sum_{\substack{j=0\\ j \text{ even}}}^{M} b_{j} \int_0^\tau f^{(j)}(\sigma)\phi_{0,\epsilon}*u(\sigma)\,d\sigma \\
		&\quad 
		-(-1)^{M} \sum_{\substack{j=0\\ j \text{ even}}}^{M-1} b_{j}
        \int_0^\tau f^{(M)}(\sigma)\phi_{M+1-j,\epsilon}*u'(\sigma) d\sigma
        - b_M \int_0^\tau f^{(M)}(\sigma) \phi_{1,\epsilon}*u'(\sigma)\,d\sigma
        \\
		&\quad 
		- \sum_{k=0}^{M-1}\sum_{\substack{j=0\\ j \text{ even}}}^{k} (-1)^k b_j \Bigl[f^{(k)}(\tau^-)\phi_{k+1-j,\epsilon}*u(\tau)-f^{(k)}(0^+)\phi_{k+1-j,\epsilon}*u(0)\Bigr] \,.
    \end{align*}
    This is almost the claimed identity for $m=M$, except that the integral involving $\phi_{1,\epsilon}*u'$ has the prefactor $-b_M$ instead of $-(-1)^M b_M$. Therefore the proof is completed by observing that $b_M = (-1)^M b_M$, which follows from the fact that $b_M=0$ if $M$ is odd.
\end{proof}

\begin{remark}\label{safarovrem}
    We close this subsection by pointing out what we believe to be an inaccuracy in~\cite{Safarov01}. The issue arises in Lemma~2.7 of that paper, in which the author claims certain bounds for three convolutions involving a function $F$ with $\supp F\subset (0, \infty)$ or its derivative. However, to our understanding the proof as written only yields these bounds for the corresponding convolutions involving the odd extension of $F$ rather than with $F$ itself. Carrying out the subsequent applications of Lemma~2.7 with $F$ replaced by its odd extension and restricting at the very end of the argument to positive values of the argument one recovers the claimed results.
\end{remark}

\appendix

\phantomsection
\addcontentsline{toc}{part}{Appendix}


\section{Convex geometry toolbox}
\label{app: Convex geometry}

In this appendix we record some elements of convex geometry. Most of the results that we shall make use of are either well known to experts or follow from well-known results.

We begin by recalling the fact that perimeter is monotonically increasing under inclusion of convex sets.
\begin{lemma}\label{lem: local monotonicity of perimieter}
	If $\Omega' \subset \Omega \subset \R^d$ are convex sets, then 
	$$
		\Haus^{d-1}(\partial\Omega') \leq \Haus^{d-1}(\partial\Omega)\,.
	$$ 
	In particular, if $\Omega \subset \R^d$ is convex and $x\in \R^d, r>0$ then
	\begin{equation*}
		\Haus^{d-1}(\partial\Omega \cap B_r(x)) \leq \Haus^{d-1}(\partial B_r(x)) \lesssim_d r^{d-1}\,.
	\end{equation*}
\end{lemma}

The second lemma provides an upper and a lower bound for the size of the level set of the distance function $d_\Omega$, defined in \eqref{eq:distance}. It is stated in terms of the \emph{inradius}
\begin{equation}
    \label{eq:inrad}
    r_{\rm in}(\Omega) := \sup_{x\in\Omega} d_\Omega(x) \,.
\end{equation}
The lower bound is proved in \cite{LarsonJFA} and the upper is a consequence of Lemma~\ref{lem: local monotonicity of perimieter}.

\begin{lemma}\label{lem: inner parallel perimeter bounds}
	If $\Omega \subset \R^d$ is a bounded convex set, then for any $s \in (0, r_{\rm in}(\Omega)]$
\begin{equation*}
	\Bigl(1-\frac{s}{r_{\rm in}(\Omega)}\Bigr)^{d-1}\Haus^{d-1}(\partial\Omega)\leq \Haus^{d-1}(\{x\in \Omega: d_\Omega(x)=s\}) \leq \Haus^{d-1}(\partial\Omega)\,.
\end{equation*}
\end{lemma}

By an application of the co-area formula and the bounds in Lemma~\ref{lem: inner parallel perimeter bounds} one can prove the next two lemmas (see for instance~\cite[Section 5]{FrankLarson_Crelle20} or \cite{LarsonJST}).
\begin{lemma}\label{lem: volume bdry neighbourhood bound}
	If $\Omega \subset \R^d$ is a bounded convex set, then for any $s \in [0, r_{\rm in}(\Omega)]$
\begin{equation*}
	|\{x\in \Omega: d_\Omega(x)<s\}| \leq s\Haus^{d-1}(\partial\Omega)\,.
\end{equation*}
\end{lemma}

\begin{lemma}
	\label{lem: inradius, volume, perimeter bound}
	If $\Omega \subset \R^d$ is a bounded convex set, then
	\begin{equation*}
		\frac{|\Omega|}{\Haus^{d-1}(\partial\Omega)} \leq r_{\rm in}(\Omega) \leq \frac{d|\Omega|}{\Haus^{d-1}(\partial\Omega)}\,.
	\end{equation*}
\end{lemma}

Throughout the paper we several times make use of the following consequence of the Bishop--Gromov comparison theorem, which in Euclidean space has an elementary proof.

\begin{lemma}\label{lem: Bishop-Gromov monotonicity}
	If $\Omega \subset \R^d$ is convex, then the function
\begin{equation*}
	(0, \infty) \ni r \mapsto \frac{|\Omega \cap B_r(a)|}{r^d}
\end{equation*}
is nonincreasing for any fixed $a\in \overline{\Omega}$.
\end{lemma} 

\begin{proof}
	For $a\in\overline\Omega$ and $R>r>0$ consider the set
	\begin{align*}
		E & := \left( 1 -\frac rR\right)\{a\} + \frac rR\, \left(\Omega\cap B_R(a)\right) \\
		& = \left\{ \left( 1 -\frac rR\right)a + \frac rR\,x :\ x \in \Omega\cap B_R(a) \right\}.
	\end{align*}
	Then,
	$$
	|E| = \left| \frac rR\, \left(\Omega\cap B_R(a)\right) \right| = \left( \frac{r}{R} \right)^d |\Omega\cap B_R(a)|\,.
	$$
	Meanwhile,
	$$
	E \subset \Omega\cap B_r(a) \,.
	$$
	Indeed, on the one hand, $E\subset \left( 1 -\frac rR\right)\{a\} + \frac rR\, B_R(a) = B_r(a)$ and, on the other hand $E\subset \left( 1 -\frac rR\right)\{a\} + \frac rR\, \Omega \subset\Omega$ by convexity of $\Omega$. Combining the bound $|E|\leq |\Omega\cap B_r(a)|$ with the explicit expression for $|E|$ we obtain the inequality
	\begin{equation*}
		\frac{|\Omega \cap B_r(a)|}{r^d} \geq \frac{|\Omega \cap B_R(a)|}{R^d}\,,
	\end{equation*}
	which is the desired monotonicity.
\end{proof}

\begin{proposition}\label{prop: Minkowski sum bounds}
	Fix $d \geq 1$. Given $c_1>0$ there exists $c_2>0$ with the following properties. 
	If $\Omega \subset \R^d$ is open, bounded, and convex, then
	\begin{align*}
		|\Omega + B_r|-|\Omega| & \geq r\Haus^{d-1}(\partial\Omega) \quad \mbox{for all }r>0\,,\\
		|\Omega + B_r|-|\Omega|&\lesssim_d \Haus^{d-1}(\partial\Omega)r\Bigl[1+\Bigl(\frac{r}{r_{\rm in}(\Omega)}\Bigr)^{d-1}\Bigr] \quad \mbox{for all }r>0\,,\\
		|\Omega+B_r| - |\Omega| &\leq \Haus^{d-1}(\partial\Omega)r\Bigl[1 + c_2 \frac{r}{r_{\rm in}(\Omega)}\Bigr]\quad \mbox{for all }r\in [0, c_1 r_{\rm in}(\Omega)]\,,\\
		|\Omega+B_r| &\leq c_2 \Haus^{d-1}(\partial\Omega) r \Bigl(\frac{r}{r_{\rm in}(\Omega)}\Bigr)^{d-1} \quad \mbox{for all } r\geq c_1 r_{\rm in}(\Omega)\,.
	\end{align*}
\end{proposition}

\begin{proof}[Proof of Proposition~\ref{prop: Minkowski sum bounds}]
	The proof of these inequalities are based on writing $|\Omega+ B_r|$ in terms of mixed volumes. (Technically mixed volumes are usually defined for compact convex sets, but as $\Omega$ is convex and may be assumed non-empty each of the geometric quantities appearing in the statement are the same if $\Omega, B_r$ are replaced by their respective closures.)

	With $W$ denoting the mixed volume,
	\begin{equation}\label{eq: Minkowski sum}
		|\Omega+B_r| = \sum_{j=0}^d \binom{d}{j}W(\underbrace{\overline{\Omega}, \ldots, \overline{\Omega}}_{d-j \textrm{ copies}}, \underbrace{\overline{B_{r}}, \ldots, \overline{B_{r}}}_{j \textrm{ copies}})\,.
	\end{equation}
	To prove the estimates we recall the following basic properties of $W$ (see, for instance,~\cite{SchneiderBook}); if $K_1, \ldots, K_d$ are convex bodies, then
	\begin{enumerate}[label=(\roman*)]
		\item\label{itm: MV nonneg} $W(K_1, \ldots, K_d)\geq 0$,
		\item\label{itm: MV monotone} $W(K_1, \ldots, K_d) \leq W(K_1', K_2, \ldots, K_d)$ for any $K_1' \supset K_1$,
		\item\label{itm: MV transl. inv.} $W(K_1 + \{x\}, K_2, \ldots, K_d)=W(K_1, K_2, \ldots, K_d)$ for any $x \in \R^d$,
		\item\label{itm: MV perm. inv.} $W(K_1, \ldots, K_d) = W(K_{\sigma(1)}, \ldots, K_{\sigma(d)})$ for any permutation $\sigma$,
		\item\label{itm: MV scaling} $W(rK_1, K_2, \ldots, K_d) = r W(K_1, \ldots, K_d)$ for any $r>0$,
		\item\label{itm: MV vol.} $W(K_1, \ldots, K_1) = |K_1|$,
		\item\label{itm: MV per.} $W(K_1, \ldots, K_1, B_1) = d^{-1}\Haus^{d-1}(\partial K_1)$.
	\end{enumerate}
	Note that by \ref{itm: MV perm. inv.}-\ref{itm: MV per.},
	\begin{equation}\label{eq: Minkowski sum first two terms}
		\sum_{j=0}^1 \binom{d}{j}W(\underbrace{\overline{\Omega}, \ldots, \overline{\Omega}}_{d-j \textrm{ copies}}, \underbrace{\overline{B_{r}}, \ldots, \overline{B_{r}}}_{j \textrm{ copies}}) = |\Omega| + r \Haus^{d-1}(\partial\Omega)\,.
	\end{equation}

	The first bound in the proposition follows by \eqref{eq: Minkowski sum first two terms} and dropping all the terms in the right-hand side of~\eqref{eq: Minkowski sum} with $j \geq 2$ (which are nonnegative by~\ref{itm: MV nonneg}).

	For the second and third bounds in the proposition note that by \ref{itm: MV perm. inv.} and \ref{itm: MV scaling},
	\begin{equation}\label{eq: Minkowski sum 2}
	\begin{aligned}
	 	|\Omega+B_{r}| 
	 	-|\Omega|&=\sum_{j=1}^d \binom{d}{j} W(\underbrace{\overline{\Omega}, \ldots, \overline{\Omega}}_{d-j \textrm{ copies}}, \underbrace{\overline{B_{r}}, \ldots, \overline{B_{r}}}_{j \textrm{ copies}})\\
	 	&= r_{\rm in}(\Omega)\sum_{j=1}^d \binom{d}{j} \Bigl(\frac{r}{r_{\rm in}(\Omega)}\Bigr)^{j} W(\underbrace{\overline{\Omega}, \ldots, \overline{\Omega}}_{d-j \textrm{ copies}}, \underbrace{\overline{B_{r_{\rm in}(\Omega)}}, \ldots, \overline{B_{r_{\rm in}(\Omega)}}}_{j-1 \textrm{ copies}}, \overline{B_1})\,.
	\end{aligned}
	\end{equation}
	Since there exists $x_0 \in \Omega$ such that $\overline{B_{r_{\rm in}(\Omega)}(x_0)}\subset \overline{\Omega}$ it follows from \ref{itm: MV monotone}-\ref{itm: MV perm. inv.} and \ref{itm: MV per.} that
	\begin{equation}\label{eq: MV per. bound}
	 	W(\underbrace{\overline{\Omega}, \ldots, \overline{\Omega}}_{d-j \textrm{ copies}}, \underbrace{\overline{B_{r_{\rm in}(\Omega)}}, \ldots, \overline{B_{r_{\rm in}(\Omega)}}}_{j-1 \textrm{ copies}}, \overline{B_1}) \leq W(\underbrace{\overline{\Omega}, \ldots, \overline{\Omega}}_{d-1 \textrm{ copies}}, \overline{B_1}) = d^{-1}\Haus^{d-1}(\partial\Omega)\,,
	 \end{equation}
	  for any $j =1, \ldots, d$ (with equality if $j=1$).

	  From~\eqref{eq: Minkowski sum 2} and~\eqref{eq: MV per. bound},
	  \begin{align*}
	  	|\Omega+B_{r}|-|\Omega| 
	 		&\leq 
	 			d^{-1}r_{\rm in}(\Omega)\Haus^{d-1}(\partial\Omega)\sum_{j=1}^d \binom{d}{j} \Bigl(\frac{r}{r_{\rm in}(\Omega)}\Bigr)^{j} \\
	 			&=r_{\rm in}(\Omega)\Haus^{d-1}(\partial\Omega)\frac{\bigl(1+\frac{r}{r_{\rm in}(\Omega)}\bigr)^d-1}{d}\\
	 			&\lesssim_d \Haus^{d-1}(\partial\Omega)r \Bigl[1+\Bigl(\frac{r}{r_{\rm in}(\Omega)}\Bigr)^{d-1}\Bigr]\,,
	  \end{align*}
	  this proves the second bound.

	  Similarly from~\eqref{eq: Minkowski sum 2},~\eqref{eq: MV per. bound}, and the assumption $r\leq c_1 r_{\rm in}(\Omega)$,
	  \begin{align*}
	  	|\Omega+B_{r}|-|\Omega| 
	 		&\leq 
	 			r\Haus^{d-1}(\partial\Omega) + d^{-1}r_{\rm in}(\Omega)\Haus^{d-1}(\partial\Omega)\sum_{j=2}^d \binom{d}{j} \Bigl(\frac{r}{r_{\rm in}(\Omega)}\Bigr)^{j} \\
	 			&=r\Haus^{d-1}(\partial\Omega) + d^{-1}r_{\rm in}(\Omega)\Haus^{d-1}(\partial\Omega)\Bigl(\frac{r}{r_{\rm in}(\Omega)}\Bigr)^{2}\sum_{j=2}^d \binom{d}{j} \Bigl(\frac{r}{r_{\rm in}(\Omega)}\Bigr)^{j-2} \\
	 			&\leq r\Haus^{d-1}(\partial\Omega) + d^{-1}r_{\rm in}(\Omega)\Haus^{d-1}(\partial\Omega)\Bigl(\frac{r}{r_{\rm in}(\Omega)}\Bigr)^{2}\sum_{j=2}^d \binom{d}{j} c_1^{j-2} \\
	 			&=
	 			r\Haus^{d-1}(\partial\Omega)\Bigl[1 + \frac{r}{r_{\rm in}(\Omega)}\frac{(1+c_1)^d-1-dc_1}{dc_1^2}\Bigr]\,.
	  \end{align*}
	  This proves the third bound for any
	  \begin{equation*}
	  	c_2 \geq \frac{(1+c_1)^d-1-dc_1}{dc_1^2}\,.
	  \end{equation*}

	  For the final bound we deduce from~\eqref{eq: Minkowski sum 2},~\eqref{eq: MV per. bound}, the inequality $|\Omega|\leq r_{\rm in}(\Omega)\Haus^{d-1}(\partial\Omega)$ (see Lemma~\ref{lem: inradius, volume, perimeter bound}), and the assumption $r \geq c_1 r_{\rm in}(\Omega)$
	  \begin{align*}
	  	|\Omega+B_{r}|
	 		&\leq 
	 			r_{\rm in}(\Omega)\Haus^{d-1}(\partial\Omega)+d^{-1}r_{\rm in}(\Omega)\Haus^{d-1}(\partial\Omega)\sum_{j=1}^d \binom{d}{j} \Bigl(\frac{r}{r_{\rm in}(\Omega)}\Bigr)^{j} \\
	 			&=r_{\rm in}(\Omega)\Haus^{d-1}(\partial\Omega)\frac{\bigl(1+\frac{r}{r_{\rm in}(\Omega)}\bigr)^d+d-1}{d}\\
	 			&\leq c_2 \Haus^{d-1}(\partial\Omega)r\Bigl(\frac{r}{r_{\rm in}(\Omega)}\Bigr)^{d-1}\,,
	  \end{align*}
	  for any
	  \begin{equation*}
	  	c_2 \geq \sup_{s \in [c_1, \infty)}\frac{(1+s)^d+d-1}{d s^d}\,.
	  \end{equation*}
	  This completes the proof of Proposition~\ref{prop: Minkowski sum bounds}.
\end{proof}


\subsection*{Acknowledgment}
The authors are grateful for the anonymous referee’s careful reading and helpful suggestions. We would like to thank Jean Lagac\'e for stimulating discussions on asymptotics for polygons and
Nikolay Filonov for making us aware of a number of inaccuracies in an earlier version of this paper.




\phantomsection
\addcontentsline{toc}{part}{References}
\bibliographystyle{amsalpha}
\addtocontents{toc}{\SkipTocEntry}

\def\myarXiv#1#2{\href{http://arxiv.org/abs/#1}{\texttt{arXiv:#1\,[#2]}}}

\end{document}